\documentclass[3p,a4paper,times]{elsarticle}

\usepackage{amsmath}
\usepackage{amssymb}
\usepackage{amsfonts}
\usepackage{amsthm}
\usepackage{verbatim}
\usepackage{mathrsfs}
\usepackage{graphicx}
\usepackage{caption}
\usepackage{subcaption}
\usepackage{tabularx}
\usepackage{longtable}
\usepackage{tikz}
\usepackage{pgfplots}
\usepackage{xcolor}
\usepackage{pgfplots}
\usepackage{pgfplotstable}
\usepackage{enumitem}
\usepackage{layout}
\allowdisplaybreaks
\usepackage[toc]{appendix}

\usepackage[colorlinks=true, bookmarksopen,
pdfauthor={CST},
pdfcreator={pdftex},
pdfsubject={algorithms},
linkcolor={blue},
anchorcolor={black},
citecolor={firebrick},
filecolor={magenta},
menucolor={black},
pagecolor={red},
plainpages=false,pdfpagelabels,
urlcolor={db} ]{hyperref}

\usepackage[capitalise]{cleveref}
\usepackage[noadjust]{cite}
\usepackage{color}

\usepackage{bm}

\newtheorem{proposition}{Proposition}[section]
\newtheorem{theorem}{Theorem}[section]
\newtheorem{remark}{Remark}[section]
\newtheorem{example}{Example}[section]
\numberwithin{equation}{section}
\numberwithin{lemma}{section}

\numberwithin{equation}{section}

\newcommand{\N}{\mathbb{N}}    
\newcommand{\R}{\mathbb{R}}    

\newcommand{\bo}{\mathcal{O}}
\newcommand{\Qu}{\textsf{u}}
\newcommand{\Qa}{\textsf{a}}
\newcommand{\Qb}{\textsf{b}}
\newcommand{\Qc}{\textsf{c}}
\newcommand{\Qd}{\textsf{d}}
\newcommand{\Qf}{\textsf{f}}
\newcommand{\Qk}{\textsf{k}}
\newcommand{\Qw}{\textsf{w}}

\newcommand{\be}{ \begin{equation} }
	\newcommand{\ee}{ \end{equation} }

\newcommand{\ind}{\Lambda}

\newcommand{\nab}{\nabla}

\begin{document}
	
		\begin{frontmatter}
\title{Fourth-order compact finite difference methods for 2D and 3D nonlinear convection-diffusion-reaction equations}
		
\author[rvt1,rvt2]{Qiwei Feng\fnref{fn1}}
\ead{qiwei.feng@kfupm.edu.sa,qif21@pitt.edu,qfeng@ualberta.ca}
\address[rvt1]{Department of Mathematics, University of Pittsburgh, Pittsburgh, PA 15260, USA.}
\address[rvt2]{Mathematics Department, King Fahd University of Petroleum and Minerals, Dhahran 31261, Saudi Arabia.}

\fntext[fn1]{Qiwei Feng is partially supported by  the Mathematics Research Center, Department of Mathematics, University of Pittsburgh, Pittsburgh, PA, USA, and the Mathematics Department, King Fahd University of Petroleum and Minerals (KFUPM), Dhahran, Saudi Arabia.}

\makeatletter \@addtoreset{equation}{section} \makeatother

	\begin{abstract}	
		In this paper, we first consider linear 2D and 3D convection-diffusion-reaction (advection-diffusion-reaction) equations $-\nab\cdot (\kappa \nab u) +    {\bm v} \cdot \nab u  + \lambda u = \phi$ (time-independent/steady/stationary)  and  $u_t-\nab\cdot (\kappa \nab u) +    {\bm v} \cdot \nab u   + \lambda u = \phi$ (time-dependent/unsteady/nonstationary),  where  all $\kappa>0,{\bm v},\lambda, \phi$ are smooth variable functions. We derive fourth-order compact 9-point (2D) and 19-point (3D) finite difference methods (FDMs) to solve linear time-independent equations. As derivations of high-order compact FDMs are very complicated and involve cumbersome notation (especially in 3D), it is usually difficult for readers not specializing in high-order FDMs to follow derivations and replicate numerical results. In this paper, we observe interesting and novel expressions of stencils of high-order FDMs  which  introduce new restrictions of stencils to help construct compact fourth-order FDMs (2D and 3D) with easy, explicit, and short  expressions. 
		These simple stencils make the analysis of the truncation error easy for readers  to understand and facilitate implementing proposed FDMs directly. For linear unsteady equations, we apply  Crank-Nicolson (CN), BDF3, BDF4 methods with above compact FDMs to compute numerical solutions.  Finally, we discuss nonlinear convection-diffusion-reaction equations in 2D and 3D, i.e., $-\nab\cdot (\kappa(u) \nab u) +    {\bm v}(u) \cdot \nab u   + \lambda(u) u = \phi$ and $u_t-\nab\cdot (\kappa(u) \nab u) +    {\bm v}(u) \cdot \nab u   + \lambda(u) u = \phi$, where each function of $\kappa(u)>0, {\bm v}(u),\lambda(u)$ depends on the solution $u$. We linearize nonlinear equations by the fixed point method (iterative method) and use above simple fourth-order compact FDMs to solve linearized equations (unsteady equations also utilize CN, BDF3, and BDF4 methods).  Each of proposed FDMs in 2D and 3D for linear, nonlinear, steady, and unsteady equations satisfies the discrete maximum principle and  forms an M-matrix for the sufficiently small $h$, if the function $\lambda$ is nonnegative. The numerical results validate the accuracy and convergence rate of above FDMs for  linear and nonlinear,  stationary and  nonstationary, convection-diffusion-reaction equations in 2D and 3D.
\end{abstract}	
	

\end{frontmatter}

	\section{Introduction}
	Nonlinear convection-diffusion-reaction equations $-\nab\cdot (\kappa(u) \nab u) +  {\bm v(u)} \cdot \nab u +   \lambda(u) u = \phi$ (steady) and $u_t-\nab\cdot (\kappa(u) \nab u) +  {\bm v(u)} \cdot \nab u +   \lambda(u) u = \phi$ (unsteady) with nonlinear functions of diffusion $\kappa(u)$, convection ${\bm v(u)}$, and reaction $\lambda(u)$ appear in many practical problems, such as fluid flow, heat transport, and various biological processes. In \citep{FengTrenchea2026}, we derived fourth-order  compact 9-point   finite difference methods (FDMs) for 2D nonlinear convection-diffusion equations   $-\nab\cdot (\kappa \nab u) +  {\bm v(u)} \cdot \nab u  = \phi$ (time-independent) and $u_t-\nab\cdot (\kappa \nab u) +  {\bm v(u)} \cdot \nab u  = \phi$ (time-dependent) with variable functions $\kappa(x,y)$, $\kappa(x,y,t)$ and the nonlinear convection function ${\bm v(u)}$. As we did not consider the reaction term $\lambda(u)$ in \citep{FengTrenchea2026},  we could not use the FDM of the stationary equation to solve the nonstationary equation. I.e., we did not provide the FDM with the explicit expression to solve the 2D time-dependent equation. In this paper, we discuss the extended problem of \citep{FengTrenchea2026}, and propose fourth-order compact 9-point (2D) and 19-point (3D) FDMs with simple stencils. Precisely, we consider linear and nonlinear, steady and unsteady,  convection-diffusion-reaction equations in 2D and 3D. As we build fourth-order FDMs in this paper, we assume that the exact solution $u$ is smooth in the computational spatial and temporal  domains.
	
	A brief review of numerical methods of discontinuous Galerkin (DG) methods, finite element methods (FEMs),  finite volume methods (FVMs), finite difference methods (FDMs), and other numerical methods for nonlinear convection-diffusion and convection-diffusion-reaction equations is provided below (see references therein for further discussions). For the nonlinear  steady   convection-diffusion-reaction equation,  the stabilized Galerkin method for
	$-\epsilon \Delta u	+ {\bm v} \cdot \nab u +\lambda(u) = f$ was derived by Burman et al. in \citep{Burman2002} (the linear case was discussed by Burman et al. in \citep{BurmanErn2005}); up to sixth-order FDMs were proposed  to solve $	-\kappa(u)\Delta u  +   {\bm v(u)} \cdot \nab u +\lambda(u) u = f$ by  Clain et al.  in \citep{Clain2024}; and the fourth-order DG method (fifth-order accuracy can be achieved by postprocessing) for  $-\nab\cdot (\kappa \nab u) +   {\bm v(u)} \cdot \nab u  = f$ was constructed by Nguyen et al.  in \citep{Nguyen2009}.
	For the nonlinear time-dependent  equation, the error analysis  for $u_t-\epsilon \Delta u	+ {\bm v}(u) \cdot \nab u  = f$ was discussed in
	\citep{Dolej2007,Dolej2005,DolejVlas2008,Feistauer2011} (DG method) and presented  by Feistauer et al. in \citep{Feistauer1997} (combined FV-FEM); Hou	et al. in \citep{Hou2024} built a local DG method with the implicit-explicit
	time marching to solve	$u_t-\epsilon \Delta u	+  \nab \cdot ({\bm v}  \sigma(u) )+\lambda(u)  = f$; Huang et al. in \citep {Huang2026} constructed the fourth-order FDM for $u_t- \Delta u	+ {\bm v} \cdot \nab u +r u = f(u(\theta(t)),u(t))$ with the nonlinear vanishing delay term $f(u(\theta(t)),u(t))$; the third-order DG method (fourth-order accuracy can be attained by postprocessing) of  $u_t-\nab\cdot (\kappa \nab u) + {\bm v}(u) \cdot \nab u  = f$ was designed by Nguyen et al. in \citep{Nguyen2009}; the convergence analysis of the DG method for  $u_t-\nab\cdot (K(u) \nab u) 	+ {\bm v}(u) \cdot \nab u  = 0$ with the positive definite matrix $K(u)$ was deduced by Xu et al. in \citep{XuShu2007}  and Yan  in \citep{Yan2013}. Eymard et al. in \citep{Eymard2010} and Kurganov et al. in \citep{Kurganov2000} numerically solved more general cases: $(\alpha(u))_t-\nab\cdot (K \nab u) 	+  \nab \cdot ({\bm v}  u)+\lambda(u)  = f$ and $u_t-(P(u,u_x,u_y))_x-(Q(u,u_x,u_y))_y+ {\bm v}(u) \cdot \nab u  = 0$, respectively. The stabilized FEM for the system  of linear stationary  convection-diffusion-reaction equations was provided by Codina et al. in \citep{Codina2000}. 
	For the system of nonlinear  time-independent   equations,  the discontinuity-capturing FEM was developed by  Tezduyar et al. in \citep{Tezduyar1986}.  
	For the system of  nonlinear  unsteady  
	convection-diffusion-reaction equations, the error estimate of the stabilized higher-order FEM was derived by Bause et al.  in \citep{Bause2012}; the extended Runge-Kutta DG method was introduced by Cockburn  in \citep{Cockburn1998}; Knoll et al. in \citep{Knoll1995} designed the Newton-Krylov method;  the  nonlinear stability analysis of the high-order  DG method was established by  Michoski et al. in \citep{Michoski2017}; and Sheshachala et al. in \citep{Sheshachala2018} developed an extended stabilized FEM.
	
	In this paper, we discuss  linear and nonlinear, time-independent and time-dependent, convection-diffusion-reaction equations  equations in 2D and 3D. Our paper is organized as follows:
	In \cref{subse:liner:elli:2D}, we discuss the 2D linear time-independent equation, and propose the fourth-order compact 9-point FDM with the simple stencil in \cref{thm:FDM:2D} (\cref{thm:FDM:2D} is our first main result). In \cref{subse:liner:para:2D}, we consider the 2D linear time-dependent equation and use CN, BDF3, BDF4 methods with \cref{thm:FDM:2D} to compute the solution. By analyzing the truncation error, the order of accuracy is at most three. So we also derive the fourth-order compact 9-point FDM  in \cref{thm:FDM:parabo:2D} to solve the linear  nonstationary equation. In 	\cref{subse:non:liner:elli:2D,subse:non:liner:para:2D}, we use the fixed method (iteration method) to transfer 2D nonlinear steady and unsteady equations to  linear equations which are solved by FDMs in \cref{thm:FDM:2D} and \cref{thm:FDM:parabo:2D}. In \cref{subse:liner:elli:3D,subse:liner:para:3D,subse:non:liner:elli:3D,subse:non:liner:para:3D}, we propose fourth-order compact 19-point FDMs in \cref{thm:FDM:3D} (\cref{thm:FDM:3D} is our second main result)	and \cref{thm:FDM:parabo:3D} to solve linear and nonlinear, stationary and nonstationary equations in 3D. In \cref{sec:Numeri}, we provide eight examples for above eight PDEs (linear and nonlinear, steady and unsteady, 2D and 3D) to verify performances of FDMs in \cref{thm:FDM:2D,thm:FDM:parabo:2D,thm:FDM:3D,thm:FDM:parabo:3D}. The numerical results revel that our two main results: FDMs in \cref{thm:FDM:2D} (the main result in 2D) and \cref{thm:FDM:3D} (the main result in 3D) with simple and explicit expressions can achieve third-order to fourth-order accuracy of errors in the $l_{\infty}$ norm for all eight PDEs. Furthermore,  FDMs in \cref{thm:FDM:parabo:2D} (2D) and \cref{thm:FDM:parabo:3D} (3D) with the BDF4 method achieve the fourth-order accuracy for time-dependent equations. Although \cref{thm:FDM:2D,thm:FDM:3D} are one order lower than \cref{thm:FDM:parabo:2D,thm:FDM:parabo:3D} for unsteady equations, respectively, the numerical results indicate that errors from \cref{thm:FDM:2D,thm:FDM:3D} are only slightly larger than those from \cref{thm:FDM:parabo:2D,thm:FDM:parabo:3D}, respectively. As stencils of \cref{thm:FDM:2D,thm:FDM:3D} are presented in easy expressions, we can conclude that we derive compact 9-point (2D) and 19-point (3D) FDMs with simple stencils and the high accuracy for   linear and nonlinear, stationary and nonstationary, convection-diffusion-reaction equations.  In \cref{sec:contribu}, we summarize the main contribution.

	\section{Fourth-order  compact FDMs in 2D and 3D}
	In this section,  we derive fourth-order  compact finite difference methods (FDMs) for linear and nonlinear,  time-independent and time-dependent,  convection-diffusion-reaction equations in 2D and 3D in following Sections \ref{subse:liner:elli:2D}--\ref{subse:non:liner:para:3D}.		
	
	\subsection{Fourth-order  compact 9-point FDM for the linear time-independent equation in 2D}\label{subse:liner:elli:2D}
	In this subsection, we discuss the following 2D linear steady  equation with the Dirichlet boundary condition:
	\be\label{Linear:Elliptic:2D}
	\begin{cases}
		-\nab\cdot (\kappa \nab u) +   \alpha u_x +   \beta u_y +   \lambda u = \phi, \qquad \qquad  &  (x,y)\in \Omega,\\
		u =g,  & (x,y)\in \partial \Omega,
	\end{cases}
	\ee
	where $\kappa=\kappa(x,y)>0$, $\alpha= \alpha(x,y)$, $\beta= \beta(x,y)$, $\lambda= \lambda(x,y)$, and $\phi=\phi(x,y)$ are smooth variable functions in the square spatial domain
	$\Omega:=(l_1,l_2)^2$. 
	According to the direct calculation of \eqref{Linear:Elliptic:2D}, we have	
	\[
	-\kappa \Delta u  + (\alpha-\kappa_x) u_x +   (\beta-\kappa_y) u_y +   \lambda u = \phi,
	\]
	and
	\[
	\Delta u  + \frac{\kappa_x-\alpha}{\kappa} u_x +   \frac{\kappa_y-\beta}{\kappa} u_y -   \frac{\lambda}{\kappa} u = -\frac{\phi }{\kappa}.
	\]
	Let
	\be\label{notation:abdf}
	a:=\frac{\kappa_x-\alpha}{\kappa}, \qquad b:=\frac{\kappa_y-\beta}{\kappa}, \qquad d:=-   \frac{\lambda}{\kappa}, \qquad f:=-\frac{\phi }{\kappa}.
	\ee
	Then
	\be\label{2D:linear:simple:eq}
	\Delta u  + a u_x +  bu_y +du =f.
	\ee
	We do not use $c:=-\frac{\lambda}{\kappa}$ in \eqref{notation:abdf} to reuse 2D notations for 3D convection-diffusion-reaction equations in \cref{subse:liner:elli:3D,subse:liner:para:3D,subse:non:liner:elli:3D,subse:non:liner:para:3D}, as we define $c:=\frac{\kappa_z-\gamma}{\kappa}$ in \eqref{abcd:3D} for 3D equations.
	In this paper, we use the uniform Cartesian grid to discretize the spatial  domain $\Omega=(l_1,l_2)^2$ as follows:
	\be \label{xiyj:2D:space}
	x_i:=l_1+i h, \qquad y_j:=l_1+j h, \qquad i,j=0,\ldots,N_1,  \qquad \text{and} \qquad h:=(l_2-l_1)/N_1,
	\ee
	where $N_1\in \N$. We define that $u_h$  is  computed by the proposed fourth-order FDM with the uniform Cartesian mesh size $h$, $(u_h)_{i,j}$ and $u_{i,j}$ are the values of the numerical solution $u_h$ and the exact solution $u$ at the grid point $(x_i,y_j)$ in \eqref{xiyj:2D:space}, respectively.
	
	{\bf{ The algorithm to derive the fourth-order FDM:}}	In \citep{FengTrenchea2026}, we discuss the convection-diffusion equation as follows:
	\be\label{conve:diff:eq}
	\Delta u  + a u_x +  bu_y  =f.
	\ee
	In this paper, we consider the convection-diffusion-reaction equation $\Delta u  + a u_x +  bu_y +du =f$ in \eqref{2D:linear:simple:eq}.
	Similarly applying \citep[eqs. (3)--(36)]{FengTrenchea2026}, the exact solution $u$ in \eqref{Linear:Elliptic:2D} and \eqref{2D:linear:simple:eq} satisfies
	\be \label{u:GH}
	u(x+x_i,y+y_j)=\sum_{(m,n)\in \ind^1_{5}}
	u^{(m,n)}G_{5,m,n}(x,y)+\sum_{(m,n)\in \ind_{3}}	f^{(m,n)} H_{5,m,n}(x,y)+\bo(h^{6}),
	\ee
	where  $x,y\in [-h,h]$,
	\[
	u^{(m,n)}:=\frac{\partial^{m+n}u(x,y)}{\partial x^m \partial y^n}\Big|_{(x,y)=(x_i,y_j)},\qquad f^{(m,n)}:=\frac{\partial^{m+n}f(x,y)}{\partial x^m \partial y^n}\Big|_{(x,y)=(x_i,y_j)},
	\]
	\be\label{ind:2D}
	\ind_{5}^{ 1}:=\{  (m,n)\in \N^2 \; : \; m=0,1, m+n\le 5\}, \qquad \ind_{3}:=\{ (m,n)\in \N^2 \; : \;  m+n\le 3 \}, 
	\ee
	$G_{5,m,n}(x,y)$ and $H_{5,m,n}(x,y)$ are polynomials of variables $x,y$ whose coefficients can be uniquely determined by  $\{\tfrac{\partial ^{m+n} \varrho(x,y)}{\partial x^m \partial y^n}|_{(x,y)=(x_i,y_j)} : (m,n)\in 	\ind_{3}\}$ with $\varrho=a,b,d$ in \eqref{notation:abdf}--\eqref{2D:linear:simple:eq}, and are independent of  $\{f^{(m,n)} : (m,n)\in 	\ind_{3}\}$. Precisely, using \citep[eqs. (3)--(36)]{FengTrenchea2026} with replacing \eqref{conve:diff:eq} by \eqref{2D:linear:simple:eq}, then $G_{5,m,n}(x,y)$ and $H_{5,m,n}(x,y)$ can be obtained uniquely by the symbolic computation. Although we do not provide explicit formulas of  $G_{5,m,n}(x,y)$ and $H_{5,m,n}(x,y)$ in \eqref{u:GH}, the first main result \cref{thm:FDM:2D} in 2D in this paper (the second main result is  \cref{thm:FDM:3D} in 3D) presents the explicit formula  of the fourth-order compact 9-point FDM (see \eqref{C:Left:2D}--\eqref{F:Right:2D}) without requiring any information from  $G_{5,m,n}(x,y)$ and $H_{5,m,n}(x,y)$. Readers can just consider \eqref{u:GH} as a tool to understand the  derivation of the high-order FDM with available but complicated $G_{5,m,n}(x,y)$ and $H_{5,m,n}(x,y)$, but the first main result \cref{thm:FDM:2D} is simple and successfully eliminates all tedious terms  in $G_{5,m,n}(x,y)$ and $H_{5,m,n}(x,y)$ by employing new observations \eqref{Crl}--\eqref{property:3:2D}. 
	
	Now, we utilize \eqref{u:GH} to derive the fourth-order compact 9-point FDM as follows (note that the following \eqref{Lh:u:1}--\eqref{EQ:2:explicit} are similar to \citep[eqs. (39)--(46)]{FengTrenchea2026}, we repeat them here to introduce novel observations \eqref{property:1:2D}--\eqref{property:3:2D}):	Let
	\be\label{Lh:u:1}
	\begin{split}
		\mathcal{L}_h u_{i,j} :=\frac{1}{h^2}\sum_{r,\ell=-1}^1 C_{r,\ell}\Big|_{(x,y)=(x_i,y_j)} u_{i+r,j+\ell}=\frac{1}{h^2}\sum_{r,\ell=-1}^1 C_{r,\ell}\Big|_{(x,y)=(x_i,y_j)} u(rh+x_i,\ell h+y_j).
	\end{split}
	\ee
	Then \eqref{u:GH} leads to
	\begin{align}\label{Lh:u}
		h^{-2}	\mathcal{L}_h u_{i,j}   = &	 h^{-2}  \sum_{r,\ell=-1}^{1} C_{r,\ell} \Big|_{(x,y)=(x_i,y_j)}	\sum_{(m,n)\in \ind^1_{5}}
		u^{(m,n)}G_{5,m,n}(rh,\ell h) \notag \\
		&+h^{-2} \sum_{r,\ell=-1}^{1} C_{r,\ell} \Big|_{(x,y)=(x_i,y_j)} \sum_{(m,n)\in \ind_{3}}	f^{(m,n)} H_{5,m,n}(rh,\ell h)+\bo(h^{4}), \notag  \\
		=& h^{-2} \sum_{(m,n)\in \ind_{5}^{1}} u^{(m,n)} I_{5,m,n} +h^{-2} \sum_{(m,n) \in \ind_{	 3}} f^{(m,n)} J_{3,m,n}   +\bo(h^{4}),
	\end{align}
	where
	\be\label{Imn}
	\begin{split}
		I_{5,m,n}:=\sum_{r,\ell=-1}^{1} C_{r,\ell} \Big|_{(x,y)=(x_i,y_j)} G_{5,m,n}  (rh, \ell h),\qquad J_{3,m,n}:= \sum_{r,\ell=-1}^{1} C_{r,\ell}\Big|_{(x,y)=(x_i,y_j)} H_{5,m,n}  (rh, \ell h).
	\end{split}
	\ee
	Let
	\be\label{FDMs:2D:u}
	\mathcal{L}_h (u_h)_{i,j} :=\frac{1}{h^2}\sum_{r,\ell=-1}^1 C_{r,\ell}\Big|_{(x,y)=(x_i,y_j)} (u_h)_{i+r,j+\ell}=F_{i,j}\Big|_{(x,y)=(x_i,y_j)},
	\ee
	where
	\be \label{Fij}
	F_{i,j}:=\text{the terms of } \Big(h^{-2} \sum_{(m,n) \in \ind_{	 3}} f^{(m,n)} J_{3,m,n}\Big) \text{ with degree}   \le 3 \text{ in } h.
	\ee
	Then \eqref{Lh:u}, \eqref{FDMs:2D:u}, and  \eqref{Fij} result in
	\be\label{L:h:u:uh} 
	h^{-2}	\mathcal{L}_h (u-u_h)_{i,j}=	 h^{-2} \sum_{(m,n)\in \ind_{5}^{1}} u^{(m,n)} I_{5,m,n} +\bo(h^{4}).
	\ee
	
	\textbf{A novel and  interesting observation of $C_{r,\ell}$ to achieve the fourth-order accuracy of \eqref{L:h:u:uh}:}  In \citep{FHM2024,FengTrenchea2026},  we set the left-hand side $C_{r,\ell}$ of the FDM \eqref{FDMs:2D:u} as the polynomial of $h$, i.e.,
	\be\label{Crl}
	C_{r,\ell}:=\sum_{p=0}^5 c_{r,\ell,p} h^p, \quad \text{with }  c_{r,\ell,p} \text{ depending on } a,b,d \text{ in } \eqref{2D:linear:simple:eq}, \quad \text{and} \quad c_{r,\ell,p}|_{(x,y)=(x_i,y_j)} \in \R.
	\ee
	From \eqref{Imn}, \eqref{L:h:u:uh}, and \eqref{Crl}, to derive the fourth-order FDM, we need to solve
	\be \label{EQ:1:explicit}
	I_{5,m,n}=\sum_{r,\ell=-1}^{1} G_{5,m,n}  (rh, \ell h) \sum_{p=0}^5 c_{r,\ell,p} \big|_{(x,y)=(x_i,y_j)} h^p =\bo(h^{6}) \quad \mbox{for all}  \quad (m,n)\in \ind_{5}^1, 
	\ee
	with
	\be \label{EQ:2:explicit}
	C_{0,0}|_{(x,y,h)=(x_i,y_j,0)}=c_{0,0,0}|_{(x,y)=(x_i,y_j)}\ne 0,  \qquad  F_{i,j}|_{(x,y,h)=(x_i,y_j,0)}=f(x_i,y_j).
	\ee
	While the solution $c_{r,\ell,p}$ in \eqref{EQ:1:explicit} is not unique and depends on  $a,b,d$, and it is very hard to obtain the $c_{r,\ell,p}$ with the short expression, especially for the equation \eqref{2D:linear:simple:eq} with variable functions $a,b,d$. Furthermore, this task is much harder for the convection-diffusion-reaction equation \eqref{simpli:PDE:3D} in 3D in \cref{subse:liner:elli:3D}.
	
	The symbolic calculation of \eqref{EQ:1:explicit} reveals following novel and  interesting observations \eqref{property:1:2D}--\eqref{property:3:2D}: each $c_{k,\ell,p}$ in  each 	$C_{k,\ell}$ always satisfies (similar observations for $	-\nab\cdot (\kappa \nab u)=\phi$ were proposed in \citep{FHMS2026})
	\be\label{property:1:2D}
	c_{r,\ell,p}=\sum_{\varsigma} \sigma_{r,\ell,p,\varsigma} \prod_{\mu_{\varsigma,1},\nu_{\varsigma,1},\rho_{\varsigma,1}}\left(\frac{\partial^{\mu_{\varsigma,1}+\nu_{\varsigma,1}} a }{\partial x^{\mu_{\varsigma,1}} \partial y^{\nu_{\varsigma,1}} }\right)^{\rho_{\varsigma,1}} \prod_{\mu_{\varsigma,2},\nu_{\varsigma,2},\rho_{\varsigma,2}}\left(\frac{\partial^{\mu_{\varsigma,2}+\nu_{\varsigma,2}} b }{\partial x^{\mu_{\varsigma,2}} \partial y^{\nu_{\varsigma,2}} }\right)^{\rho_{\varsigma,2}}  \prod_{\mu_{\varsigma,3},\nu_{\varsigma,3},\rho_{\varsigma,3}}\left(\frac{\partial^{\mu_{\varsigma,3}+\nu_{\varsigma,3}} d }{\partial x^{\mu_{\varsigma,3}} \partial y^{\nu_{\varsigma,3}} }\right)^{\rho_{\varsigma,3}},
	\ee
	where 
	\be\label{property:2:2D}
	\text{each }  \sigma_{r,\ell,p,\varsigma} \text{ is independent of } a,b,d,f,
	\ee
	and
	\be\label{property:3:2D}
	(\mu_{\varsigma,1}+\nu_{\varsigma,1}+1)\rho_{\varsigma,1}+	(\mu_{\varsigma,2}+\nu_{\varsigma,2}+1)\rho_{\varsigma,2}+	(\mu_{\varsigma,3}+\nu_{\varsigma,3}+2)\rho_{\varsigma,3}=p, \quad \text{for each} \quad  \varsigma.
	\ee
	As \eqref{property:1:2D}--\eqref{property:3:2D}  are key contributions in this paper,	we provide following \eqref{c004}--\eqref{c004:detail} as an example to help readers understand: In \eqref{C:Left:2D} in \cref{thm:FDM:2D}, the coefficient of $h^4$ in $C_{0,0}$ is $\tfrac{1}{12}[ad_x+\Delta d]h^4$. So
	\be\label{c004}
	c_{0,0,4}=\tfrac{1}{12}ad_x+\tfrac{1}{12} d_{xx}+\tfrac{1}{12} d_{yy}.
	\ee
	Comparing with \eqref{property:1:2D}, we have
	\be\label{c004:detail}
	\begin{split}
		& (r,\ell,p)=(0,0,4),\qquad \varsigma=1,2,3;\\
		& (\mu_{1,1},\nu_{1,1},\rho_{1,1})=(0,0,1),\quad \rho_{1,2}=0,\quad (\mu_{1,3},\nu_{1,3},\rho_{1,3})=(1,0,1), \quad \sigma_{0,0,4,1}=\tfrac{1}{12}\quad \text{for }\varsigma=1;\\
		& \rho_{2,1}=0,\quad \rho_{2,2}=0,\quad (\mu_{2,3},\nu_{2,3},\rho_{2,3})=(2,0,1), \quad \sigma_{0,0,4,2}=\tfrac{1}{12}\quad \text{for }\varsigma=2;\\
		& \rho_{3,1}=0,\quad \rho_{3,2}=0,\quad (\mu_{3,3},\nu_{3,3},\rho_{3,3})=(0,2,1), \quad \sigma_{0,0,4,3}=\tfrac{1}{12}\quad \text{for }\varsigma=3.
	\end{split}
	\ee
	
	\textbf{The contribution of this observation:}	
	Different from computing $c_{r,\ell,p}$ directly in \citep{FHM2024,FengTrenchea2026}, we choose to calculate $	\sigma_{r,\ell,p,\varsigma}$ in this paper. Using restrictions \eqref{property:2:2D} and \eqref{property:3:2D} for \eqref{property:1:2D}, the computed $C_{r,\ell}$ is very simple and short (see the following expression \eqref{C:Left:2D} in \cref{thm:FDM:2D}). Furthermore, as we consider variable functions $a,b,d,f$ in \eqref{2D:linear:simple:eq},  derivations from \eqref{u:GH}--\eqref{EQ:1:explicit} to construct $C_{r,\ell}$ are not easy for readers to understand. With the aid of the short and easy expression of $C_{r,\ell}$ in this paper, 
	the verification of the truncation error is very easy (see the proof of the following \cref{thm:FDM:2D}). Furthermore, although we consider the more complicated equation than \citep{FengTrenchea2026}, the expression of $F_{i,j}$ in the following \eqref{F:Right:2D} is simpler and shorter than $F_{i,j}$ in \citep[eq. (51)]{FengTrenchea2026}.
	
	Now we propose the fourth-order compact 9-point FDM for  convection-diffusion-reaction equations \eqref{Linear:Elliptic:2D} and \eqref{2D:linear:simple:eq} in the following \cref{thm:FDM:2D} which is obtained by symbolically solving \eqref{EQ:1:explicit}--\eqref{EQ:2:explicit} under new restrictions of $c_{r,\ell,p}$ \eqref{property:1:2D}--\eqref{property:3:2D}. 
	\begin{theorem}\label{thm:FDM:2D}
		Let $\kappa>0,u,\alpha,\beta,\lambda, \phi$ be smooth in $\overline{\Omega}$ in \eqref{Linear:Elliptic:2D}, and define
		\begin{align}\label{C:Left:2D}
			&C_{-1,-1}:=\tfrac{1}{6}-\tfrac{a+b}{12}h, \notag \\ 
			&C_{-1,0}:=\tfrac{2}{3}-\tfrac{a}{3}h+\tfrac{1}{12}[a^2+ab+d+2a_x+a_y+b_x]h^2 \notag \\
			&\qquad \quad -\tfrac{1}{24}[a(a_x+d)+ba_y+2d_x+\Delta a]h^3, \notag \\
			&C_{-1,1}:=\tfrac{1}{6}-\tfrac{1}{12}[a-b]h-\tfrac{1}{12}[ab+a_y+b_x]h^2,  \notag \\ 
			&C_{0,-1}:=\tfrac{2}{3}-\tfrac{b}{3}h+\tfrac{1}{12}[b^2+ab+a_y+b_x+2b_y+d]h^2 \notag\\
			&\qquad \quad-\tfrac{1}{24}[ab_x+b(b_y+d)+2d_y+\Delta b]h^3+\tfrac{bd_y}{12}h^4, \notag\\
			&C_{0,0}:=\tfrac{-10}{3}-\tfrac{1}{6} [ a^2+ab+b^2-4d+2a_x+a_y+b_x+2b_y ]h^2 \notag \\
			&\qquad \quad+\tfrac{1}{12}[ad_x+\Delta d]h^4,\\
			&C_{0,1}:=\tfrac{2}{3}+\tfrac{b}{3}h+\tfrac{1}{12}[ab+b^2+d+a_y+b_x+2b_y]h^2 \notag \\
			&\qquad \quad+\tfrac{1}{24}[ ab_x+b(b_y+d)+2d_y+\Delta b ]h^3, \notag \\
			&C_{1,-1}:=\tfrac{1}{6}+\tfrac{1}{12}[a-b]h-\tfrac{1}{12}[ab+a_y+b_x]h^2-\tfrac{bd_y}{12}h^4, \notag \\
			&C_{1,0}:= \tfrac{2}{3}+\tfrac{a}{3}h+\tfrac{1}{12}[a^2+ab+d+2a_x+a_y+b_x]h^2 \notag \\
			&\qquad \quad+\tfrac{1}{24}[a(a_x+d)+ba_y+2d_x+\Delta a]h^3+\tfrac{bd_y}{12}h^4, \notag \\
			&C_{1,1}:=\tfrac{1}{6} +\tfrac{1}{12}[a+b]h, \notag
		\end{align}
		\be\label{F:Right:2D}
		F_{i,j}:=f+\tfrac{1}{12}[ af_x+bf_y+\Delta f  ]h^2,
		\ee
		and $a,b,d,f$ are defined in \eqref{notation:abdf}. Then
		\be\label{truncation:error:h:4}
		\mathcal{L}_h (u_h)_{i,j}-\mathcal{L}_h (u)_{i,j}=\bo(h^4), 
		\ee	
		where $\mathcal{L}_h u_{i,j}$  and $\mathcal{L}_h (u_h)_{i,j}$ are defined in \eqref{Lh:u:1} and \eqref{FDMs:2D:u}, respectively. I.e., $\mathcal{L}_h (u_h)_{i,j}$ in \eqref{FDMs:2D:u} with $C_{r,\ell}$ in \eqref{C:Left:2D} and $F_{i,j}$ in \eqref{F:Right:2D} achieves a consistency order four for  convection-diffusion-reaction equations \eqref{Linear:Elliptic:2D} and \eqref{2D:linear:simple:eq} at $(x_i,y_j)$ in \eqref{xiyj:2D:space}.
	\end{theorem}
	\begin{proof}
		At the grid point $(x_i,y_j)\in \Omega$, the Taylor expansion  with the Lagrange remainder yields
		\be\label{taylor:basis:2D}
		u(x_i+x,y_j+y)
		=
		\sum_{\substack{0\le m,n \le M+1 \\ m+n\le M+1}}
		\frac{\partial^{m+n}u(x_i,y_j) }{\partial x^m \partial y^n}\frac{x^my^n}{m!n!}+\sum_{\substack{0\le m,n \le M+2  \\ m+n=M+2}} \frac{\partial^{M+2}u(\xi,\eta) }{\partial x^m \partial y^{n}}\frac{x^my^{n}}{m!n!},\qquad M\in \N,
		\ee
		where the point $(\xi,\eta)$ is in the open line segment between $(x_i,y_j)$ and $(x_i+x,y_j+y)$. Plug $M=4$ and $(x,y)=(rh,\ell h)$ into \eqref{taylor:basis:2D}, 
		\be \label{taylor:u}
		u(x_i+rh,y_j+\ell h)
		=
		\sum_{\substack{0\le m,n\le 5  \\ m+n\le 5}}
		\frac{\partial^{m+n}u(x_i,y_j) }{\partial x^m \partial y^n}\frac{(rh)^m(\ell h)^n}{m!n!}+\zeta_{i,j,r,\ell}h^6,
		\ee
		where  $\zeta_{i,j,r,\ell} < \infty$ only depends on $\{ \frac{\partial^{6}u(\xi,\eta) }{\partial x^m \partial y^{n}} :  0\le m,n \le 6 \text{ and }  m+n=6 \}$. 
		Substitute \eqref{C:Left:2D} and \eqref{taylor:u} into \eqref{Lh:u:1}, 
		\be\label{L:h:u:2}
		\mathcal{L}_h (u)_{i,j} :=\sum_{p=-2}^3\theta_p \Big|_{(x,y)=(x_i,y_j)}h^p+\bo(h^4),
		\ee
		where 
		\begin{align}
			&\theta_{-2}=\theta_{-1}=\theta_1=\theta_3=0,\qquad 	\theta_0=\Delta u  + a u_x +  bu_y +du, \notag\\
			&\theta_2=\tfrac{1}{12}\big[a^2u_{xx}+b^2u_{yy}+ 2a_{x}u_{xx}+2b_{y}u_{yy}+ 2(a_{y}+b_{x}-ab)u_{xy}\notag\\
			&\qquad \ +a(2\Delta u_x+2bu_{xy}+a_{x}u_{x}+b_{x}u_{y}+du_{x}+d_{x}u)\notag\\
			&\qquad \ +b(2\Delta u_y+2au_{xy}+a_{y}u_{x}+b_{y}u_{y}+du_{y}+d_{y}u)\notag\\
			&\qquad \ +(\Delta a+2d_{x})u_{x}+(\Delta b+2d_{y})u_{y}\notag\\
			&\qquad \ +d \Delta u +u\Delta d  +\Delta^2 u\big].\notag
		\end{align}
		After the direct calculation, 
		\begin{align*}
			\theta_2=&\tfrac{1}{12}\big\{	a[\Delta u_x+au_{xx}+bu_{xy}+(a_{x}+d)u_{x}+b_{x}u_{y}+d_{x}u]\\
			&+	b[\Delta u_y+bu_{yy}+au_{xy}+a_{y}u_{x}+(b_{y}+d)u_{y}+d_{y}u]\\
			&+\Delta^2 u+a\Delta u_x+b\Delta u_y+(2a_{x}+d)u_{xx}+(2b_{y}+d)u_{yy}\\
			&+2(a_{y}+b_{x})u_{xy}+(\Delta a+2d_{x})u_{x}+(\Delta b+2d_{y})u_{y}+u\Delta d \big\}.
		\end{align*}
		To prove the  fourth-order consistency of the proposed FDM in \eqref{FDMs:2D:u} for the equation \eqref{2D:linear:simple:eq}, we define that 
		\be\label{L:u}
		\mathcal{L} u:= \Delta u + au_x+bu_y+du.
		\ee
		Now, by the algebraic computation, it is straightforward to verify that
		\[
		\theta_0=	\mathcal{L} u, \qquad \theta_2=\tfrac{1}{12}\big\{a[\mathcal{L} u]_x+b[\mathcal{L} u]_y+[\mathcal{L} u]_{xx}+[\mathcal{L} u]_{yy}\big\}.
		\]
		As $\mathcal{L} u$ in \eqref{L:u} satisfies that $\mathcal{L} u=f$, we can say that \eqref{L:h:u:2} yields
		\be\label{L:h:u:3}
		\mathcal{L}_h (u)_{i,j} :=\theta_0+\theta_2h^2+\bo(h^4)=f\big|_{(x,y)=(x_i,y_j)}+\tfrac{h^2}{12}[af_x+bf_y+ \Delta f]\big|_{(x,y)=(x_i,y_j)}+\bo(h^4).
		\ee
		On the other hand,	\eqref{FDMs:2D:u} and \eqref{F:Right:2D} lead to
		\be\label{L:h:u:h:2}
		\mathcal{L}_h (u_h)_{i,j}=f\big|_{(x,y)=(x_i,y_j)}+\tfrac{h^2}{12}[ af_x+bf_y+\Delta f  ]\big|_{(x,y)=(x_i,y_j)}.
		\ee
		By \eqref{L:h:u:3} and \eqref{L:h:u:h:2}, we prove \eqref{truncation:error:h:4}.
	\end{proof}
	
	For the fourth-order compact 9-point FDM in \cref{thm:FDM:2D}, we have two properties in following \cref{prop:1,prop:lead:h}:
	\begin{proposition}\label{prop:1}
		The nine polynomials $C_{r,\ell}$ of the variable $h$ with $r,\ell=-1,0,1$ in \eqref{C:Left:2D} satisfy
		\be\label{sign:sum}
		\begin{split} 
			& C_{\pm 1,\pm 1}|_{h=0}=C_{\pm 1,\mp 1}|_{h=0}=\tfrac{1}{6}, \qquad C_{\pm 1,0}|_{h=0}=C_{0,\pm 1}|_{h=0}=\tfrac{2}{3},\qquad C_{0,0}|_{h=0}=-\tfrac{10}{3},\\
			& \sum_{k,\ell=-1}^1C_{k,\ell}=\sum_{q=0}^4 \varpi_qh^q, \quad \text{with} \quad \varpi_0=\varpi_1=0, \quad \varpi_2=d=-\lambda  / \kappa,\\
			& \sum_{r,\ell=-1}^1C_{r,\ell}=0 \quad \text{if} \quad  \lambda=0.
		\end{split}
		\ee
		Furthermore, $-\mathcal{L}_h (u_h)_{i,j} :=\frac{1}{h^2}\sum_{r,\ell=-1}^1 -C_{r,\ell}\big|_{(x,y)=(x_i,y_j)} (u_h)_{i+r,j+\ell}=-F_{i,j}\big|_{(x,y)=(x_i,y_j)}$ with $C_{r,\ell}$ in \eqref{C:Left:2D} and $F_{i,j}$ in \eqref{F:Right:2D} in \cref{thm:FDM:2D} is the   maximum principle preserving and monotone scheme, and forms an M-matrix for the sufficiently small $h$, if $\lambda\ge 0$ in $\Omega$ in \eqref{Linear:Elliptic:2D}.
	\end{proposition}
	\begin{proof}
		\eqref{sign:sum} is obtained straightforwardly by \eqref{notation:abdf} and \eqref{C:Left:2D}. Similar to discussions in \citep{FHM2024,LiIto2001,LiZhang2020,Vejchodsky2009,XuZikatanov1999},  if $\{ C_{r,\ell}\}_{r,\ell=-1,0,1}$  satisfy 
		\be\label{sign:and:sum}
		\begin{split}
			& -C_{0,0}>0, \qquad -C_{r,\ell}\le 0, \quad \text{if} \quad (r,\ell)\neq (0,0), \quad \text{the sign condition}; \\ 
			& \sum_{r,\ell}-C_{r,\ell}\ge 0,  \quad \text{the sum condition};
		\end{split}
		\ee
		then the corresponding FDM $\sum_{r,\ell} -C_{r,\ell} (u_h)_{i+r,j+\ell}=-h^2F_{i,j}$ with $u =g$ on $\partial \Omega$ is the maximum principle preserving and monotone scheme.
		Recall that a real square matrix is an M-matrix if  diagonal entries are positive, off-diagonal entries are non-positive, row sums are nonnegative, and at least one row sum is strictly positive \citep{FHM2024,LiIto2001,LiZhang2020,Vejchodsky2009,XuZikatanov1999}. Thus, we can say that $-\mathcal{L}_h (u_h)_{i,j}$ is monotone, satisfies the discrete maximum principle, and forms an M-matrix if $h$ is sufficiently small and $\lambda \ge 0$ by \eqref{sign:sum}, \eqref{sign:and:sum}, and $u =g$ on $(x,y)\in \partial \Omega$.
	\end{proof}
	
	In the following \cref{prop:lead:h}, we provide the explicit expression of the leading term of the truncation error for the FDM in \cref{thm:FDM:2D} for \eqref{2D:linear:simple:eq}. We use \cref{prop:lead:h} to analyze the truncation error of the FDM in \cref{thm:FDM:2D} for the linear time-dependent  equation \eqref{Linear:Parabolic:2D} in \cref{subse:liner:para:2D}.
	\begin{proposition}\label{prop:lead:h}
		The leading term of the truncation error of the fourth-order compact 9-point FDM in \cref{thm:FDM:2D}  for \eqref{2D:linear:simple:eq} is 
		\be\label{trunca:2D}
		\begin{split} 
			h^4\max_{(x,y)\in \Omega}\big\{& \tfrac{1}{72} [ \Delta b+ ab_x+(b_y+d)b+2d_y   ]u_{yyy}+\tfrac{1}{144} [ 2b_y + b^2+d   ]u_{yyyy}+\tfrac{b}{120}u_{yyyyy}\\
			&+\tfrac{b d_y}{12} u_{xy}+\tfrac{1}{36} [ a_y+b_x + ab ]u_{xyyy}+\tfrac{a}{72}u_{xyyyy}-\tfrac{1}{24} [a_y + b_x + ab]u_{xxyy}\\
			& + \tfrac{b}{36}u_{xxyyy}+\tfrac{1}{72} [ \Delta a + (a_x+d)a+a_yb+ 2d_x ]u_{xxx}+\tfrac{1}{36}[ a_y+b_x+ ab  ] u_{xxxy}\\
			& +\tfrac{a}{36} u_{xxxyy}+\tfrac{1}{144} [  a^2+2a_x +d  ]u_{xxxx}+\tfrac{b}{72}u_{xxxxy}+\tfrac{a}{120}u_{xxxxx} \big\}.
		\end{split}
		\ee
	\end{proposition}
	\begin{proof}
		\eqref{trunca:2D} is obtained by calculating the coefficient of $h^4$ in \eqref{L:h:u:3} by repeating \eqref{taylor:u}--\eqref{L:h:u:3}.
	\end{proof}
	\begin{remark}
		Form \eqref{C:Left:2D} and \eqref{F:Right:2D}, we only need to compute $ \{  \tfrac{\partial^{m+n} \varrho(x,y) }{ \partial x^m \partial y^n } : m+n\le 2  \}$ with $\varrho=a,b,d,f$ to achieve the consistency order four, which guarantees the efficiency of the FDM for nonlinear  convection-diffusion-reaction equations in \cref{subse:non:liner:elli:2D,subse:non:liner:para:2D}. Furthermore, the M-matrix property can improve the stability and accuracy when solving the corresponding linear system to compute the numerical solution. 
	\end{remark}

	\subsection{Fourth-order  compact 9-point FDM for the linear time-dependent  equation in 2D}\label{subse:liner:para:2D}		
	In this subsection, we discuss the following 2D	linear time-dependent  equation in the temporal domain $I:=[0,T]$ and the spatial domain $\Omega:=(l_1,l_2)^2$: 
	\be\label{Linear:Parabolic:2D}
	\begin{cases}
		u_t-\nab\cdot (\kappa \nab u) +   \alpha u_x +   \beta u_y +   \lambda u = \phi, \qquad \qquad  &  (x,y)\in \Omega \quad \hspace{1.8mm} \text{and} \quad  t\in I,\\
		u =g,  & (x,y)\in \partial \Omega \quad \text{and} \quad  t\in I,\\
		u=u^0,  & (x,y)\in  \Omega \quad \hspace{1.8mm}  \text{and} \quad t=0,
	\end{cases}
	\ee
	where $\kappa=\kappa(x,y,t)>0$,  $\alpha= \alpha(x,y,t)$, $\beta= \beta(x,y,t)$, $\lambda= \lambda(x,y,t)$, and $\phi=\phi(x,y,t)$ are smooth variable functions in $\Omega$ and $I$. Various methods to solve \eqref{Linear:Parabolic:2D} can be found in \citep{Hundsdorfer2003}.
	We also use the uniform mesh for the time discretization as follows:
	\be\label{tim:discre} 
	u^{n+i}:=u|_{t=t_{n+i}},\qquad t_{n+i}:=(n+i) \tau, \qquad i=0,\tfrac{1}{2},1,2,3,4,\qquad \tau:=T/N_2, \qquad N_2\in \N.
	\ee 
	Form \eqref{Linear:Parabolic:2D}, we have
	\be\label{reform:para}
	u_t-\kappa \Delta u  + (\alpha-\kappa_x) u_x +   (\beta-\kappa_y) u_y +   \lambda u = \phi.
	\ee
	For simplification, we also define
	\be\label{ani:bni} 
	a^{n+i}:=\frac{\kappa^{n+i}_x-\alpha^{n+i}}{\kappa^{n+i}}, \qquad b^{n+i}:=\frac{\kappa^{n+i}_y-\beta^{n+i}}{\kappa^{n+i}}, \qquad i=\tfrac{1}{2},3,4.
	\ee
	
	\textbf{The second-order Crank-Nicolson (CN) method \citep{Burkardt2020}:}
	Now, the second-order CN method with \eqref{Linear:Parabolic:2D} and \eqref{reform:para} leads to  
	\be\label{original:CN}
	\begin{cases}
		& \hspace{-0.3cm} \displaystyle{\frac{u^{n+1/2}-u^{n}}{ \tau/2}-  \kappa^{n+1/2} \Delta u^{n+1/2}} \displaystyle{ +[\alpha^{n+1/2}-\kappa^{n+1/2}_x] u^{n+1/2}_x +[\beta^{n+1/2}-\kappa^{n+1/2}_y] u^{n+1/2}_y}\\
		&+ \lambda^{n+1/2} u^{n+1/2}   = \phi^{n+1/2}, \\
		& \hspace{-0.3cm} u^{n+1}=2u^{n+1/2}-u^{n}, \quad  \text{ with the given } u^0\in \Omega \text{ and }  \ u^{n+1/2}\in \partial \Omega \text{ in } \eqref{Linear:Parabolic:2D}.
	\end{cases}
	\ee
	Then
	\[
	\begin{split} 
		& \Delta u^{n+1/2}  +\frac{\kappa^{n+1/2}_x-\alpha^{n+1/2}}{\kappa^{n+1/2}} u^{n+1/2}_x +\frac{\kappa^{n+1/2}_y-\beta^{n+1/2}}{\kappa^{n+1/2}} u^{n+1/2}_y -\left[\frac{2}{\tau}\frac{1}{\kappa^{n+1/2}} +\frac{\lambda^{n+1/2}}{\kappa^{n+1/2}}\right] u^{n+1/2}  \\
		& = \frac{\phi^{n+1/2}}{-\kappa^{n+1/2}  }-\frac{2}{\tau}\frac{u^{n}}{\kappa^{n+1/2}}.
	\end{split} 
	\]
	Let
	\be\label{nota:dn:half}
	d^{n+1/2}:=-\frac{2}{\tau}\frac{1}{\kappa^{n+1/2}} -\frac{\lambda^{n+1/2}}{\kappa^{n+1/2}}, \qquad f^{n+1/2}:=\frac{\phi^{n+1/2}}{-\kappa^{n+1/2}  }-\frac{2}{\tau}\frac{u^{n}}{\kappa^{n+1/2}}.
	\ee
	The second-order CN method yields
	\be\label{CN:eq}
	\begin{split}
		\Delta u^{n+1/2}  +a^{n+1/2} u^{n+1/2}_x +b^{n+1/2} u^{n+1/2}_y +d^{n+1/2} u^{n+1/2}   &= f^{n+1/2},\\
		u^{n+1}& =2u^{n+1/2}-u^{n},
	\end{split}
	\ee
	where $a^{n+1/2}$ and $b^{n+1/2}$ are defined in \eqref{ani:bni} with $i=1/2$.
	
	\textbf{The third-order backward differentiation formula (BDF3) \citep{Hairer1993}:} Similarly,  the third-order BDF3 method with \eqref{Linear:Parabolic:2D} and \eqref{reform:para} implies
	\be\label{original:BDF3}
	\begin{split}
		&  \frac{11u^{n+3}-18u^{n+2}+9u^{n+1}-2u^{n} }{ 6 \tau}-  \kappa^{n+3} \Delta u^{n+3}+(\alpha^{n+3}-\kappa^{n+3}_x) u^{n+3}_x\\
		&  +(\beta^{n+3}-\kappa^{n+3}_y) u^{n+3}_y+\lambda^{n+3}u^{n+3}= \phi ^{n+3}, \\
		& \text{with the given } u^0\in \Omega \text{ and } u^{n+3}\in \partial \Omega \text{ in } \eqref{Linear:Parabolic:2D},\\
		& \text{where } u^1 \text{ and } u^2 \text{ are computed by the CN method } \eqref{CN:eq}. 
	\end{split}
	\ee
	Then
	\[
	\begin{split}
		&    \Delta u^{n+3}+\frac{ \kappa^{n+3}_x-\alpha^{n+3}}{\kappa^{n+3}} u^{n+3}_x+\frac{ \kappa^{n+3}_y-\beta^{n+3}}{\kappa^{n+3}} u^{n+3}_y\\
		& -\left(\frac{11}{6\tau } \frac{1}{\kappa^{n+3}}+\frac{\lambda^{n+3}}{\kappa^{n+3}}\right)u^{n+3}= -\frac{\phi ^{n+3}}{\kappa^{n+3}  }-\frac{1}{6 \tau }\frac{1}{\kappa^{n+3} }(18u^{n+2}-9u^{n+1}+2u^{n}).
	\end{split}
	\]
	Let 
	\be\label{nota:dn3}
	d^{n+3}:=-\frac{11}{6\tau } \frac{1}{\kappa^{n+3}}-\frac{\lambda^{n+3}}{\kappa^{n+3}}, \qquad f^{n+3}:=-\frac{\phi ^{n+3}}{\kappa^{n+3}  }-\frac{1}{6 \tau }\frac{1}{\kappa^{n+3} }(18u^{n+2}-9u^{n+1}+2u^{n}).
	\ee
	Then the third-order BDF3 method establishes
	\be\label{BDF3:eq}
	\Delta u^{n+3}  +a^{n+3} u^{n+3}_x +b^{n+3} u^{n+3}_y +d^{n+3} u^{n+3}   = f^{n+3},
	\ee
	where $a^{n+3}$ and $b^{n+3}$ are defined in \eqref{ani:bni} with $i=3$.
	
	\textbf{The fourth-order backward differentiation formula (BDF4) \citep{Hairer1993}:} Finally, the fourth-order BDF4 method with \eqref{Linear:Parabolic:2D} and \eqref{reform:para} results in
	\be\label{original:BDF4}
	\begin{split}
		&  \frac{25u^{n+4}-48u^{n+3}+36u^{n+2}-16u^{n+1}+3u^{n} }{ 12 \tau} -  \kappa^{n+4} \Delta u^{n+4}+(\alpha^{n+4}-\kappa^{n+4}_x) u^{n+4}_x\\
		&  +(\beta^{n+4}-\kappa^{n+4}_y) u^{n+4}_y+\lambda^{n+4}u^{n+4}= \phi ^{n+4}, \\
		& \text{with the given } u^0\in \Omega \text{ and } u^{n+4}\in \partial \Omega \text{ in } \eqref{Linear:Parabolic:2D},\\
		& \text{where } u^1, u^2, \text{ and } u^3 \text{ are computed by the CN method } \eqref{CN:eq}.
	\end{split}
	\ee
	Then
	\[
	\begin{split}
		&    \Delta u^{n+4}+\frac{ \kappa^{n+4}_x-\alpha^{n+4}}{\kappa^{n+4}} u^{n+4}_x+\frac{ \kappa^{n+4}_y-\beta^{n+4}}{\kappa^{n+4}} u^{n+4}_y\\
		& -\left(\frac{25}{12\tau } \frac{1}{\kappa^{n+4}}+\frac{\lambda^{n+4}}{\kappa^{n+4}}\right)u^{n+4}= -\frac{\phi ^{n+4}}{\kappa^{n+4}  }-\frac{1}{12 \tau }\frac{1}{\kappa^{n+4} }(48u^{n+3}-36u^{n+2}+16u^{n+1}-3u^{n}).
	\end{split}
	\]
	Let 
	\be\label{nota:dn4}
	d^{n+4}:=-\frac{25}{12\tau } \frac{1}{\kappa^{n+4}}-\frac{\lambda^{n+4}}{\kappa^{n+4}},
	\ee
	and
	\be\label{nota:BDF4:f}
	f^{n+4}:=-\frac{\phi ^{n+4}}{\kappa^{n+4}  }-\frac{1}{12 \tau }\frac{1}{\kappa^{n+4} }(48u^{n+3}-36u^{n+2}+16u^{n+1}-3u^{n}).
	\ee
	Then the BDF4 method develops
	\be\label{BDF4:eq}
	\Delta u^{n+4}  +a^{n+4} u^{n+4}_x +b^{n+4} u^{n+4}_y +d^{n+4} u^{n+4}   = f^{n+4},
	\ee
	where $a^{n+4}$ and $b^{n+4}$ are defined in \eqref{ani:bni} with $i=4$.
	
	In this subsection, we first use the compact 9-point FDM in \cref{thm:FDM:2D} to solve \eqref{CN:eq}, \eqref{BDF3:eq}, and \eqref{BDF4:eq}. By $\kappa>0$,  \cref{prop:1} and definitions of $d^{n+1/2}$, $d^{n+3}$, $d^{n+4}$, we have the following proposition.
	\begin{proposition}\label{propos:tau:h}
		The FDM $-\mathcal{L}_h (u_h)_{i,j}$ in \cref{thm:FDM:2D} to solve \eqref{CN:eq}, \eqref{BDF3:eq}, and \eqref{BDF4:eq} is the maximum principle preserving and monotone scheme, and forms an M-matrix for the sufficiently small $h$, if $\lambda\ge 0$ in $\Omega \cup I$ in \eqref{Linear:Parabolic:2D}. If $\tau$ and $h$ are both sufficiently small, then discrete maximum principle, monotone, and M-matrix properties can be guaranteed for any $\lambda$.
	\end{proposition}
	\begin{proof}
		By \cref{prop:1}, the FDM $-\mathcal{L}_h (u_h)_{i,j}$ in \cref{thm:FDM:2D} to solve \eqref{CN:eq}, \eqref{BDF3:eq}, and \eqref{BDF4:eq} satisfies
		\begin{align*}
			& -C_{0,0}|_{h=0}=\tfrac{10}{3}>0, \qquad -C_{r,\ell}|_{h=0}=-\tfrac{1}{6}, -\tfrac{2}{3}< 0, \quad \text{if} \quad (r,\ell)\neq (0,0), \\ 
			& \sum_{k,\ell=-1}^1 -C_{k,\ell}=\sum_{q=2}^4 -\varpi_qh^q, \quad \text{where} \\
			& \varpi_2=d=d^{n+1/2} && \hspace{-9cm} \text{ for } \eqref{CN:eq},\\               
			& \varpi_2=d=d^{n+3} && \hspace{-9cm} \text{ for }  \eqref{BDF3:eq}, \\ 
			& \varpi_2=d=d^{n+4} && \hspace{-9cm} \text{ for }  \eqref{BDF4:eq}. 
		\end{align*} 
		If  $\lambda\ge 0$, then each of $d^{n+1/2}$, $d^{n+3}$, $d^{n+4}$ in \eqref{nota:dn:half}, \eqref{nota:dn3}, and  \eqref{nota:dn4} is negative by $\kappa>0$. Furthermore, \eqref{nota:dn:half}, \eqref{nota:dn3}, and  \eqref{nota:dn4} with $\kappa>0$ also imply that  each of $d^{n+1/2}$, $d^{n+3}$, $d^{n+4}$ is negative for any $\lambda$, if $\tau$ is sufficiently small. Now we can prove this proposition by the same statement of the proof for \cref{prop:1}. 
	\end{proof}
	\begin{remark}\label{remark:order:3}
		We numerically solve \eqref{CN:eq} with $\tau=h/2$ (the CN method) to obtain $u^1$ and $u^2$ for  \eqref{BDF3:eq} (the BDF3 method), and compute $u^1,u^2,u^3$ for  \eqref{BDF4:eq} (the BDF4 method).  As \eqref{trunca:2D} in \cref{prop:lead:h} contains the function $d$,  the term $\tfrac{25}{12\tau } \frac{-1}{\kappa^{n+4}}$ in $d^{n+4}$ in \eqref{nota:dn4} implies that the degree of  $h$	of the leading term of the truncation error of the FDM in \cref{thm:FDM:2D} to solve  \eqref{BDF4:eq} is three if $\tau=rh$ with the positive constant $r$, which implies that the consistency order of the BDF4 method can only be three with $\tau=rh$. We can choose $\tau$ to be the sufficiently small constant for the different spatial mesh size $h$ to achieve the fourth-order accuracy for CN, BDF3, BDF4 methods. But the small time step $\tau$ significantly increases the computational time, especially for the small $h$. To demonstrate that the compact 9-point FDM in \cref{thm:FDM:2D} with the simple stencil (our first main result) is accurate for solving \eqref{Linear:Parabolic:2D}, although its consistency order is at most three, we also	
		propose the following algorithm to derive the fourth-order compact 9-point FDM to solve \eqref{CN:eq}, \eqref{BDF3:eq}, and \eqref{BDF4:eq} for comparisons. 
	\end{remark}	
	
	{\bf{ The algorithm to derive the fourth-order FDM:}} Let 
	\be\label{tau:f:f1:f2}
	\tau:=rh \ \text{ with the positive constant} \ r, \qquad \Qf:=\Qf_1+ \tau \Qf_2=  \Qf_1+ \tfrac{1}{rh} \Qf_2.
	\ee
	Then \eqref{CN:eq}, \eqref{BDF3:eq}, and \eqref{BDF4:eq} result in
	\be\label{para:simple:1}
	\Delta \Qu  +  \Qa \Qu_x + \Qb  \Qu_y + \tfrac{1}{rh} \Qd_1 \Qu  +  \Qd_2 \Qu   =\Qf,
	\ee
	where
	\be\label{d1:d2:f1:f2}
	\begin{split}
		& \Qu=u^{n+1/2}, \quad \Qa=a^{n+1/2}, \quad   \Qb=b^{n+1/2}, \quad   \Qd_1=\frac{-2}{\kappa^{n+1/2}}, \quad  \Qd_2=\frac{-\lambda^{n+1/2}}{\kappa^{n+1/2}},  \\
		&  \Qf_1=\frac{-\phi^{n+1/2}}{\kappa^{n+1/2}}, \quad  \Qf_2=\frac{-2u^{n}}{\kappa^{n+1/2}}, \qquad \text{for the CN method } \eqref{CN:eq};    \\
		& \Qu=u^{n+3}, \quad \Qa=a^{n+3}, \quad   \Qb=b^{n+3}, \quad   \Qd_1=\frac{-11}{6 \kappa^{n+3}}, \quad  \Qd_2=\frac{-\lambda^{n+3}}{\kappa^{n+3}},   \\
		&  \Qf_1=\frac{-\phi^{n+3}}{\kappa^{n+3}}, \quad  \Qf_2=\frac{18u^{n+2}-9u^{n+1}+2u^{n}}{-6\kappa^{n+3}}, \qquad \text{for the BDF3 method } \eqref{BDF3:eq};\\
		& \Qu=u^{n+4}, \quad \Qa=a^{n+4}, \quad   \Qb=b^{n+4}, \quad   \Qd_1=\frac{-25}{12\kappa^{n+4}}, \quad  \Qd_2=\frac{-\lambda^{n+4}}{\kappa^{n+4}},   \\
		&  \Qf_1=\frac{-\phi^{n+4}}{\kappa^{n+4}}, \quad  \Qf_2=\frac{48u^{n+3}-36u^{n+2}+16u^{n+1}-3u^{n}}{-12\kappa^{n+4}}, \qquad \text{for the BDF4 method } \eqref{BDF4:eq}. 
	\end{split}
	\ee
	
	Recall that \cref{thm:FDM:2D} is obtained by symbolically solving \eqref{EQ:1:explicit}--\eqref{EQ:2:explicit} with \eqref{property:1:2D}--\eqref{property:3:2D}. 
	Similar to \eqref{conve:diff:eq}--\eqref{EQ:2:explicit} and according to the symbolic computation, we propose the fourth-order compact 9-point FDM to solve \eqref{para:simple:1} in the
	following theorem.
	\begin{theorem} \label{thm:FDM:parabo:2D}
		Let $\kappa>0,u,\alpha,\beta,\lambda, \phi$ be smooth for $(x,y)\in \overline{\Omega}$ and $t\in I$ in \eqref{Linear:Parabolic:2D}, and define
		\be\label{Lh:uh:para}
		\begin{split}
			\mathcal{L}_h (u_h)_{i,j} :=&\frac{1}{h^2}\sum_{r,\ell=-1}^1 C_{r,\ell}\Big|_{(x,y)=(x_i,y_j)}(u_h)_{i+r,j+\ell}=\frac{1}{h^2}\sum_{r,\ell=-1}^1 \sum_{p=0}^5 c_{r,\ell,p} h^p\Big|_{(x,y)=(x_i,y_j)} (u_h)_{i+r,j+\ell}\\
			=&F_{i,j}\Big|_{(x,y)=(x_i,y_j)},
		\end{split}
		\ee
		where
		\[
		\begin{split}
			F_{i,j}:=& \text{the terms of } \Big(h^{-2} \sum_{(m,n) \in \ind_{	 3}}  \frac{ \partial^{m+n}\textup \Qf(x,y) } {\partial x^m \partial y^n}  \\
			&\times \sum_{r,\ell=-1}^{1} \sum_{p=0}^5 c_{r,\ell,p} h^p\Big|_{(x,y)=(x_i,y_j)} H_{5,m,n}  (rh, \ell h)\Big) \text{ with degree}   \le 3 \text{ in } h,
		\end{split}
		\]
		and $c_{r,\ell,p}$ can be uniquely determined by solving
		\[
		I_{5,m,n}=\sum_{r,\ell=-1}^{1} G_{5,m,n}  (rh, \ell h) \sum_{p=0}^5 c_{r,\ell,p} \big|_{(x,y)=(x_i,y_j)} h^p =\bo(h^{6}) \quad \mbox{for all}  \quad (m,n)\in \ind_{5}^1, 
		\]
		with
		\[
		\text{the lowest degree term of } h \ (\text{the }h^0 \text{ term}) \text{ in }  F_{i,j}=   \textup\Qf,
		\]
		and satisfying
		\[
		\begin{split}
			\{&c_{-1, 0, 5}, c_{-1, 1, p_1}, c_{0, -1, 5}, c_{0, 0, p_1}, c_{0, 1, p_2}, c_{1, -1, p_2}, c_{1, 0, p_3}, c_{1, 1, p_4}\}=\{0\}\quad \text{and} \quad c_{0,0,0}=-10/3,
		\end{split}
		\]
		for $p_1=4,5$, $p_2=3,4,5$, $p_3=2,\dots,5$, $p_4=1,\dots,5$, $\ind_{5}^1$ is defined in \eqref{ind:2D}, $G_{5,m,n}(x,y)$ and $H_{5,m,n}(x,y)$ can be uniquely determined by applying \citep[eqs. (3)--(36)]{FengTrenchea2026} with replacing \eqref{conve:diff:eq} by \eqref{para:simple:1} and considering $rh$ as the constant. Then \eqref{Lh:uh:para} achieves a consistency order four for the PDE \eqref{para:simple:1} at $(x_i,y_j)$ in \eqref{xiyj:2D:space}.
	\end{theorem}
	
	\begin{remark}\label{remark:thm22}
		The existence of the fourth-order compact 9-point FDM in \eqref{Lh:uh:para} is confirmed by the symbolic computation, but in \cref{thm:FDM:parabo:2D}, $C_{r,\ell}$, $F_{i,j}$, $G_{5,m,n}$, and $H_{5,m,n}$ are not explicitly given because of their long expressions. Recall that \eqref{property:1:2D}--\eqref{property:3:2D} help us to obtain the short and easy stencil in \cref{thm:FDM:2D}. But the expression of the FDM in \cref{thm:FDM:parabo:2D}
		is long and complicated, and we do not observe a similar expression for $c_{r,\ell,p}$ to solve \eqref{para:simple:1} as in \eqref{property:1:2D}--\eqref{property:3:2D}. So we only propose the above algorithm to obtain the unique FDM of  \cref{thm:FDM:parabo:2D} used in this paper. For the PDE \eqref{para:simple:1}, \cref{thm:FDM:2D} and  \cref{thm:FDM:parabo:2D}  are third-order and fourth-order consistent, respectively,
		but numerical results of \cref{thm:FDM:2D} and  \cref{thm:FDM:parabo:2D} for the 2D	linear time-dependent convection-diffusion-reaction  equation \eqref{Linear:Parabolic:2D} in \cref{Example:2} in \cref{sec:Numeri} are almost the same (see details in \cref{remark:examp:2}). I.e., the proposed compact 9-point FDM in \cref{thm:FDM:2D} with the short and explicit expression can yield accurate and convincing solution for  \eqref{Linear:Parabolic:2D}.
		So \cref{thm:FDM:2D} is the main result for all 2D convection-diffusion-reaction equations (linear, nonlinear, time-independent, time-dependent) in this paper. 
		Thus readers do not need to worry the messy derivation and undefined notations of  \cref{thm:FDM:parabo:2D} which are used to demonstrate the efficiency of \cref{thm:FDM:2D}. But we also plan to modify \cref{thm:FDM:parabo:2D} to derive the fourth-order compact 9-point FDM with the short and easy stencil to solve \eqref{para:simple:1} in future. 
	\end{remark}		
	
	Due to the long expression of $C_{r,\ell}$ in \cref{thm:FDM:parabo:2D}, we could not present the explicit formula for each $C_{r,\ell}$. But we can still establish following two properties  in \cref{prop:M:matrix:para} and \cref{remark:three:deriva}:
	\begin{proposition}\label{prop:M:matrix:para}
		The nine polynomials  $C_{r,\ell}= \sum_{p=0}^5 c_{r,\ell,p} h^p$ with $r,\ell=-1,0,1$ in \eqref{Lh:uh:para} satisfy
		\[
		\begin{split} 
			& C_{\pm 1,\pm 1}|_{h=0}=C_{\pm 1,\mp 1}|_{h=0}=\tfrac{1}{6}, \qquad C_{\pm 1,0}|_{h=0}=C_{0,\pm 1}|_{h=0}=\tfrac{2}{3},\qquad C_{0,0}|_{h=0}=-\tfrac{10}{3},\\
			& \sum_{k,\ell=-1}^1C_{k,\ell}=\sum_{q=0}^5 \varpi_qh^q, \quad \text{with} \quad 	\varpi_0=0 \quad \text{and} \quad \varpi_1=\textup \Qd_1/r,
		\end{split}
		\]
		where  $r=\tau/h>0$ ($\tau$ is the time step)
		\begin{align*}
			& \textup \Qd_1=\frac{-2}{\kappa^{n+1/2}},   && \hspace{-4cm} \text{for the CN method } \eqref{CN:eq};\\
			& \textup \Qd_1=\frac{-11}{6 \kappa^{n+3}},  && \hspace{-4cm}  \text{for the BDF3 method } \eqref{BDF3:eq};\\
			& \textup \Qd_1=\frac{-25}{12\kappa^{n+4}},  && \hspace{-4cm}  \text{for the BDF4 method } \eqref{BDF4:eq}.
		\end{align*}
		Furthermore, $-\mathcal{L}_h (u_h)_{i,j} :=\frac{1}{h^2}\sum_{r,\ell=-1}^1 -C_{r,\ell}\big|_{(x,y)=(x_i,y_j)} (u_h)_{i+r,j+\ell}=-F_{i,j}\big|_{(x,y)=(x_i,y_j)}$  in \cref{thm:FDM:parabo:2D} is the maximum principle preserving and monotone scheme, and forms an M-matrix for any $\lambda$ in \eqref{Linear:Parabolic:2D}, if $h$ is sufficiently small.
	\end{proposition}
	\begin{proof}
		Similar to the proof of \cref{prop:1}
	\end{proof}		
	\begin{remark}\label{remark:three:deriva}
		The symbolic computation indicates that we need to calculate $ \{  \tfrac{\partial^{m+n} \varrho(x,y,t_{n+i}) }{ \partial x^m \partial y^n } : m+n\le 3 \}$ with $\varrho=\textup \Qa,\textup \Qb,\textup \Qd_1, \textup \Qd_2, \textup \Qf$ defined in \eqref{tau:f:f1:f2} and \eqref{d1:d2:f1:f2} for the FDM in \cref{thm:FDM:parabo:2D}.
	\end{remark}

	\subsection{Fourth-order  compact 9-point FDM for the nonlinear time-independent  equation in 2D}	\label{subse:non:liner:elli:2D}
	In the subsection, we consider the nonlinear version of \eqref{Linear:Elliptic:2D}. 
	For simplification, we assume that $\kappa>0$, $\alpha$, $\beta$, and $\lambda$ only depend on $u(x,y)$ and are smooth in $\Omega$. I.e., we discuss the following nonlinear time-independent  equation in $\Omega=(l_1,l_2)^2$:
	\be\label{Non:Linear:Elliptic:2D}
	\begin{cases}
		-\nab\cdot (\kappa(u) \nab u) +   \alpha(u) u_x +   \beta(u) u_y +   \lambda(u) u = \phi, \qquad \qquad  &  (x,y)\in \Omega,\\
		u =g,  & (x,y)\in \partial \Omega.
	\end{cases}
	\ee
	Clearly, \eqref{Non:Linear:Elliptic:2D} produces
	\[
	-(\kappa(u) u_x)_x-(\kappa(u) u_y)_y  + \alpha(u) u_x +   \beta(u) u_y +   \lambda(u) u = \phi, 
	\]
	\[
	-\kappa(u) \Delta u-\kappa_u(u)u_x u_x-\kappa_u(u)u_y u_y  + \alpha(u) u_x +   \beta(u) u_y +   \lambda(u) u = \phi,
	\]
	and
	\be\label{nonlinear:2D:ellip}
	\Delta u  + \frac{\kappa_u(u)u_x-\alpha(u)}{\kappa(u)} u_x +   \frac{\kappa_u(u)u_y-\beta(u)}{\kappa(u)} u_y -   \frac{\lambda(u)}{\kappa(u)} u = -\frac{\phi }{\kappa(u)}.
	\ee
	Now, we use the fixed point method (iteration method) to solve the nonlinear equation \eqref{nonlinear:2D:ellip}. We define $u_{\Qk}$ as the solution in the $\Qk$th iteration. 	
	Let
	\begin{align} \label{ak:bk:dk:fk}
		&	a_\Qk:=\frac{\kappa_u(u_\Qk)(u_\Qk)_x-\alpha(u_\Qk)}{\kappa(u_\Qk)}, &&  \hspace{-3cm} b_\Qk:=\frac{\kappa_u(u_\Qk)(u_\Qk)_y-\beta(u_\Qk)}{\kappa(u_\Qk)}, \notag \\
		& d_\Qk:=-   \frac{\lambda(u_\Qk)}{\kappa(u_\Qk)}, && \hspace{-3cm} f_\Qk:=-\frac{\phi }{\kappa(u_\Qk)}.
	\end{align}
	Then we reformulate \eqref{nonlinear:2D:ellip} as
	\be\label{nonlinear:2D:ellip:ite}
	\Delta u_{\Qk+1}  + a_\Qk (u_{\Qk+1})_x +  b_\Qk(u_{\Qk+1})_y +d_\Qk u_{\Qk+1} =f_\Qk.
	\ee
	Now, we use the fourth-order compact 9-point FDM in \cref{thm:FDM:2D} to solve \eqref{nonlinear:2D:ellip:ite} until $\Qk=40$ with the initial guess $u_{\Qk}=0$ with $\Qk=0$.

	\subsection{Fourth-order  compact 9-point FDM for the nonlinear time-dependent  equation in 2D}	\label{subse:non:liner:para:2D}
	In this section, we discuss the nonlinear version of the 2D time-dependent  equation \eqref{Linear:Parabolic:2D} as follows:
	\be\label{parabo:nonlinear:2D}
	u_t-\nab\cdot (\kappa(u) \nab u) +   \alpha(u) u_x +   \beta(u) u_y  +  \lambda(u) u = \phi.
	\ee
	For simplification, we assume that $\kappa>0$, $\alpha$, $\beta$, and $\lambda$ only depend on $u(x,y,t)$ and are smooth in $\Omega$, and define
	\be\label{akni}
	\begin{split}
		& a_\Qk^{n+i}:=\frac{\kappa_u(u^{n+i}_\Qk)(u^{n+i}_\Qk)_x-\alpha(u^{n+i}_\Qk)}{\kappa(u^{n+i}_\Qk)}, \quad b_\Qk^{n+i}:=\frac{\kappa_u(u^{n+i}_\Qk)(u^{n+i}_\Qk)_y-\beta(u^{n+i}_\Qk)}{\kappa(u^{n+i}_\Qk)},\\
		&  u_{\Qk}^{n+i} \text{ is the solution in the } \Qk\text{th iteration and  at } t=(n+i)\tau,\\
		&  u^{n+i}:=u_{\Qk}^{n+i} \text{ with } \Qk=20,
	\end{split}
	\ee
	where $\tau=T/N_2$ and $N_2\in \N$.
	From \eqref{parabo:nonlinear:2D},
	\[
	u_t-\kappa(u) \Delta u  + (\alpha(u)-\kappa_u(u)u_x) u_x +   (\beta(u)-\kappa_u(u)u_y) u_y  +   \lambda(u) u = \phi.
	\]
	Similar to \eqref{original:CN}, \eqref{CN:eq}, and \eqref{nonlinear:2D:ellip:ite}, the CN method implies
	\be\label{CN:2D:itera}
	\begin{split}
		& \Delta u_{\Qk+1}^{n+1/2}  +a_\Qk^{n+1/2} (u^{n+1/2}_{\Qk+1})_x +b_\Qk^{n+1/2}(u^{n+1/2}_{\Qk+1})_y +d_\Qk^{n+1/2} u^{n+1/2}_{\Qk+1}   = f^{n+1/2}_\Qk,\\
		& u^{n+1} =2u^{n+1/2}-u^{n}, \quad \Qk:=\Qk+1, \\
		&  \text{with the given } u^0\in \Omega \text{ and }  \ u^{n+1/2}\in \partial \Omega,\\
		&  \text{the initial guess } u^{n+1/2}_{\Qk}=u^{n} \text{ with }\Qk=0,
	\end{split}
	\ee
	where $a_\Qk^{n+1/2}$ and $b_\Qk^{n+1/2}$ are defined in \eqref{akni} with $i=1/2$, 
	\be\label{dk:CN:2D}
	d_\Qk^{n+1/2}:=-\frac{2}{\tau}\frac{1}{\kappa(u^{n+1/2}_\Qk)} -\frac{\lambda(u^{n+1/2}_\Qk)}{\kappa(u^{n+1/2}_\Qk)},
	\ee 
	and
	\be\label{fk:CN:2D}
	f^{n+1/2}_\Qk:=\frac{\phi^{n+1/2}}{-\kappa(u^{n+1/2}_\Qk)  }-\frac{2}{\tau}\frac{u^{n}}{\kappa(u^{n+1/2}_\Qk)}.
	\ee
	
	Similar to \eqref{original:BDF3}, \eqref{BDF3:eq}, and \eqref{nonlinear:2D:ellip:ite}, the BDF3 method leads to 
	\be\label{BDF3:2D:itera}
	\begin{split}
		& \Delta u_{\Qk+1}^{n+3}  +a_\Qk^{n+3} (u_{\Qk+1}^{n+3})_x +b_\Qk^{n+3} (u_{\Qk+1}^{n+3})_y  +d_\Qk^{n+3} u_{\Qk+1}^{n+3}   = f^{n+3}_\Qk,\\
		& \text{with the given } u^0\in \Omega \text{ and } u^{n+3}\in \partial \Omega, \\
		& u^1 \text{ and } u^2 \text{ are computed by the CN method in \eqref{CN:2D:itera} with } \tau=h/2,\\
		& \Qk:=\Qk+1, \quad \text{the initial guess } u^{n+3}_{\Qk}=u^{n+2} \text{ with }\Qk=0,
	\end{split}
	\ee
	where $a_\Qk^{n+3}$ and $b_\Qk^{n+3}$ are defined in \eqref{akni} with $i=3$, 
	\be\label{dk:BDF3:2D}
	d_\Qk^{n+3}:=-\frac{11}{6\tau } \frac{1}{\kappa(u^{n+3}_\Qk)}-\frac{\lambda(u^{n+3}_\Qk)}{\kappa(u^{n+3}_\Qk)},
	\ee
	and
	\be\label{fk:BDF3:2D}
	f^{n+3}_\Qk:=-\frac{\phi ^{n+3}}{\kappa (u^{n+3}_\Qk) }-\frac{1}{6 \tau }\frac{1}{\kappa(u^{n+3}_\Qk) }(18u^{n+2}-9u^{n+1}+2u^{n}).
	\ee
	
	Similar to  \eqref{original:BDF4}, \eqref{BDF4:eq} and \eqref{nonlinear:2D:ellip:ite}, the BDF4 method leads to 
	\be\label{BDF4:2D:itera}
	\begin{split}
		& \Delta u_{\Qk+1}^{n+4}  +a_\Qk^{n+4} (u_{\Qk+1}^{n+4})_x +b_\Qk^{n+4} (u_{\Qk+1}^{n+4})_y   +d_\Qk^{n+4} u_{\Qk+1}^{n+4}   = f^{n+4}_\Qk,\\
		& \text{with the given } u^0\in \Omega \text{ and } u^{n+4}\in \partial \Omega, \\
		& u^1, u^2, \text{ and } u^3 \text{ are computed by the CN method in \eqref{CN:2D:itera} with } \tau=h/2,\\
		& \Qk:=\Qk+1, \quad \text{the initial guess } u^{n+4}_{\Qk}=u^{n+3} \text{ with }\Qk=0,
	\end{split}
	\ee
	where $a_\Qk^{n+4}$ and $b_\Qk^{n+4}$ are defined in \eqref{akni} with $i=4$, 
	\be\label{dk:BDF4:2D}
	d_\Qk^{n+4}:=-\frac{25}{12\tau } \frac{1}{\kappa(u^{n+4}_\Qk)}-\frac{\lambda(u^{n+4}_\Qk)}{\kappa(u^{n+4}_\Qk)},
	\ee
	\be\label{fk:BDF4:2D}
	f^{n+4}_\Qk:=-\frac{\phi ^{n+4}}{\kappa (u^{n+4}_\Qk) }-\frac{1}{12 \tau }\frac{1}{\kappa (u^{n+4}_\Qk)}(48u^{n+3}-36u^{n+2}+16u^{n+1}-3u^{n}).
	\ee
	
	Now, we use the compact 9-point FDMs in \cref{thm:FDM:2D}  and \cref{thm:FDM:parabo:2D} to solve \eqref{CN:2D:itera}, \eqref{BDF3:2D:itera}, and \eqref{BDF4:2D:itera} until $\Qk=20$.

	\subsection{Fourth-order  compact 19-point FDM for the linear time-independent   equation in 3D}\label{subse:liner:elli:3D}
	In this subsection, we discuss the following 3D linear time-independent  equation with the Dirichlet boundary condition:
	\be\label{Model：Elliptic:3D}
	\begin{cases}
		-\nab\cdot (\kappa \nab u) +   \alpha u_x +   \beta u_y  +   \gamma u_z +   \lambda u = \phi, \qquad \qquad  &  (x,y)\in \Omega,\\
		u =g,  & (x,y)\in \partial \Omega,
	\end{cases}
	\ee
	where $\kappa=\kappa(x,y,z)>0$, $\alpha= \alpha(x,y,z)$, $\beta= \beta(x,y,z)$, $\gamma=\gamma(x,y,z)$, $\lambda= \lambda(x,y,z)$, and $\phi=\phi(x,y,z)$ are smooth variable functions in the cubic domain
	$\Omega:=(l_1,l_2)^3$. 
	Similar to \eqref{2D:linear:simple:eq}, \eqref{Model：Elliptic:3D} implies
	\be\label{simpli:PDE:3D}
	\Delta u  + a u_x +  bu_y +  cu_z +du =f,
	\ee
	where 
	\be\label{abcd:3D}
	a:=\frac{\kappa_x-\alpha}{\kappa}, \qquad b:=\frac{\kappa_y-\beta}{\kappa}, \qquad c:=\frac{\kappa_z-\gamma}{\kappa}, \qquad d:=-   \frac{\lambda}{\kappa}, \qquad f:=-\frac{\phi }{\kappa}.
	\ee
	Similar to \eqref{xiyj:2D:space}, we also use the uniform Cartesian grid to discretize the spatial  domain $\Omega=(l_1,l_2)^3$ as follows:
	\be \label{xiyj:3D:space}
	x_i:=l_1+i h, \quad y_j:=l_1+j h, \quad z_k:=l_1+k h, \quad  i,j,k=0,\ldots,N_1,  \quad \text{and} \quad h:=(l_2-l_1)/N_1,
	\ee
	where $N_1\in \N$. We also define $(u_h)_{i,j,k}$ and $u_{i,j,k}$ are the values of the numerical solution $u_h$ computed by the proposed FDM and the exact solution $u$ at the grid point $(x_i,y_j,z_k)$ in \eqref{xiyj:3D:space}, respectively.
	
	{\bf{ The algorithm to derive the fourth-order FDM:}}
	Let 	\[
	u^{(m,n,q)}:=\frac{\partial^{m+n+q}u(x,y,z)}{\partial x^m \partial y^n \partial z^q}\Big|_{(x,y,z)=(x_i,y_j,z_k)},\qquad f^{(m,n,q)}:=\frac{\partial^{m+n+q}f(x,y,z)}{\partial x^m \partial y^n \partial z^q}\Big|_{(x,y,z)=(x_i,y_j,z_k)},
	\]
	\be\label{ind:3D}
	\ind_{5}^{ 1}:=\{ (m,n,q)\in \N^3 \; : \; m=0,1, m+n+q\le 5\},\qquad 	\ind_{3}:=\{ (m,n,q)\in \N^3 \; : \;  m+n+q\le 3 \},
	\ee
	and define
	\be\label{Lh:u:3D}
	\mathcal{L}_h u_{i,j,k} :=\frac{1}{h^2}\sum_{r,\ell,s=-1}^1 C_{r,\ell,s}\Big|_{(x,y,z)=(x_i,y_j,z_k)} u_{i+r,j+\ell,k+s},
	\ee
	\be\label{FDMs:3D:u}
	\mathcal{L}_h (u_h)_{i,j,k} :=\frac{1}{h^2}\sum_{r,\ell,s=-1}^1 C_{r,\ell,s}\Big|_{(x,y,z)=(x_i,y_j,z_k)} (u_h)_{i+r,j+\ell,k+s}=F_{i,j,k}\Big|_{(x,y,z)=(x_i,y_j,z_k)},
	\ee
	\[
	C_{r,\ell,s}:=\sum_{p=0}^M c_{r,\ell,s,p} h^p,
	\]
	with
	\begin{align}\label{c:rlsp:pro:1}
		c_{r,\ell,s,p}=&\sum_{\varsigma} \sigma_{r,\ell,s,p,\varsigma} \prod_{\mu_{\varsigma,1},\nu_{\varsigma,1},\omega_{\varsigma,1},\rho_{\varsigma,1}}\left(\frac{\partial^{\mu_{\varsigma,1}+\nu_{\varsigma,1}+\omega_{\varsigma,1}} a }{\partial x^{\mu_{\varsigma,1}} \partial y^{\nu_{\varsigma,1}} \partial z^{\omega_{\varsigma,1}} }\right)^{\rho_{\varsigma,1}}  
		\prod_{\mu_{\varsigma,2},\nu_{\varsigma,2},\omega_{\varsigma,2},\rho_{\varsigma,2}}\left(\frac{\partial^{\mu_{\varsigma,2}+\nu_{\varsigma,2}+\omega_{\varsigma,2}} b }{\partial x^{\mu_{\varsigma,2}} \partial y^{\nu_{\varsigma,2}} \partial z^{\omega_{\varsigma,2}} }\right)^{\rho_{\varsigma,2}}\notag\\
		& \times \prod_{\mu_{\varsigma,3},\nu_{\varsigma,3},\omega_{\varsigma,3},\rho_{\varsigma,3}}\left(\frac{\partial^{\mu_{\varsigma,3}+\nu_{\varsigma,3}+\omega_{\varsigma,3}} c }{\partial x^{\mu_{\varsigma,3}} \partial y^{\nu_{\varsigma,3}} \partial z^{\omega_{\varsigma,3}} }\right)^{\rho_{\varsigma,3}}
		\prod_{\mu_{\varsigma,4},\nu_{\varsigma,4},\omega_{\varsigma,4},\rho_{\varsigma,4}}\left(\frac{\partial^{\mu_{\varsigma,4}+\nu_{\varsigma,4}+\omega_{\varsigma,4}} d }{\partial x^{\mu_{\varsigma,4}} \partial y^{\nu_{\varsigma,4}} \partial z^{\omega_{\varsigma,4}} }\right)^{\rho_{\varsigma,4}},
	\end{align}
	where 
	\be\label{c:rlsp:pro:2}
	\text{each }	\sigma_{r,\ell,s,p,\varsigma} \text{ is independent of } a,b,c,d,f \text{ defined in } \eqref{abcd:3D},
	\ee
	and
	\be\label{c:rlsp:pro:3}
	\begin{split}
		p=&(\mu_{\varsigma,1}+\nu_{\varsigma,1}+\omega_{\varsigma,1}+1)\rho_{\varsigma,1}+	(\mu_{\varsigma,2}+\nu_{\varsigma,2}+\omega_{\varsigma,2}+1)\rho_{\varsigma,2} \\
		&+	(\mu_{\varsigma,3}+\nu_{\varsigma,3}+\omega_{\varsigma,3}+1)\rho_{\varsigma,3}+	(\mu_{\varsigma,4}+\nu_{\varsigma,4}+\omega_{\varsigma,4}+2)\rho_{\varsigma,4},
	\end{split}
	\ee
	for each $\varsigma$.
	Similar to \eqref{Fij} with \eqref{Imn}, we define
	\[
	F_{i,j,k}:=\text{the terms of } \Big(h^{-2} \sum_{(m,n) \in \ind_{	 3}} f^{(m,n,q)} J_{3,m,n,q}\Big) \text{ with degree}   \le 3 \text{ in } h,
	\]
	where
	\be\label{J3mnp} 
	\begin{split}
		J_{3,m,n,q}:= \sum_{r,\ell,s=-1}^{1} C_{r,\ell,s} \big|_{(x,y,z)=(x_i,y_j,z_k)} H_{5,m,n,q}  (rh, \ell h, sh).
	\end{split}
	\ee
	
	Similar to \eqref{EQ:1:explicit}--\eqref{property:3:2D}, the $C_{r,\ell,s}$ of the fourth-order compact 27-point FDM in \eqref{FDMs:3D:u} for \eqref{simpli:PDE:3D} can be obtained by computing $\sigma_{r,\ell,s,p,\varsigma}$ by solving 
	\be\label{G5mnp} 
	\sum_{r,\ell,s=-1}^{1} G_{5,m,n,q}  (rh, \ell h, sh) \sum_{p=0}^5 c_{r,\ell,s,p} \big|_{(x,y,z)=(x_i,y_j,z_k)} h^p =\bo(h^{6}) \quad \mbox{for all}  \quad (m,n,q)\in \ind_{5}^1, 
	\ee
	with also satisfying \eqref{c:rlsp:pro:1}--\eqref{c:rlsp:pro:3} and
	\[
	C_{0,0,0}|_{(x,y,z,h)=(x_i,y_j,z_k,0)}=c_{0,0,0,0}|_{(x,y,z)=(x_i,y_j,z_k)}\ne 0,  \qquad  F_{i,j,k}|_{(x,y,z,h)=(x_i,y_j,z_k,0)}=f(x_i,y_j,z_k).
	\]
	
	Apply \citep[eqs. (3)--(36)]{FengTrenchea2026} with replacing \eqref{conve:diff:eq} by \eqref{simpli:PDE:3D}, then the corresponding $G_{5,m,n,q}(x,y,z)$  in \eqref{G5mnp} and $H_{5,m,n,q}(x,y,z)$ in \eqref{J3mnp}  for \eqref{simpli:PDE:3D} can be derived uniquely. Similar to \cref{subse:liner:elli:2D}, expressions of $G_{5,m,n,q}(x,y,z)$ and $H_{5,m,n,q}(x,y,z)$ are not given explicitly, while the main result \cref{thm:FDM:3D}  for all 3D convection-diffusion-reaction equations (linear, nonlinear, time-independent, and time-dependent) in this paper does need explicit formulas of $G_{5,m,n,q}(x,y,z)$ and $H_{5,m,n,q}(x,y,z)$ (see \cref{thm:FDM:3D} with \eqref{C:1:Left:3D}--\eqref{F:Right:3D} for details).
	
	By the symbolic computation, we have the following $C_{r,\ell,s}$ and $F_{i,j,k}$ of the fourth-order compact 27-point FDM in  \eqref{FDMs:3D:u} for \eqref{simpli:PDE:3D}:\\
	$C_{r,\ell,s}$ with $r=-1$, $\ell,s\in \{-1,0,1\}$:
	\begin{align}\label{C:1:Left:3D}
		&C_{-1,-1,-1}=   C_{-1,-1,1}=  C_{-1,1,-1}= C_{-1,1,1}:=0, \notag \\	&C_{-1,-1,0}:=  \tfrac{1}{6}-\tfrac{\varphi_1}{12}h-\tfrac{aa_x}{24}h^3-\tfrac{1}{12}[ad_x+cd_z]h^4, \quad 
		C_{-1,0,-1}:=  \tfrac{1}{6}-\tfrac{\varphi_2}{12}h,  \notag \\
		&C_{-1,0,0}:= \tfrac{1}{3}-\tfrac{a}{6}h+\tfrac{1}{12}[a \varphi_7+d+[\varphi_7]_x+\varphi_9]h^2-\tfrac{1}{12}[d \varphi_1+b[\varphi_7]_y ]h^3 +\tfrac{cd_z}{12}h^4,\\
		&C_{-1,0,1}:=\tfrac{1}{6}-\tfrac{\varphi_5}{12}h-\tfrac{\varphi_{20}}{12}h^2+\tfrac{1}{24}[ ad +b(c_y+d+[\varphi_7]_y)-a_zc-2d_x-\Delta a ]h^3, \notag \\
		&C_{-1,1,0}:=\tfrac{1}{6}-\tfrac{\varphi_4}{12}h-\tfrac{\varphi_{19}}{12}h^2+ \tfrac{b}{24}[b_y+d]h^3; \notag 
	\end{align}
	$C_{r,\ell,s}$ with $r=0$, $\ell,s\in \{-1,0,1\}$:
	\begin{align}\label{C:2:Left:3D}
		&C_{0,-1,-1}:=\tfrac{1}{6}-\tfrac{\varphi_3}{12}h, \notag \\
		&C_{0,-1,0}:=\tfrac{1}{3}-\tfrac{b}{6}h+\tfrac{1}{12}[b \varphi_7+d+[\varphi_7]_y+\varphi_{10}]h^2 \notag \\
		&\qquad \qquad -\tfrac{1}{12}[2\varphi_{14}+\varphi_{16}+\varphi_{18} ]h^3 +\tfrac{1}{12}[ad_x+\varphi_{15}]h^4, \notag \\
		&C_{0,-1,1}:=\tfrac{1}{6}-\tfrac{\varphi_6}{12}h-\tfrac{\varphi_{21}}{12}h^2 \notag \\
		&\qquad \qquad+\tfrac{1}{24}[a(a_x+[\varphi_4]_x)-cb_z-2d_{y}-b_{zz}+4\varphi_{14} +\varphi_{16}+2\varphi_{18}]h^3, \notag \\
		&C_{0,0,-1}:= \tfrac{1}{3}-\tfrac{c}{6}h+\tfrac{1}{12}[ c \varphi_7+d+[\varphi_7]_z+\varphi_{11} ]h^2-\tfrac{bd_y}{12}h^4, \notag \\
		&C_{0,0,0}:=-4-\tfrac{1}{6}[ a \varphi_7+b \varphi_3+c^2-3d+\varphi_9+\varphi_{10}+\varphi_{11}+\varphi_{12} ]h^2 \notag \\
		&\qquad \qquad+\tfrac{1}{12}[  a \varphi_{13} +b[\varphi_7]_y-c[\varphi_7]_z+d \varphi_8 -b_{zz} +2\varphi_{14} +\varphi_{16}+\varphi_{18}    ]h^3,\\
		&C_{0,0,1}:=  \tfrac{1}{3}+\tfrac{c}{6}h+\tfrac{1}{12}[  c \varphi_7 +d+[\varphi_7]_z+\varphi_{11}  ]h^2 +\tfrac{1}{12}[  c[\varphi_7]_z-a[\varphi_4+\varphi_5]_x\notag \\
		&\qquad \qquad-b[\varphi_7]_y -d \varphi_8+b_{zz}  -2\varphi_{14}-\varphi_{16}-\varphi_{18}   ]h^3+\tfrac{d_{xx}}{12}h^4, \notag \\
		&C_{0,1,-1}:=\tfrac{1}{6}+\tfrac{\varphi_6}{12}h-\tfrac{\varphi_{21}}{12}h^2-\tfrac{1}{24}[  c([\varphi_2]_z+d)+ac_x+2(d_x+d_z) -b_{zz}+\varphi_{18} ]h^3 \notag \\
		&\qquad \qquad +\tfrac{1}{12} [   ad_x+bd_y+ \Delta d   ]h^4, \notag \\
		&C_{0,1,0}:=\tfrac{1}{3}+\tfrac{b}{6}h+\tfrac{1}{12}[ b \varphi_7+d+\varphi_{10}+[\varphi_7]_y  ]h^2+\tfrac{1}{12}[ c([\varphi_7]_z+d)-a \varphi_{13}  ]h^3 -\tfrac{d_{xx}}{12} h^4, \notag \\
		&C_{0,1,1}:=\tfrac{1}{6}+\tfrac{\varphi_3}{12}h-\tfrac{1}{24}[ c([\varphi_7]_z+d)-a(a_x+\varphi_{13})-2\varphi_{14} -\varphi_{16} - \varphi_{18} ]h^3; \notag
	\end{align}
	$C_{r,\ell,s}$ with $r=1$, $\ell,s\in \{-1,0,1\}$:
	\begin{align}\label{C:3:Left:3D}
		&C_{1,-1,-1}=	C_{1,-1,1}=  C_{1,1,-1}=  C_{1,1,1}:=0, \notag \\	
		&C_{1,-1,0}:= \tfrac{1}{6}+\tfrac{\varphi_4}{12}h-\tfrac{\varphi_{19}}{12}h^2-\tfrac{1}{24}[ a_x a + b(b_y+d) ]h^3  +\tfrac{\varphi_{17}}{12}h^4,\notag \\
		&C_{1,0,-1}:=\tfrac{1}{6}+\tfrac{\varphi_5}{12}h-\tfrac{\varphi_{20}}{12}h^2+\tfrac{1}{24}[ a_zc-bc_y+2d_x-b_{zz}+  \Delta a ]h^3 \notag \\
		&\qquad \qquad  -\tfrac{1}{12}[ ad_x+\varphi_{17} ]h^4,\\
		&C_{1,0,0}:= \tfrac{1}{3}+\tfrac{a}{6}h+\tfrac{1}{12}[ a \varphi_7+d+[\varphi_7]_x +\varphi_9 ]h^2+\tfrac{1}{12}[aa_x+b_{zz}]h^3, \notag \\
		&C_{1,0,1}:=\tfrac{1}{6}+\tfrac{\varphi_2}{12}h+\tfrac{1}{24}[d \varphi_1 +b[\varphi_7]_y-b_{zz}]h^3,\quad 	C_{1,1,0}=\tfrac{1}{6}+\tfrac{\varphi_1}{12}h; \notag 
	\end{align}
	where
	\begin{align}\label{r:Left:3D}
		& \varphi_1:=a+b,\quad \varphi_2:=a+c, \quad \varphi_3:=b+c, \notag  \\
		& \varphi_4:=a-b, \quad \varphi_5:=a-c, \quad \varphi_6:=b-c, \notag \\
		& \varphi_7:=a+b+c, \quad \varphi_8:=a+b-c, \notag \\
		& \varphi_9:=a_x+a_y+a_z,\quad \varphi_{10}:=b_x+b_y+b_z, \notag \\
		&\varphi_{11}:=c_x+c_y+c_z,\quad  \varphi_{12}:=a_x+b_y+c_z, \\
		&\varphi_{13}:=a_x-b_x-c_x, \quad  \varphi_{14}:=d_x+d_y+d_z, \notag \\
		& \varphi_{15}:=ad_x+bd_y+cd_z, \quad  \varphi_{16}:= b_{xx}+b_{yy},   \notag \\
		& \varphi_{17}:= d_{yy}+d_{zz}, \quad \varphi_{18}:=\Delta a+\Delta c, \notag \\
		& \varphi_{19}:=ab+a_y+b_x, \quad \varphi_{20}:=ac+a_z+c_x, \notag \\
		& \varphi_{21}:=bc+b_z+c_y; \notag 
	\end{align}	
	and the right-hand side of the FDM in  \eqref{FDMs:3D:u} for \eqref{simpli:PDE:3D} is defined as
	\be\label{F:Right:3D}
	F_{i,j,k}:=f +\tfrac{1}{12}[ af_x+bf_y+cf_z+\Delta f  ]h^2.
	\ee
	
	In summary, we propose the fourth-order compact 27-point (it is actually a 19-point FDM, see \cref{rem:only:19:point}) FDM  for \eqref{Model：Elliptic:3D} and \eqref{simpli:PDE:3D} in the following theorem.
	\begin{theorem}\label{thm:FDM:3D}
		Let $\kappa>0,\alpha,\beta,\gamma, \lambda, \phi$ be smooth in $\overline{\Omega}$ in \eqref{Model：Elliptic:3D}, 
		where $C_{r,\ell,s}$ with $r,\ell,s\in \{-1,0,1\}$ are defined in \eqref{C:1:Left:3D}--\eqref{r:Left:3D}, and $F_{i,j,k}$ is defined in \eqref{F:Right:3D}. 
		Then
		\[
		\mathcal{L}_h (u_h)_{i,j,k}-\mathcal{L}_h (u)_{i,j,k}=\bo(h^4),  \quad  (x_i,y_j,z_k)\in \Omega,
		\]	
		where $\mathcal{L}_h u_{i,j,k}$ and $\mathcal{L}_h (u_h)_{i,j,k}$ are defined in \eqref{Lh:u:3D} and \eqref{FDMs:3D:u}, respectively. I.e., $\mathcal{L}_h (u_h)_{i,j,k}$ in \eqref{FDMs:3D:u} achieves a  consistency order four for 3D convection-diffusion-reaction equations \eqref{Model：Elliptic:3D} and \eqref{simpli:PDE:3D} at $(x_i,y_j,z_k)$ in \eqref{xiyj:3D:space}.
	\end{theorem}
	\begin{proof}
		Similar to \eqref{taylor:basis:2D}--\eqref{L:h:u:3}, $\mathcal{L}_h (u)_{i,j,k}$ in \eqref{Lh:u:3D} satisfies 
		\[
		\mathcal{L}_h (u)_{i,j,k}=f\big|_{(x,y,z)=(x_i,y_j,z_k)}+\tfrac{h^2}{12}[af_x+bf_y+cf_z+ \Delta f]\big|_{(x,y,z)=(x_i,y_j,z_k)}+\bo(h^4).
		\]
		Then the proof is done by computing $\mathcal{L}_h (u_h)_{i,j,k}-\mathcal{L}_h (u)_{i,j,k}$ with $\mathcal{L}_h (u_h)_{i,j,k}$ defined in \eqref{FDMs:3D:u} and \eqref{C:1:Left:3D}--\eqref{F:Right:3D}.
	\end{proof}
	
	\begin{remark}\label{rem:only:19:point}
		We define 27 values of  $C_{r,\ell,s}$ with $r,\ell,s\in \{-1,0,1\}$ in \eqref{C:1:Left:3D}--\eqref{r:Left:3D} for the fourth-order compact FDM in \cref{thm:FDM:3D}, but we have
		\[
		C_{-1,-1,-1}= C_{-1,-1,1}=  C_{-1,1,-1}= C_{-1,1,1}=C_{1,-1,-1}=	C_{1,-1,1}=  C_{1,1,-1}=  C_{1,1,1}=0.
		\]
		Actually, we derive the fourth-order compact 19-point FDM in \cref{thm:FDM:3D}. 	Form \eqref{C:1:Left:3D}--\eqref{r:Left:3D}, we only need to compute $\{  \tfrac{\partial^{m+n+q} \varrho(x,y,z) }{ \partial x^m \partial y^n  \partial z^q} : m+n+q\le 2  \}$ for $\varrho=a,b,c,d,f$ defined in \eqref{abcd:3D} to achieve the consistency order four. Furthermore, \citep{Zhang1998} derived the fourth-order compact 19-point FDM for \eqref{simpli:PDE:3D} with $d=0$ employing a different approach from that used in this paper. 
	\end{remark}
	
	Similar to \cref{prop:1}, we have the following proposition:
	\begin{proposition}
		The 27 polynomials $\{ C_{r,\ell,s} \}_{r,\ell,s=-1,0,1}$ of the variable $h$ in \cref{thm:FDM:3D} satisfy that
		\[
		\begin{split} 
			& 	C_{\pm 1,\pm 1,\pm 1}= C_{\mp 1,\pm 1,\pm 1}= C_{\pm1,\mp 1,\pm 1} =C_{\pm1,\pm 1,\mp 1}=0,\\
			& C_{0,\pm 1,\pm 1}|_{h=0}=C_{0,\pm 1,\mp 1}|_{h=0}=C_{\pm 1,0,\pm 1}|_{h=0}=C_{\pm 1,0,\mp 1}|_{h=0}=C_{\pm 1,\pm 1,0}|_{h=0}=C_{\pm 1,\mp 1,0}|_{h=0}=\tfrac{1}{6}, \\
			&  C_{0,0,\pm 1}|_{h=0}=C_{0,\pm 1,0}|_{h=0}=C_{\pm 1,0,0}|_{h=0}=\tfrac{1}{3}, \qquad  C_{0,0,0}|_{h=0}=-4,\\
			& \sum_{r,\ell,s=-1}^1C_{r,\ell,s}=\sum_{q=0}^4 \varpi_qh^q, \quad \text{with} \quad 	\varpi_0=\varpi_1=0, \qquad \varpi_2=d=-\lambda  / \kappa,\\
			& \sum_{r,\ell,s=-1}^1C_{r,\ell,s}=0 \quad \text{if} \quad  \lambda=0.
		\end{split}
		\]
		Furthermore, $-\mathcal{L}_h (u_h)_{i,j,k} :=\frac{1}{h^2}\sum_{r,\ell,s=-1}^1 -C_{r,\ell,s}\big|_{(x,y,z)=(x_i,y_j,z_k)} (u_h)_{i+r,j+\ell,k+s}=-F_{i,j,k}\big|_{(x,y,z)=(x_i,y_j,z_k)}$ in \cref{thm:FDM:3D} is the   maximum principle preserving and monotone scheme, and forms an M-matrix for the sufficiently small $h$, if $\lambda \ge  0$ in $\Omega$ in \eqref{Model：Elliptic:3D}.
	\end{proposition}
	
	\subsection{Fourth-order  compact 19-point FDM for the linear time-dependent   equation in 3D}	\label{subse:liner:para:3D}	
	In this subsection, we discuss the following 3D	linear time-dependent  equation in the temporal domain $I:=[0,T]$ and the spatial domain $\Omega:=(l_1,l_2)^3$: 
	\be\label{Linear:Parabolic:3D}
	\begin{cases}
		u_t-\nab\cdot (\kappa \nab u) +   \alpha u_x +   \beta u_y +\gamma u_z +  \lambda u = \phi, \qquad \qquad  &  (x,y)\in \Omega \quad \hspace{1.8mm} \text{and} \quad  t\in I,\\
		u =g,  & (x,y)\in \partial \Omega \quad \text{and} \quad  t\in I,\\
		u=u^0,  & (x,y)\in  \Omega \quad \hspace{1.8mm}  \text{and} \quad t=0,
	\end{cases}
	\ee
	where $\kappa=\kappa(x,y,z,t)>0$, $\alpha= \alpha(x,y,z,t)$, $\beta= \beta(x,y,z,t)$, $\gamma= \gamma(x,y,z,t)$,  $\lambda= \lambda(x,y,z,t)$, and $\phi=\phi(x,y,z,t)$ are smooth variable functions in $\Omega$ and $I$. We use the uniform mesh for the time discretization in \eqref{tim:discre} and spatial discretization in \eqref{xiyj:3D:space}. For simplification, we define
	\be \label{nota:cni}
	c^{n+i}:=\frac{\kappa^{n+i}_z-\gamma^{n+i}}{\kappa^{n+i}}.
	\ee 
	
	Similar to \eqref{CN:eq}, \eqref{BDF3:eq}, and \eqref{BDF4:eq},  the CN, BDF3, and BDF4 methods yield
	\be\label{CN:eq:3D}
	\begin{split}
		& \Delta u^{n+i}  +a^{n+i} u^{n+i}_x +b^{n+i} u^{n+i}_y +c^{n+i} u^{n+i}_z +d^{n+i} u^{n+i}   = f^{n+i},\\
		& \text{CN} \quad i=1/2, \qquad \text{BDF3} \quad i=3,  \qquad \text{BDF4} \quad i=4, 
	\end{split}
	\ee
	where $a^{n+1/2}$, $b^{n+1/2}$, and  $c^{n+1/2}$  are defined in \eqref{ani:bni} and \eqref{nota:cni} with $i=1/2$, $d^{n+1/2}$ and $f^{n+1/2}$ are defined in \eqref{nota:dn:half}; $a^{n+3}$, $b^{n+3}$, and  $c^{n+3}$  are defined in \eqref{ani:bni} and \eqref{nota:cni} with $i=3$, $d^{n+3}$ and $f^{n+3}$ are defined in \eqref{nota:dn3}; $a^{n+4}$, $b^{n+4}$, and  $c^{n+4}$  are defined in \eqref{ani:bni} and \eqref{nota:cni} with $i=4$, $d^{n+4}$ and $f^{n+4}$ are defined in \eqref{nota:dn4} and \eqref{nota:BDF4:f}, respectively. Then we can use the compact 19-point FDM in \cref{thm:FDM:3D} to solve \eqref{CN:eq:3D}. 
	
	Similar to \cref{propos:tau:h}, we have the following proposition of the compact 19-point FDM in \cref{thm:FDM:3D} solving \eqref{CN:eq:3D}.
		\begin{proposition}
		The FDM $-\mathcal{L}_h (u_h)_{i,j,k}$ in \cref{thm:FDM:3D} to solve \eqref{CN:eq:3D} is the maximum principle preserving and monotone scheme, and forms an M-matrix for the sufficiently small $h$, if $\lambda\ge 0$ in $\Omega \cup I$ in \eqref{Linear:Parabolic:3D}. If $\tau$ and $h$ are both sufficiently small, then discrete maximum principle, monotone, and M-matrix properties can be guaranteed for any $\lambda$.
	\end{proposition}
	\begin{proof}
Similar to the proof of \cref{propos:tau:h}.
	\end{proof}

	Similar to \cref{remark:order:3}, the consistency order of the FDM in \cref{thm:FDM:3D} to solve \eqref{CN:eq:3D} is three, if
	$\tau=rh$ with the positive constant $r$. Similar to \cref{thm:FDM:parabo:2D}, we propose the algorithm to derive the fourth-order compact 19-point FDM to solve \eqref{CN:eq:3D}.
	
	{\bf{ The algorithm to derive the fourth-order FDM:}} Let $\tau:=rh$ with the positive constant $r$. 
	Then \eqref{CN:eq:3D} implies
	\be\label{para:simple:3D:1}
	\Delta \Qu  +  \Qa \Qu_x + \Qb  \Qu_y + \Qc  \Qu_z + \tfrac{1}{rh} \Qd_1 \Qu  +  \Qd_2 \Qu   =\Qf,
	\ee
	where
	\be 
	\begin{split}\label{para:simple:3D:cc}
		& \Qc=c^{n+i}, \\
		& i=1/2\text{ for the CN method in } \eqref{CN:eq:3D},  \\
		& i=3\text{ for the BDF3 method in  } \eqref{CN:eq:3D},  \\
		& i=4\text{ for the BDF4 method in  } \eqref{CN:eq:3D}, 
	\end{split}
	\ee
	$\Qf$ is defined in \eqref{tau:f:f1:f2}, $\Qu$,   $\Qa$,   $\Qb$, $\Qd_1$, $\Qd_2$, $ \Qf_1$, $ \Qf_2$ are defined in \eqref{d1:d2:f1:f2}.
	Then, we propose the fourth-order compact 19-point FDM to solve \eqref{para:simple:3D:1} in the
	following theorem.
	\begin{theorem} \label{thm:FDM:parabo:3D}
		Let $\kappa>0,u,\alpha,\beta,\gamma, \lambda, \phi$ be smooth  for $(x,y,z)\in \overline{\Omega}$ and $t\in I$ in \eqref{Linear:Parabolic:3D}, and define
		\be\label{Lh:uh:para:3D}
		\begin{split}
			\mathcal{L}_h (u_h)_{i,j,k} :=&\frac{1}{h^2}\sum_{r,\ell,s=-1}^1 C_{r,\ell,s}\Big|_{(x,y,z)
				=(x_i,y_j,z_k)}(u_h)_{i+r,j+\ell,k+s}\\
			=&\frac{1}{h^2}\sum_{r,\ell,s=-1}^1 \sum_{p=0}^5 c_{r,\ell,s,p} h^p\Big|_{(x,y,z)=(x_i,y_j,z_k)} (u_h)_{i+r,j+\ell,k+s}=F_{i,j,k}\Big|_{(x,y,z)=(x_i,y_j,z_k)},
		\end{split}
		\ee
		where
		\[
		\begin{split}
			F_{i,j,k}:=& \text{the terms of } \Big(h^{-2} \sum_{(m,n,q) \in \ind_{	 3}}  \frac{ \partial^{m+n+q}\textup \Qf(x,y,z) } {\partial x^m \partial y^n \partial z^q}  \\
			&\times \sum_{r,\ell,s=-1}^{1} \sum_{p=0}^5 c_{r,\ell,s,p} h^p\Big|_{(x,y,z)=(x_i,y_j,z_k)} H_{5,m,n,q}  (rh, \ell h, sh)\Big) \text{ with degree}   \le 3 \text{ in } h,
		\end{split}
		\]
		and $c_{r,\ell,s,p}$ can be uniquely determined by solving
		\[
		I_{5,m,n,q}=\sum_{r,\ell,s=-1}^{1} G_{5,m,n,q}  (rh, \ell h, sh) \sum_{p=0}^5 c_{r,\ell,s,p} \big|_{(x,y,z)=(x_i,y_j,z_k)} h^p =\bo(h^{6}) \quad \mbox{for all}  \quad (m,n,q)\in \ind_{5}^1, 
		\]
		with
		\[
		\text{the lowest degree term of } h \ (\text{the }h^0 \text{ term})  \text{ in }  F_{i,j,k}=  \textup\Qf,
		\]
		and
		\[
		\begin{split}
			\{ & c_{-1, 0, -1, 5}, c_{-1, 0, 0, 5}, c_{-1, 0, 1, p_1}, c_{-1, 1, 0, p_1}, c_{0, -1, -1, 5}, c_{0, -1, 0, p_1},  \\      
			& c_{0, -1, 1, p_2}, c_{0, 0, -1, p_1}, c_{0, 0, 0, p_1}, c_{0, 0, 1, p_2}, c_{0, 1, -1, p_2}, c_{0, 1, 0, p_2},\\ 
			&   c_{0, 1, 1, p_3},   c_{1, -1, 0, p_2}, c_{1, 0, -1, p_2}, c_{1, 0, 0, p_3}, c_{1, 0, 1, p_3}, c_{1, 1, 0, p_4}\}=\{0\},
		\end{split}
		\]
		for	$p_1=4,5$, $p_2=3,4,5$, $p_3=2,\dots,5$, $p_4=1,\dots,5$, and
		\[
		C_{\pm 1,\pm 1,\pm 1}= C_{\mp 1,\pm 1,\pm 1}= C_{\pm1,\mp 1,\pm 1} =C_{\pm1,\pm 1,\mp 1}=0 \quad \text{and} \quad c_{0,0,0,0}=-4,
		\]
		$\ind_{5}^1$ is defined in \eqref{ind:3D},  $G_{5,m,n,q}(x,y,z)$ and $H_{5,m,n,q}(x,y,z)$ can be uniquely determined by applying \citep[eqs. (3)--(36)]{FengTrenchea2026} with replacing \eqref{conve:diff:eq} by \eqref{para:simple:3D:1} and considering $rh$ as the constant. Then \eqref{Lh:uh:para:3D} achieves a consistency order four for the equation \eqref{para:simple:3D:1} at $(x_i,y_j,z_k)$ in \eqref{xiyj:3D:space}.
	\end{theorem}
	
	\begin{remark}
		Similar to \cref{remark:thm22}, we do not provide explicit expressions of $C_{r,\ell,s}$, $F_{i,j,k}$, $G_{5,m,n,q}$, and $H_{5,m,n,q}$ for the FDM \eqref{Lh:uh:para:3D} in \cref{thm:FDM:parabo:3D}.
		Although  \cref{thm:FDM:3D} has the lower order of consistency  than \cref{thm:FDM:parabo:3D} for the time-dependent equation, errors from \cref{thm:FDM:3D} are smaller than those from \cref{thm:FDM:parabo:3D} in \cref{Example:8} (see \cref{Example:8:table} with $h\ge 1/2^5$). From our numerical experience, the reason may be that the higher-order FDM in \cref{thm:FDM:parabo:3D} requires the complicated expression, and the simple stencil in \cref{thm:FDM:3D} has the smaller pollution effect when the mesh size $h$ is not sufficiently small. So, for 3D linear and nonlinear, time-independent and time-dependent equations, the FDM with the simple stencil in \cref{thm:FDM:3D} is the main result in this paper. We also aim to add new restrictions like \eqref{c:rlsp:pro:1}--\eqref{c:rlsp:pro:3} in our algorithm to propose the fourth-order compact 19-point FDM with the simple stencil to solve \eqref{para:simple:3D:1}. 
	\end{remark}
	
	Although $C_{r,\ell,s}$ in \eqref{Lh:uh:para:3D} is not presented explicitly, we can still provide the explicit expression of $C_{r,\ell,s}$ with $h=0$ and the leading term of  $ \sum_{r,\ell,s=-1}^1C_{r,\ell,s}$ to demonstrate that the FDM in \eqref{Lh:uh:para:3D} can form an M-matrix if $h$ is sufficiently small in the following proposition.
	\begin{proposition}
		The 27 polynomials $\{ C_{r,\ell,s} \}_{r,\ell,s=-1,0,1}$ of the variable $h$ in  \eqref{Lh:uh:para:3D} satisfy that		
		\[
		\begin{split} 
			& 	C_{\pm 1,\pm 1,\pm 1}= C_{\mp 1,\pm 1,\pm 1}= C_{\pm1,\mp 1,\pm 1} =C_{\pm1,\pm 1,\mp 1}=0,\\
			& C_{0,\pm 1,\pm 1}|_{h=0}=C_{0,\pm 1,\mp 1}|_{h=0}=C_{\pm 1,0,\pm 1}|_{h=0}=C_{\pm 1,0,\mp 1}|_{h=0}=C_{\pm 1,\pm 1,0}|_{h=0}=C_{\pm 1,\mp 1,0}|_{h=0}=\tfrac{1}{6}, \\
			&  C_{0,0,\pm 1}|_{h=0}=C_{0,\pm 1,0}|_{h=0}=C_{\pm 1,0,0}|_{h=0}=\tfrac{1}{3}, \qquad  C_{0,0,0}|_{h=0}=-4,\\
			& \sum_{r,\ell,s=-1}^1C_{r,\ell,s}=\sum_{q=0}^5 \varpi_qh^q, \quad \text{with} \quad 	\varpi_0=0 \quad \text{and} \quad \varpi_1=\textup \Qd_1/r,
		\end{split}
		\]
		where  $r=\tau/h>0$ ($\tau$ is the time step)
		\begin{align*}
			& \textup \Qd_1=\frac{-2}{\kappa^{n+1/2}},   && \hspace{-4cm} \text{for the CN method } \eqref{CN:eq:3D} \text{ with } i=1/2;\\
			& \textup \Qd_1=\frac{-11}{6 \kappa^{n+3}},  && \hspace{-4cm} \text{for the BDF3 method } \eqref{CN:eq:3D}  \text{ with } i=3;\\
			& \textup \Qd_1=\frac{-25}{12\kappa^{n+4}},  && \hspace{-4cm}  \text{for the BDF4 method } \eqref{CN:eq:3D}  \text{ with } i=4.
		\end{align*}
		Furthermore, $-\mathcal{L}_h (u_h)_{i,j,k} :=\frac{1}{h^2}\sum_{r,\ell,s=-1}^1 -C_{r,\ell,s}\big|_{(x,y,z)=(x_i,y_j,z_k)} (u_h)_{i+r,j+\ell,k+s}=-F_{i,j,k}\big|_{(x,y,z)=(x_i,y_j,z_k)}$ in \eqref{Lh:uh:para:3D} is the   maximum principle preserving and monotone scheme, and forms an M-matrix for any $\lambda$ in \eqref{Linear:Parabolic:3D}, if $h$ is sufficiently small.
	\end{proposition}
	\begin{proof}
		Similar to the proof of \cref{prop:M:matrix:para}
	\end{proof}		
	The symbolic computation also shows that we need to calculate $ \{  \tfrac{\partial^{m+n+q} \varrho(x,y,z,t_{n+i}) }{ \partial x^m \partial y^n  \partial z^q } : m+n+q\le 3 \}$ with $\varrho=\Qa,\Qb,\Qc,\Qd_1, \Qd_2, \Qf$ defined in \eqref{tau:f:f1:f2}, \eqref{d1:d2:f1:f2}, and \eqref{para:simple:3D:cc}  in \cref{thm:FDM:parabo:3D}.
	
	\subsection{Fourth-order  compact 19-point FDM for the nonlinear time-independent  equation  in 3D}\label{subse:non:liner:elli:3D}
	In this subsection, we consider the 3D version of \eqref{Non:Linear:Elliptic:2D} as follows:
	\be\label{Non:Linear:Elliptic:3D}
	\begin{cases}
		-\nab\cdot (\kappa(u) \nab u) +   \alpha(u) u_x +   \beta(u) u_y +   \gamma(u) u_z +   \lambda(u) u = \phi, \qquad \qquad  &  (x,y)\in \Omega,\\
		u =g,  & (x,y)\in \partial \Omega,
	\end{cases}
	\ee
	where $\kappa>0$, $\alpha$, $\beta$,  $\gamma$, and $\lambda$ only depend on $u(x,y,z)$ and are smooth in $\Omega$.
	Similar to \eqref{nonlinear:2D:ellip:ite}, we have
	\be\label{simplied:eq:non:elli:3D}
	\begin{split}
		& \Delta u_{\Qk+1}  + a_\Qk (u_{\Qk+1})_x +  b_\Qk(u_{\Qk+1})_y+  c_\Qk(u_{\Qk+1})_z  +d_\Qk u_{\Qk+1} =f_\Qk,\\
		& \Qk:=\Qk+1, \quad  \text{the initial guess } u_\Qk=0 \text{ for } \Qk=0,
	\end{split}
	\ee
	where $a_\Qk, b_\Qk, d_\Qk, f_\Qk$ are defined in \eqref{ak:bk:dk:fk}, and 
	\be\label{cQk:nota}
	c_\Qk:=\frac{\kappa_u(u_\Qk)(u_\Qk)_z-\gamma(u_\Qk)}{\kappa(u_\Qk)}.
	\ee
	Then we use the fourth-order compact 19-point FDM in \cref{thm:FDM:3D} to solve \eqref{simplied:eq:non:elli:3D} until $\Qk=40$. 
			
	We also aim to extend the above high-order FDM to solve other important nonlinear steady equations like the $p$-Laplace equation in \citep{Dogbatsey2026}.

	\subsection{Fourth-order  compact 19-point FDM for the nonlinear time-dependent   equation in 3D}	\label{subse:non:liner:para:3D}		
	
	In this subsection, we consider the nonlinear version of \eqref{Linear:Parabolic:3D} as follows
	\be\label{No:Linear:Parabolic:3D} 
	u_t-\nab\cdot (\kappa(u) \nab u) +   \alpha(u) u_x +   \beta(u) u_y + \gamma(u) u_z +  \lambda (u) u = \phi.
	\ee
	For simplification,  we assume that $\kappa>0$, $\alpha$, $\beta$,  $\gamma$, and $\lambda$ only depend on $u(x,y,z,t)$ and are smooth in $\Omega$ and $I$, and define
	\be\label{ckni:3D}
	c_\Qk^{n+i}:=\frac{\kappa_u(u^{n+i}_\Qk)(u^{n+i}_\Qk)_z-\gamma(u^{n+i}_\Qk)}{\kappa(u^{n+i}_\Qk)}.
	\ee
	Similar to \eqref{CN:2D:itera}, \eqref{BDF3:2D:itera}, and \eqref{BDF4:2D:itera},  the CN, BDF3, and BDF4 methods yield
	\be\label{simplied:eq:non:para:3D}
	\begin{split} 
		& \Delta u_{\Qk+1}^{n+i}  +a_\Qk^{n+i} (u^{n+i}_{\Qk+1})_x +b_\Qk^{n+i}(u^{n+i}_{\Qk+1})_y +c_\Qk^{n+i} (u^{n+i}_{\Qk+1})_z +d_\Qk^{n+i} u^{n+i}_{\Qk+1}   = f^{n+i}_\Qk,\\
		& \text{CN} \quad i=1/2, \qquad \text{BDF3} \quad i=3,  \qquad \text{BDF4} \quad i=4, 
	\end{split} 
	\ee
	where $a_\Qk^{n+1/2}$, $b_\Qk^{n+1/2}$, and $c_\Qk^{n+1/2}$  are defined in \eqref{akni} and \eqref{ckni:3D} with $i=1/2$, $d_\Qk^{n+1/2}$ and $f^{n+1/2}_\Qk$ are defined in \eqref{dk:CN:2D} and  \eqref{fk:CN:2D}, respectively; $a_\Qk^{n+3}$, $b_\Qk^{n+3}$, and $c_\Qk^{n+3}$  are defined in \eqref{akni} and \eqref{ckni:3D} with $i=3$,  $d_\Qk^{n+3}$ and $f^{n+3}_\Qk$ are defined in \eqref{dk:BDF3:2D} and  \eqref{fk:BDF3:2D}, respectively; $a_\Qk^{n+4}$, $b_\Qk^{n+4}$, and $c_\Qk^{n+4}$  are defined in \eqref{akni} and \eqref{ckni:3D} with $i=4$,  $d_\Qk^{n+4}$ and $f^{n+4}_\Qk$ are defined in \eqref{dk:BDF4:2D} and  \eqref{fk:BDF4:2D}, respectively.
	Then we use the compact 19-point FDMs in \cref{thm:FDM:3D}  and \cref{thm:FDM:parabo:3D} to solve \eqref{simplied:eq:non:para:3D} until $\Qk=40$, where the initial guess $u^{n+i}_{\Qk}$ ($\Qk=0$), initial and boundary conditions, $u^1,u^2$ for the BDF3 method,  $u^1,u^2,u^3$ for the BDF4 method use the same settings  in  \eqref{CN:2D:itera}, \eqref{BDF3:2D:itera}, and \eqref{BDF4:2D:itera}.
			
	In future, we plan to derive high-order FDMs to solve other crucial nonlinear time-dependent equations like the model of two-phase flow in porous media in \citep{Jones2024}, the Cahn-Hilliard equation in \citep{LiQiao2016}, the Navier-Stokes equation in vorticity formulation in \citep{WangLiu2002}, and the time-fractional nonlinear diffusion-reaction equation in \citep{Zaky2023}.

	\subsection{Approximate (high-order) partial derivatives used in Sections \ref{subse:liner:elli:2D}-\ref{subse:non:liner:para:3D}}\label{app:ux}
	For the FDM in \cref{thm:FDM:2D}, we need to calculate $ \{  \tfrac{\partial^{m+n} \varrho(x,y) }{ \partial x^m \partial y^n },  \tfrac{\partial^{m+n} \varrho(x,y,t_{n+i}) }{ \partial x^m \partial y^n }  : m+n\le 2  \}$ with $\varrho=a,b,d,f$ in \eqref{notation:abdf} and $\varrho=a^{n+i},b^{n+i},d^{n+i},f^{n+i}$ in \eqref{ani:bni}--\eqref{BDF4:eq}. For the FDM in \cref{thm:FDM:parabo:2D},  we need to calculate $ \{  \tfrac{\partial^{m+n} \varrho(x,y,t_{n+i}) }{ \partial x^m \partial y^n } : m+n\le 3  \}$ with $\varrho=\Qa,\Qb,\Qd_1,\Qd_2, \Qf$ in \eqref{tau:f:f1:f2} and \eqref{d1:d2:f1:f2}. For the FDM in \cref{thm:FDM:3D}, we need to  compute $\{  \tfrac{\partial^{m+n+q} \varrho(x,y,z) }{ \partial x^m \partial y^n  \partial z^q}, \tfrac{\partial^{m+n+q} \varrho(x,y,z,t_{n+i}) }{ \partial x^m \partial y^n  \partial z^q} : m+n+q\le 2  \}$ for $\varrho=a,b,c,d,f$ in \eqref{abcd:3D}  and $\varrho=a^{n+i},b^{n+i},c^{n+i},d^{n+i},f^{n+i}$ in \eqref{CN:eq:3D}. For the FDM in \cref{thm:FDM:parabo:3D}, we need to calculate $ \{  \tfrac{\partial^{m+n+q} \varrho(x,y,z,t_{n+i}) }{ \partial x^m \partial y^n  \partial z^q } : m+n+q\le 3 \}$ with $\varrho=\Qa,\Qb,\Qc,\Qd_1, \Qd_2, \Qf$ in \eqref{tau:f:f1:f2}, \eqref{d1:d2:f1:f2}  and \eqref{para:simple:3D:cc}. For nonlinear equations, we need to compute $(u_\Qk)_x$ and $(u_\Qk)_y$ in \eqref{ak:bk:dk:fk}, $(u^{n+i}_\Qk)_x$ and $(u^{n+i}_\Qk)_y$ in \eqref{akni}, $(u_\Qk)_z$ in  \eqref{cQk:nota}, and $(u^{n+i}_\Qk)_z$  in \eqref{ckni:3D}. If symbolic expressions of above (high-order) partial derivatives are available, we use their expressions directly. Otherwise, we use following finite difference (FD) operators to approximate required (high-order) partial derivatives.
	
	For the 1D function $\varrho(x)$ and $x_i$ defined in \eqref{xiyj:2D:space}, we have following FD operators to approximate first-order to third-order derivatives of $\varrho(x)$ for all $x_i\in [l_1,l_2]$:\\
	{\bf{ The first-order derivatives with the second-order accuracy (3-point FD operators):}}
	\be \label{ux:order:2}
	\begin{split}
		\varrho_{x}(x_i)=&\tfrac{1}{h}\big[ -\tfrac{1}{2} \varrho(x_{i-1})+0\times \varrho(x_{i}) +\tfrac{1}{2}  \varrho(x_{i+1})\big]+\bo(h^{2}),  \text{ if} \quad  l_1 \le x_{i \pm 1} \le l_2,\\
		\varrho_{x}(x_i)=&\tfrac{\pm 1}{h}\big[  -\tfrac{1}{2} \varrho(x_{i\pm 2})+  2 \varrho(x_{i\pm 1})-\tfrac{3}{2}\varrho(x_{i})\big]+\bo(h^{2}),  \text{ if} \quad  l_1 \le x_{i\pm 2}, x_{i} \le l_2.
	\end{split}
	\ee
	{\bf{ The first-order derivatives with the third-order accuracy (4-point FD operators):}}
	\be \label{ux:order:3}
	\begin{split}
		\varrho_{x}(x_i)=&\tfrac{1}{h}\big[ -\tfrac{1}{3} \varrho(x_{i-1}) - \tfrac{1}{2} \varrho(x_{i}) +\varrho(x_{i+1})-\tfrac{1}{6}  \varrho(x_{i+2})\big]+\bo(h^{3}),  \text{ if} \quad  l_1 \le x_{i-1}, x_{i+2} \le l_2,\\
		\varrho_{x}(x_i)=&\tfrac{1}{h}\big[ \tfrac{1}{6} \varrho(x_{i-2}) -\varrho(x_{i-1}) + \tfrac{1}{2} \varrho(x_{i}) +\tfrac{1}{3}  \varrho(x_{i+1})\big]+\bo(h^{3}),  \text{ if} \quad  l_1 \le x_{i-2}, x_{i+1} \le l_2,\\
		\varrho_{x}(x_i)=&\tfrac{\pm 1}{h}\big[  \tfrac{1}{3}\varrho(x_{i\pm 3}) -\tfrac{3}{2} \varrho(x_{i\pm 2})+ 3 \varrho(x_{i\pm 1})  -\tfrac{11}{6}  \varrho(x_{i})\big]+\bo(h^{3}),  \text{ if} \quad  l_1 \le x_{i\pm 3}, x_{i} \le l_2.
	\end{split}
	\ee
	{\bf{ The first-order derivatives with the fourth-order accuracy (5-point FD operators):}}
	\be \label{ux:order:4}
	\begin{split}
		\varrho_{x}(x_i)=&\tfrac{1}{h}\big[\tfrac{1}{12} \varrho(x_{i-2})-  \tfrac{2}{3} \varrho(x_{i-1})+0\times \varrho(x_{i})+  \tfrac{2}{3} \varrho(x_{i+1})-\tfrac{1}{12}  \varrho(x_{i+2})\big]+\bo(h^{4}),  \\
		& \text{ if} \quad  l_1 \le x_{i\pm 2} \le l_2,\\
		\varrho_{x}(x_i)=&\tfrac{\pm 1}{h}\big[   -\tfrac{1}{4}\varrho(x_{i\pm 4})  +  \tfrac{4}{3} \varrho(x_{i\pm 3}) -3 \varrho(x_{i\pm 2}) +4  \varrho(x_{i\pm 1})   -\tfrac{25}{12}\varrho(x_{i})  \big]+\bo(h^{4}), \\
		&  \text{ if} \quad  l_1 \le x_{i\pm 4}, x_{i} \le l_2,\\		
		\varrho_{x}(x_i)=&\tfrac{\pm 1}{h}\big[   \tfrac{1}{12}  \varrho(x_{i \pm 3}) -\tfrac{1}{2}  \varrho(x_{i\pm 2})  + \tfrac{3}{2} \varrho(x_{i\pm 1}) -\tfrac{5}{6}\varrho(x_{i})-\tfrac{1}{4} \varrho(x_{i\mp 1}) \big]+\bo(h^{4}), \\
		&  \text{ if} \quad  l_1 \le  x_{i\pm 3}, x_{i\mp 1} \le l_2.
	\end{split}
	\ee
	{\bf{ The first-order derivatives with the fifth-order accuracy (6-point FD operators):}}
	\be \label{ux:order:5}
	\begin{split}
		\varrho_{x}(x_i)=&\tfrac{1}{h}\big[\tfrac{1}{20} \varrho(x_{i-2})-  \tfrac{1}{2} \varrho(x_{i-1})-\tfrac{1}{3}\varrho(x_{i})+  \varrho(x_{i+1})-\tfrac{1}{4}  \varrho(x_{i+2})+\tfrac{1}{30}  \varrho(x_{i+3})\big]+\bo(h^{5}), \\
		&  \text{ if} \quad  l_1 \le x_{i- 2}, x_{i+3} \le l_2,\\
		\varrho_{x}(x_i)=&\tfrac{1}{h}\big[-\tfrac{1}{30} \varrho(x_{i-3})+  \tfrac{1}{4} \varrho(x_{i-2})-\varrho(x_{i-1})+  \tfrac{1}{3} \varrho(x_{i})+\tfrac{1}{2}  \varrho(x_{i+1})-\tfrac{1}{20}  \varrho(x_{i+2})\big]+\bo(h^{5}), \\
		&  \text{ if} \quad  l_1 \le x_{i-3}, x_{i+2} \le l_2,\\
		\varrho_{x}(x_i)=&\tfrac{\pm 1}{h}\big[  \tfrac{1}{5} \varrho(x_{i\pm5}) -\tfrac{5}{4} \varrho(x_{i\pm4}) +\tfrac{10}{3}\varrho(x_{i\pm3})-5\varrho(x_{i\pm2})   	+5  \varrho(x_{i\pm1})  -\tfrac{137}{60}  \varrho(x_{i}) \big]+\bo(h^{5}), \\
		& \text{ if} \quad l_1 \le x_{i\pm5}, x_{i}\le l_2, \\
		\varrho_{x}(x_i)=& \tfrac{\pm1}{h}\big[ -\tfrac{1}{20} \varrho(x_{i\pm4})+ \tfrac{1}{3} \varrho(x_{i\pm3})  -\varrho(x_{i\pm2})+  2 \varrho(x_{i\pm1}) - \tfrac{13}{12} \varrho(x_{i})-\tfrac{1}{5} \varrho(x_{i\mp1})\big]+\bo(h^{5}), \\
		& \text{ if} \quad l_1 \le x_{i\pm4}, x_{i\mp1}\le l_2.
	\end{split}
	\ee
	{\bf{ The second-order derivatives  with the second-order accuracy (3-point to 4-point FD operators):}}
	\be \label{uxx:order:2}
	\begin{split}
		\varrho_{xx}(x_i)=&\tfrac{1}{h^2}\big[ \varrho(x_{i-1})-  2 \varrho(x_{i})+ \varrho(x_{i+1})\big]+\bo(h^{2}),   \text{ if} \quad  l_1 \le x_{i\pm 1} \le l_2,\\
		\varrho_{xx}(x_i)=&\tfrac{1}{h^2}\big[ -\varrho(x_{i\pm3})+  4 \varrho(x_{i\pm2}) -  5  \varrho(x_{i\pm1})+2\varrho(x_{i}) \big]+\bo(h^{2}),  \text{ if} \quad l_1 \le x_{i\pm3}, x_{i}\le l_2.
	\end{split}
	\ee
	{\bf{ The second-order derivatives  with the  fourth-order to third-order accuracy  (5-point FD operators):}}
	\be \label{uxx:order:43}
	\begin{split}
		\varrho_{xx}(x_i)=&\tfrac{1}{h^2}\big[-\tfrac{1}{12} \varrho(x_{i-2})+  \tfrac{4}{3} \varrho(x_{i-1})-  \tfrac{5}{2} \varrho(x_{i})+\tfrac{4}{3}  \varrho(x_{i+1})-\tfrac{1}{12}  \varrho(x_{i+2})\big]+\bo(h^{4}), \\
		&  \text{ if} \quad  l_1 \le x_{i\pm 2} \le l_2,\\
		\varrho_{xx}(x_i)=&\tfrac{1}{h^2}\big[ \tfrac{11}{12} \varrho(x_{i\pm4})  -\tfrac{14}{3} \varrho(x_{i\pm3})+  \tfrac{19}{2} \varrho(x_{i\pm2})  -\tfrac{26}{3} \varrho(x_{i\pm1}) +\tfrac{35}{12} \varrho(x_{i}) \big]+\bo(h^{3}), \\
		& \text{ if} \quad l_1 \le x_{i\pm4}, x_{i}\le l_2, \\
		\varrho_{xx}(x_i)=&\tfrac{1}{h^2}\big[-\tfrac{1}{12}  \varrho(x_{i\pm3})  +\tfrac{1}{3} \varrho(x_{i\pm2})+ \tfrac{1}{2} \varrho(x_{i\pm1})   -\tfrac{5}{3}   \varrho(x_{i}) +\tfrac{11}{12}      \varrho(x_{i\mp1})\big]+\bo(h^{3}), \\
		& \text{ if} \quad l_1 \le x_{i\pm3}, x_{i\mp1}\le l_2.
	\end{split}
	\ee
	{\bf{ The second-order derivatives  with the fourth-order accuracy  (5-point to 6-point FD operators):}}
	\begin{align}\label{uxx:order:4}
		\varrho_{xx}(x_i)=&\tfrac{1}{h^2}\big[-\tfrac{1}{12} \varrho(x_{i-2})+  \tfrac{4}{3} \varrho(x_{i-1})-  \tfrac{5}{2} \varrho(x_{i})+\tfrac{4}{3}  \varrho(x_{i+1})-\tfrac{1}{12}  \varrho(x_{i+2})\big]+\bo(h^{4}), \notag \\
		&  \text{ if} \quad  l_1 \le x_{i\pm 2} \le l_2, \notag \\
		\varrho_{xx}(x_i)=&\tfrac{1}{h^2}\big[ - \tfrac{5}{6}\varrho(x_{i\pm5})+\tfrac{61}{12}\varrho(x_{i\pm4})-13 \varrho(x_{i\pm3})+  \tfrac{107}{6} \varrho(x_{i\pm2})   	-  \tfrac{77}{6}   \varrho(x_{i\pm1})+\tfrac{15}{4}\varrho(x_{i}) \big]+\bo(h^{4}), \notag \\
		& \text{ if} \quad l_1 \le x_{i\pm5}, x_{i}\le l_2, \\
		\varrho_{xx}(x_i)=&\tfrac{1}{h^2}\big[\tfrac{1}{12} \varrho(x_{i\pm4})-\tfrac{1}{2} \varrho(x_{i\pm3})  +\tfrac{7}{6}\varrho(x_{i\pm2})- \tfrac{1}{3} \varrho(x_{i\pm1})-  \tfrac{5}{4} \varrho(x_{i})+\tfrac{5}{6}   \varrho(x_{i\mp1})\big]+\bo(h^{4}), \notag \\
		& \text{ if} \quad l_1 \le x_{i\pm4}, x_{i\mp1}\le l_2. \notag
	\end{align}
	{\bf{ The third-order derivatives  with the first-order accuracy  (4-point FD operators):}}
	\be\label{uxxx:order:1}
	\begin{split}
		\varrho_{xxx}(x_i)=&\tfrac{1}{h^3}\big[ -\varrho(x_{i-1})+3\varrho(x_{i})-3 \varrho(x_{i+1})+   \varrho(x_{i+2})\big]+\bo(h),   \text{ if} \quad  l_1 \le x_{i- 1},  x_{i+2}\le l_2,\\
		\varrho_{xxx}(x_i)=&\tfrac{1}{h^3}\big[ -\varrho(x_{i-2})+3\varrho(x_{i-1})-3 \varrho(x_{i})+   \varrho(x_{i+1})\big]+\bo(h),   \text{ if} \quad  l_1 \le x_{i- 2},  x_{i+1}\le l_2,\\	
		\varrho_{xxx}(x_i)=&\tfrac{\pm 1}{h^3}\big[ \varrho(x_{i \pm 3})-3\varrho(x_{i \pm 2})+3 \varrho(x_{i \pm 1})-   \varrho(x_{i})\big]+\bo(h),   \text{ if} \quad  l_1 \le x_{i \pm 3},x_i\le l_2.
	\end{split}
	\ee
	{\bf{ The third-order derivatives  with the second-order accuracy  (5-point FD operators):}}
	\be \label{uxxx:order:2} 
	\begin{split}
		\varrho_{xxx}(x_i)=&\tfrac{1}{h^3}\big[ -\tfrac{1}{2}\varrho(x_{i-2})+\varrho(x_{i-1})+0\times \varrho(x_{i})- \varrho(x_{i+1})+   \tfrac{1}{2}\varrho(x_{i+2})\big]+\bo(h^2),  \\
		& \text{ if} \quad  l_1 \le  x_{i\pm 2}\le l_2,\\
		\varrho_{xxx}(x_i)=&\tfrac{\pm 1}{h^3}\big[-\tfrac{3}{2} \varrho(x_{i \pm 4})+7\varrho(x_{i \pm 3})-12\varrho(x_{i \pm 2})+9 \varrho(x_{i \pm 1}) -\tfrac{5}{2} \varrho(x_{i})\big]+\bo(h^2),   \\
		&\text{ if} \quad  l_1 \le x_{i \pm 4},x_i\le l_2,\\
		\varrho_{xxx}(x_i)=&\tfrac{\pm 1}{h^3}\big[ -\tfrac{1}{2}\varrho(x_{i \pm 3})+3\varrho(x_{i \pm 2})-6 \varrho(x_{i \pm 1}) +5 \varrho(x_{i}) -\tfrac{3}{2} \varrho(x_{i \mp 1})\big]+\bo(h^2),  \\
		& \text{ if} \quad  l_1 \le x_{i \pm 3}, x_{i \mp 1}\le l_2.	
	\end{split}
	\ee
	{\bf{ The third-order derivatives  with the third-order accuracy  (6-point FD operators):}}
	\be \label{uxxx:order:3} 
	\begin{split}
		\varrho_{xxx}(x_i)=&\tfrac{1}{h^3}\big[-\tfrac{1}{4} \varrho(x_{i-2})-\tfrac{1}{4}  \varrho(x_{i-1})+\tfrac{5}{2} \varrho(x_{i})-  \tfrac{7}{2} \varrho(x_{i+1})+\tfrac{7}{4} \varrho(x_{i+2}) -  \tfrac{1}{4}\varrho(x_{i+3})\big]+\bo(h^{3}), \\
		&  \text{ if} \quad  l_1 \le x_{i- 2}, x_{i+3} \le l_2,\\
		\varrho_{xxx}(x_i)=&\tfrac{1}{h^3}\big[\tfrac{1}{4} \varrho(x_{i-3})-  \tfrac{7}{4} \varrho(x_{i-2})+\tfrac{7}{2}\varrho(x_{i-1})-  \tfrac{5}{2} \varrho(x_{i})+\tfrac{1}{4}  \varrho(x_{i+1})+\tfrac{1}{4}  \varrho(x_{i+2})\big]+\bo(h^{3}), \\
		&  \text{ if} \quad  l_1 \le x_{i-3}, x_{i+2} \le l_2,\\
		\varrho_{xxx}(x_i)=&\tfrac{\pm 1}{h^3}\big[ \tfrac{7}{4} \varrho(x_{i\pm5}) -\tfrac{41}{4} \varrho(x_{i\pm4})+ \tfrac{49}{2} \varrho(x_{i\pm3})-\tfrac{59}{2} \varrho(x_{i\pm2})   	+\tfrac{71}{4} \varrho(x_{i\pm1})   -\tfrac{17}{4}  \varrho(x_{i}) \big]+\bo(h^{3}), \\
		& \text{ if} \quad l_1 \le x_{i\pm5}, x_{i}\le l_2, \\
		\varrho_{xxx}(x_i)=&\tfrac{\pm 1}{h^3}\big[ \tfrac{1}{4} \varrho(x_{i\pm4}) -\tfrac{7}{4} \varrho(x_{i\pm3})+\tfrac{11}{2}\varrho(x_{i\pm2})-\tfrac{17}{2} \varrho(x_{i\pm1})+\tfrac{25}{4} \varrho(x_{i}) -\tfrac{7}{4} \varrho(x_{i\mp1})\big]+\bo(h^{3}), \\
		& \text{ if} \quad l_1 \le x_{i\pm4}, x_{i\mp1}\le l_2.
	\end{split}
	\ee
	For 2D function $\varrho(x,y)$, we can use  above FD operators to calculate $\varrho_{x}$, $\varrho_{xx}$, $\varrho_{xxx}$, $\varrho_{y}$, $\varrho_{yy}$, $\varrho_{yyy}$ at the grid point $(x_i,y_j)$ defined in  \eqref{xiyj:2D:space}. Then we calculate mixed high-order partial derivatives as follows:
	\[
	\varrho_{xy}=(\varrho_x)_y, \qquad \varrho_{xxy}=(\varrho_{xx})_y, \qquad  \varrho_{xyy}=(\varrho_{yy})_x.
	\]
	We can also calculate first-order to third-order partial derivatives
	for 3D function $\varrho(x,y,z)$ similarly. Note that the moving least-squares method in \citep{DLevin1998} can also generate high-order partial derivatives with the high-accuracy.
	
	In \cref{sec:Numeri}, we test performances of FD operators \eqref{ux:order:2}--\eqref{uxxx:order:3} for  proposed compact  FDMs in \cref{thm:FDM:2D,thm:FDM:parabo:2D,thm:FDM:3D,thm:FDM:parabo:3D}. 
	The numerical results show that  \eqref{ux:order:5}, \eqref{uxx:order:4}, and \eqref{uxxx:order:3} yield smallest errors (see \cref{Example:4:table} in \cref{Example:4}), while these three FD operators require 6-point stencils. 
	We also observe that 3-point to 4-point FD operators \eqref{ux:order:2}--\eqref{ux:order:3} and  \eqref{uxx:order:2} are sufficient to produce accurate solutions for most examples because of the robustness and stability of proposed FDMs  (see \cref{Example:2:table,Example:3:table} for \cref{Example:2,Example:3} and \cref{Example:6:table,Example:7:table} for \cref{Example:6,Example:7}). Furthermore, considering the accuracy and efficiency, 4-point to 5-point FD operators \eqref{ux:order:3}--\eqref{ux:order:4}, \eqref{uxx:order:43}, and \eqref{uxxx:order:2} are  optimal choices for FDMs in \cref{thm:FDM:2D,thm:FDM:parabo:2D,thm:FDM:3D,thm:FDM:parabo:3D} (see \cref{Example:3:table,Example:4:table} in \cref{Example:3,Example:4} and \cref{Example:6:table,Example:7:table,Example:8:table} in  \cref{Example:6,Example:7,Example:8}).
	
	\section{Numerical experiments}\label{sec:Numeri}
	We use the uniform Cartesian mesh \eqref{xiyj:2D:space} for the 2D spatial  domain $\Omega=(l_1,l_2)^2$, \eqref{xiyj:3D:space} for   the 3D spatial  domain $\Omega=(l_1,l_2)^3$, and \eqref{tim:discre} for  the temporal domain $I=[0,T]$. To verify the performance of FDMs in  \cref{thm:FDM:2D,thm:FDM:parabo:2D,thm:FDM:3D,thm:FDM:parabo:3D}, we provide the $l_{\infty}$  norm of errors for linear and nonlinear, time-independent and time-dependent convection-diffusion-reaction equations in 2D and 3D as follows:
	\[
	\begin{split}
		& \|u_h-u\|_\infty
		:=\max_{0\le i,j\le N_1} \left|(u_h)_{i,j}-u(x_i,y_j)\right|,  \\ 
		& \text{for 2D linear and nonlinear time-independent equations \eqref{Linear:Elliptic:2D} and \eqref{Non:Linear:Elliptic:2D}};\\
		& \|u_h-u\|_\infty
		:=\max_{0\le i,j\le N_1} \left|(u^{N_2}_h)_{i,j}-u(x_i,y_j,T)\right|,\\
		& \text{for 2D linear and nonlinear time-dependent equations \eqref{Linear:Parabolic:2D} and \eqref{parabo:nonlinear:2D}};\\
		& \|u_h-u\|_\infty
		:=\max_{0\le i,j,k\le N_1} \left|(u_h)_{i,j,k}-u(x_i,y_j,z_k)\right|,  \\ 
		& \text{for 3D linear and nonlinear time-independent equations \eqref{Model：Elliptic:3D} and \eqref{Non:Linear:Elliptic:3D}};\\
		& \|u_h-u\|_\infty
		:=\max_{0\le i,j,k\le N_1} \left|(u^{N_2}_h)_{i,j,k}-u(x_i,y_j,z_k,T)\right|,\\
		& \text{for 3D linear and nonlinear time-dependent equations \eqref{Linear:Parabolic:3D}  and \eqref{No:Linear:Parabolic:3D}};\\
	\end{split}
	\]
	where 
	\begin{align*}
		&(u_{h})_{i,j}=u_h|_{(x,y)=(x_i,y_j)}, && \hspace{-1cm} (u_{h}^{N_2})_{i,j}=u_h|_{(x,y,t)=(x_i,y_j,T)},\\
		&(u_{h})_{i,j,k}=u_h|_{(x,y,z)=(x_i,y_j,z_k)}, && \hspace{-1cm} (u_{h}^{N_2})_{i,j,k}=u_h|_{(x,y,z,t)=(x_i,y_j,z_k,T)}.
	\end{align*}
	As we consider 8 PDEs in this paper, we provide one example for each PDE (total 8 examples) in this section.
	
	We test 2D linear and nonlinear, steady and unsteady equations  in Examples \ref{Example:1}--\ref{Example:4}. Recall that the fourth-order compact 9-point FDM in \cref{thm:FDM:2D} is our main result in 2D.
	
	In the following \cref{Example:1}, we test the 2D linear time-independent equation \eqref{Linear:Elliptic:2D}.
	\begin{example}\label{Example:1}
		\normalfont
		The functions  of the 2D linear time-independent equation \eqref{Linear:Elliptic:2D} are given by
		\begin{align*}
			&u=\exp(x-2y)\cos(8x+4y),\qquad \kappa= ( 3+\sin(4x)\cos(4y) )^{-1} ,\qquad  \\
			&\alpha= \cos(x)\cos(y), \qquad \beta= \sin(x)\sin(y), \qquad \lambda=\exp(x+2y), \qquad \Omega=(0,1)^2,
		\end{align*}
		the source term $\phi$	and the Dirichlet boundary function $g$  are obtained by plugging above functions into \eqref{Linear:Elliptic:2D}. We use  the proposed FDM in \cref{thm:FDM:2D} to solve this example.	The numerical results are presented in \cref{Example:1:table} and \cref{Example:1:fig}.	
	\end{example}
	\begin{table}[htbp]
		\caption{Performance in \cref{Example:1} of the proposed FDM in \cref{thm:FDM:2D}.}
		\centering
		{\renewcommand{\arraystretch}{1.0}
		\scalebox{1}{
			\setlength{\tabcolsep}{5mm}{
				\begin{tabular}{c|c|c|c|c|c}
					\hline
					$h$ &   $\|u_{h}-u\|_\infty$    &order &   	$h$ &   $\|u_{h}-u\|_{\infty}$    &order \\
					\hline
					$1/2$  & 2.6500E+00  &   & $1/2^6$  & 1.7536E-06  & 4.00\\
					$1/2^2$  & 1.1246E-01  & 4.56  & $1/2^7$  & 1.0957E-07  & 4.00\\
					$1/2^3$  & 7.1938E-03  & 3.97  & $1/2^8$  & 6.8480E-09  & 4.00\\
					$1/2^4$  & 4.4988E-04  & 4.00  & $1/2^9$  & 4.2893E-10  & 4.00\\
					$1/2^5$  & 2.7960E-05  & 4.01  & $1/2^{10}$  & 3.0641E-11  & 3.81\\			
					\hline
		\end{tabular}}}}
		\label{Example:1:table}
	\end{table}
	\begin{figure}[htbp]
		\centering
		\begin{subfigure}[b]{0.31\textwidth}
			\includegraphics[width=6cm,height=6cm]{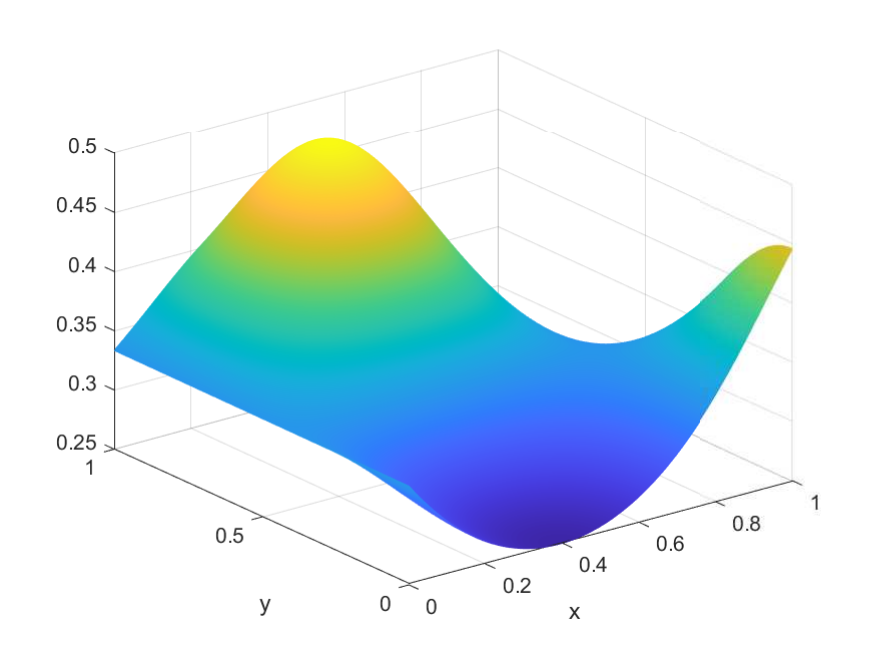}
		\end{subfigure}
		\begin{subfigure}[b]{0.31\textwidth}
			\includegraphics[width=6cm,height=6cm]{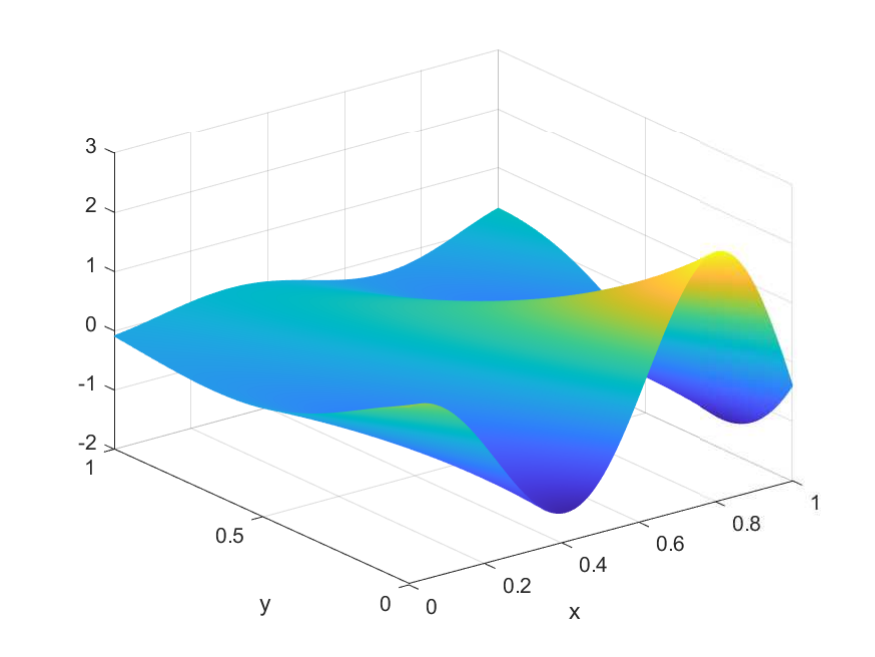}
		\end{subfigure}
		\begin{subfigure}[b]{0.31\textwidth}
			\includegraphics[width=6cm,height=6cm]{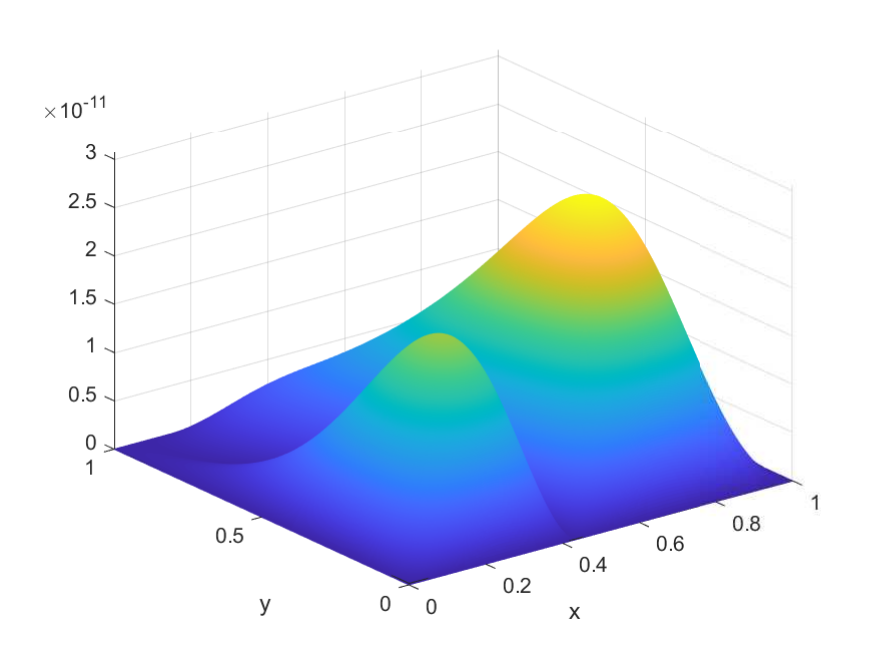}
		\end{subfigure}
		\caption
		{\cref{Example:1}: The diffusion coefficient $\kappa$ (first),  the exact solution $u$ (second), the error $|u_h-u|$ with the numerical solution $u_h$ computed by the FDM in \cref{thm:FDM:2D} (third) on the closure of the spatial domain $[0,1]^2$ with $h=1/2^{10}$.}
		\label{Example:1:fig}
	\end{figure}	
	
	In the following \cref{Example:2}, we test the 2D linear time-dependent equation \eqref{Linear:Parabolic:2D} with the high-frequency solution.
	\begin{example}\label{Example:2}
		\normalfont
		The functions  of the 2D linear time-dependent equation \eqref{Linear:Parabolic:2D} are given by
		\begin{align*}
			&u=\sin(100x)\sin(100y)\exp(t),\qquad \kappa=\exp(2x-y+t),\qquad  \alpha= \cos(x)\cos(y)\sin(t), \\
			&\beta= \sin(x)\sin(y)\cos(t), \qquad \lambda=\exp(x-y)\sin(t-x), \qquad \Omega=(0,1)^2,\\
			&  I=[0,1]\ (\text{temporal domain}),
		\end{align*}
		the source term $\phi$, the initial function $u^0$, and the Dirichlet boundary function $g$  are obtained by plugging above functions into \eqref{Linear:Parabolic:2D}. We use  the proposed FDM in \cref{thm:FDM:2D} with the CN method \eqref{CN:eq}, the BDF3 method \eqref{BDF3:eq}, and the BDF4 method \eqref{BDF4:eq}  to solve this example. Note that we apply the CN method with the FDM in \cref{thm:FDM:2D} using $\tau=h/2$ to obtain $u^1$ and $u^2$ for the BDF3 method and  compute $u^1,u^2,u^3$ for the BDF4 method.  Furthermore, second-order 3-point to 4-point FD operators \eqref{ux:order:2} and \eqref{uxx:order:2} are utilized to approximate required  first-order and second-order partial derivatives used in \cref{thm:FDM:2D} for \eqref{CN:eq}, \eqref{BDF3:eq}, and \eqref{BDF4:eq}.
		The numerical results are presented in \cref{Example:2:table} and \cref{Example:2:fig}.	
	\end{example}
	\begin{table}[htbp]
		\caption{Performance in \cref{Example:2} of the proposed FDM in \cref{thm:FDM:2D}. Note that $\Qw^{(m,n)}:=\tfrac{\partial^{m+n}\Qw}{\partial x^m \partial y^n }$ with $m+n\le 2$ represent partial derivatives of $\Qw=a^{n+i},b^{n+i},d^{n+i},f^{n+i}$ with $i=1/2,3,4$ that used in \eqref{CN:eq}, \eqref{BDF3:eq}, and \eqref{BDF4:eq}  for the FDM in \cref{thm:FDM:2D}.}
		\centering
		{\renewcommand{\arraystretch}{1.0}
		\scalebox{1}{
			\setlength{\tabcolsep}{2mm}{
				\begin{tabular}{c|c|c|c|c|c|c|c|c|c|c|c}
					\hline
					\multicolumn{12}{c}{Use \eqref{ux:order:2} and \eqref{uxx:order:2} for $\Qw^{(m,n)}$}\\
					\hline
					\multicolumn{4}{c|}{CN method \eqref{CN:eq}} &
					\multicolumn{4}{c|}{ BDF3 method \eqref{BDF3:eq}} &
					\multicolumn{4}{c}{ BDF4 method \eqref{BDF4:eq} } \\
					\cline{1-12}
					$h$ & $\tau$ &  $\|u_{h}-u\|_\infty$    &order &   	$h$ &  $\tau=h$ &   $\|u_{h}-u\|_{\infty}$    &order & 	$h$ &  $\tau=h$ &   $\|u_{h}-u\|_{\infty}$    &order \\
					\hline
					$1/2^2$  &  $1/2^2$  &  2.5701E+04  &    &  $1/2^2$  &  $1/2^2$  &  2.9143E+04  &    &  $1/2^2$  &  $1/2^2$  &  2.9188E+04  &  \\
					$1/2^3$  &  $1/2^4$  &  7.5992E+03  &  1.76  &  $1/2^3$  &  $1/2^3$  &  7.5729E+03  &  1.94  &  $1/2^3$  &  $1/2^3$  &  7.5755E+03  &  1.95\\
					$1/2^4$  &  $1/2^6$  &  1.6875E+03  &  2.17  &  $1/2^4$  &  $1/2^4$  &  1.6875E+03  &  2.17  &  $1/2^4$  &  $1/2^4$  &  1.6875E+03  &  2.17\\
					$1/2^5$  &  $1/2^8$  &  2.7609E+00  &  9.26  &  $1/2^5$  &  $1/2^5$  &  2.7643E+00  &  9.25  &  $1/2^5$  &  $1/2^5$  &  2.7643E+00  &  9.25\\
					$1/2^6$  &  $1/2^{10}$  &  3.4206E-01  &  3.01  &  $1/2^6$  &  $1/2^6$  &  3.4205E-01  &  3.01  &  $1/2^6$  &  $1/2^6$  &  3.4205E-01  &  3.01\\
					$1/2^7$  &  $1/2^{12}$  &  1.7998E-02  &  4.25  &  $1/2^7$  &  $1/2^7$  &  1.7998E-02  &  4.25  &  $1/2^7$  &  $1/2^7$  &  1.7998E-02  &  4.25\\
					$1/2^8$  &  $1/2^{14}$  &  1.0799E-03  &  4.06  &  $1/2^8$  &  $1/2^8$  &  1.0799E-03  &  4.06  &  $1/2^8$  &  $1/2^8$  &  1.0799E-03  &  4.06\\
					&    &    &    &  $1/2^9$  &  $1/2^9$  &  6.7549E-05  &  4.00  &  $1/2^9$  &  $1/2^9$  &  6.7549E-05  &  4.00\\
					\hline
		\end{tabular}}}}
		\label{Example:2:table}
	\end{table}
	\begin{figure}[htbp]
		\centering
		\begin{subfigure}[b]{0.31\textwidth}
			\includegraphics[width=6cm,height=6cm]{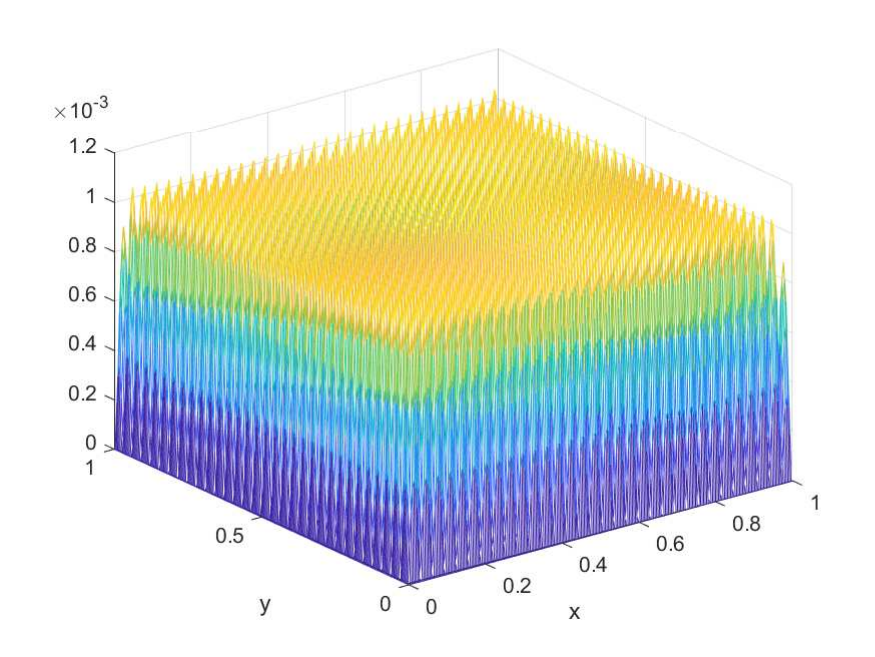}
		\end{subfigure}
		\begin{subfigure}[b]{0.31\textwidth}
			\includegraphics[width=6cm,height=6cm]{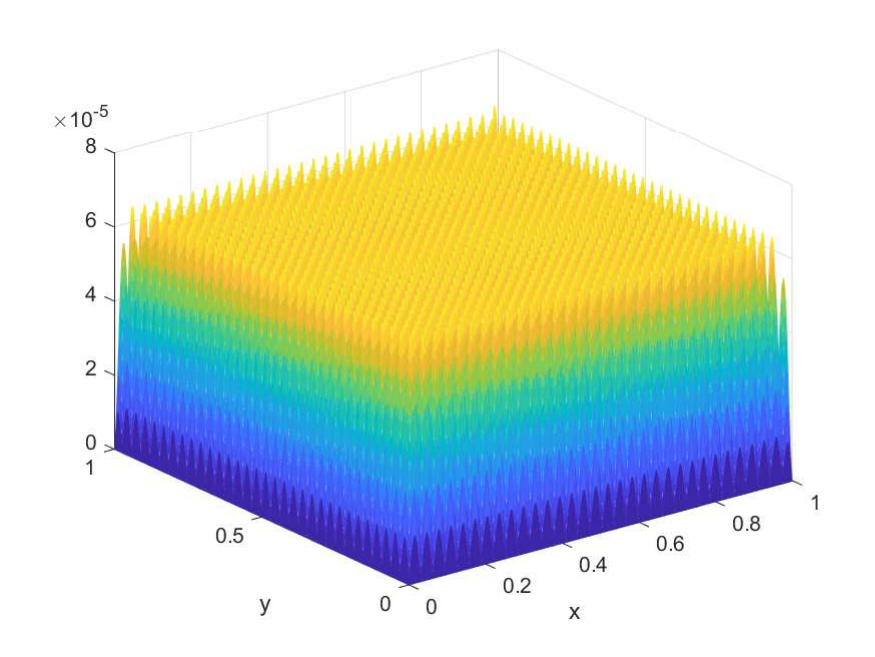}
		\end{subfigure}
		\begin{subfigure}[b]{0.31\textwidth}
			\includegraphics[width=6cm,height=6cm]{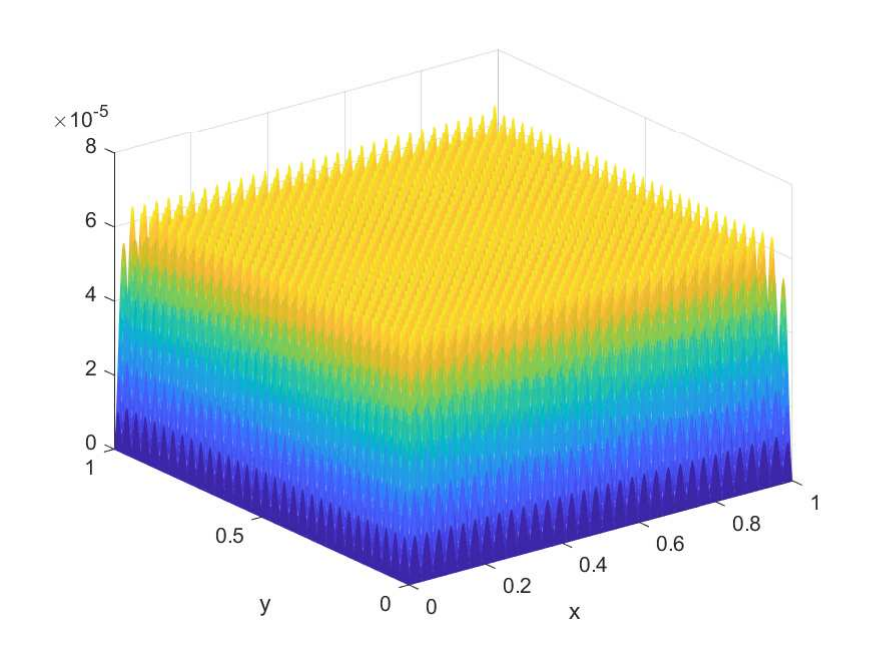}
		\end{subfigure}
		\caption
		{\cref{Example:2}: The error $|u_h-u|$ at $t=1$ with the numerical solution $u_h$ computed by the FDM in \cref{thm:FDM:2D} with the CN method \eqref{CN:eq}  with $\tau=1/2^{14}$ and $h=1/2^{8}$ (first), the error $|u_h-u|$ at $t=1$  with the numerical solution $u_h$ computed by the FDM in \cref{thm:FDM:2D} with the BDF3 method \eqref{BDF3:eq}  with $\tau=h=1/2^{9}$ (second), and the error $|u_h-u|$ at $t=1$ with the numerical solution $u_h$ computed by the FDM in \cref{thm:FDM:2D} with the BDF4 method \eqref{BDF4:eq}  with $\tau=h=1/2^{9}$ (third) on the closure of the spatial domain $[0,1]^2$.}
		\label{Example:2:fig}
	\end{figure}	
	\begin{remark} \label{remark:examp:2}
		We also apply the FDM in  \cref{thm:FDM:parabo:2D} with the CN method \eqref{CN:eq}, the BDF3 method \eqref{BDF3:eq}, and the BDF4 method \eqref{BDF4:eq} using FD operators \eqref{ux:order:2}, \eqref{uxx:order:2}, and \eqref{uxxx:order:1} to test \cref{Example:2}. As performances of  FDMs in  \cref{thm:FDM:2D} and \cref{thm:FDM:parabo:2D} are nearly identical, we omit the results from \cref{thm:FDM:parabo:2D}. Furthermore, we use second-order FD operators  \eqref{ux:order:2} and \eqref{uxx:order:2} to approximate required  first-order and second-order partial derivatives used in \cref{thm:FDM:2D} for \eqref{CN:eq}, \eqref{BDF3:eq}, and \eqref{BDF4:eq}, but numerical results in \cref{Example:2:table} still present the fourth-order accuracy, which demonstrates the robustness of our proposed FDM.
	\end{remark}	
	
	In the following \cref{Example:3}, we test the 2D nonlinear time-independent equation \eqref{Non:Linear:Elliptic:2D} with the challenging coefficient $\kappa<2.2\times 10^{-4}$.
	\begin{example}\label{Example:3}
		\normalfont
		The functions  of the 2D nonlinear time-independent equation \eqref{Non:Linear:Elliptic:2D} are given by
		\begin{align*}
			&u=\sin(x)\sin(y),\qquad \kappa=10^{-4}\exp(u), \\
			&\alpha= u^2, \qquad \beta= u^3, \qquad \lambda=\exp(-u), \qquad \Omega=(0,1)^2,
		\end{align*}
		the source term $\phi$	and the Dirichlet boundary function $g$  are obtained by plugging above functions into \eqref{Non:Linear:Elliptic:2D}. We use  the proposed FDM in \cref{thm:FDM:2D} to solve this example. Furthermore, FD operators
		\eqref{ux:order:3}--\eqref{ux:order:5} are used to calculate $(u_\Qk)_x$ and $(u_\Qk)_y$ in \eqref{ak:bk:dk:fk} for \eqref{nonlinear:2D:ellip:ite} for the FDM in \cref{thm:FDM:2D}, FD operators \eqref{ux:order:2}--\eqref{ux:order:4} and \eqref{uxx:order:2}--\eqref{uxx:order:4} are used to calculate the first-order and second-order partial derivatives of $a_\Qk$, $b_\Qk$, $d_\Qk$, and $f_\Qk$  in \eqref{nonlinear:2D:ellip:ite} for the  FDM in \cref{thm:FDM:2D}.
		The numerical results are presented in \cref{Example:3:table} and \cref{Example:3:fig}.	From results in \cref{Example:3:table}, we observe that higher-order or simpler FD operators for  $(u_\textup \Qk)_x, (u_\textup\Qk)_y$,  first-order to  second-order partial derivatives of $a_\textup\Qk$, $b_\textup\Qk$, $d_\textup\Qk$, and $f_\textup\Qk$  in \eqref{nonlinear:2D:ellip:ite}  produce smaller errors. As the FDM in \cref{thm:FDM:2D} is the fourth-order accuracy, results of Settings 2 and 3 are very close in  \cref{Example:3:table}.
	\end{example}
	\begin{table}[htbp]
		\caption{Performance in \cref{Example:3} of the proposed FDM in \cref{thm:FDM:2D}. Note that $\Qw^{(m,n)}:=\tfrac{\partial^{m+n}\Qw}{\partial x^m \partial y^n }$ with $m+n\le 2$ represent partial derivatives of $\Qw=a_\Qk,b_\Qk,d_\Qk,f_\Qk$ that used in \eqref{nonlinear:2D:ellip:ite} for the FDM in \cref{thm:FDM:2D}. }
		\centering
		{\renewcommand{\arraystretch}{1.0}
		\scalebox{1}{
			\setlength{\tabcolsep}{4mm}{
				\begin{tabular}{c|c|c|c|c|c|c|c|c}
					\hline
					\multicolumn{3}{c|}{Setting 1} &
					\multicolumn{3}{c|}{Setting 2} &
					\multicolumn{3}{c}{ Setting 3 } \\
					\hline
					\multicolumn{3}{c|}{\eqref{ux:order:3} for  $(u_\Qk)_x$ and $(u_\Qk)_y$} &
					\multicolumn{3}{c|}{\eqref{ux:order:4} for  $(u_\Qk)_x$ and $(u_\Qk)_y$} &
					\multicolumn{3}{c}{ \eqref{ux:order:5} for  $(u_\Qk)_x$ and $(u_\Qk)_y$ } \\
					\hline
					\multicolumn{3}{c|}{\eqref{ux:order:2} and \eqref{uxx:order:2} for $\Qw^{(m,n)}$} &
					\multicolumn{3}{c|}{\eqref{ux:order:3} and \eqref{uxx:order:43} for $\Qw^{(m,n)}$} &
					\multicolumn{3}{c}{ \eqref{ux:order:4} and \eqref{uxx:order:4} for $\Qw^{(m,n)}$ } \\
					\cline{1-9}
					\hline
					$h$ &   $\|u_{h}-u\|_\infty$    &order &   	$h$ &   $\|u_{h}-u\|_{\infty}$    &order & $h$ &   $\|u_{h}-u\|_{\infty}$    &order \\
					\hline
					$1/2^2$  &  2.1759E-02  &    &  $1/2^2$  &  6.3150E-03  &    &    &    &  \\
					$1/2^3$  &  3.5959E-03  &  2.60  &  $1/2^3$  &  1.0915E-03  &  2.53  &  $1/2^3$  &  1.3093E-03  &  \\
					$1/2^4$  &  5.2502E-04  &  2.78  &  $1/2^4$  &  1.8267E-04  &  2.58  &  $1/2^4$  &  2.0504E-04  &  2.67\\
					$1/2^5$  &  5.6626E-05  &  3.21  &  $1/2^5$  &  2.1998E-05  &  3.05  &  $1/2^5$  &  2.3288E-05  &  3.14\\
					$1/2^6$  &  4.6018E-06  &  3.62  &  $1/2^6$  &  1.9364E-06  &  3.51  &  $1/2^6$  &  1.9867E-06  &  3.55\\
					$1/2^7$  &  3.2363E-07  &  3.83  &  $1/2^7$  &  1.4022E-07  &  3.79  &  $1/2^7$  &  1.4181E-07  &  3.81\\
					$1/2^8$  &  2.1230E-08  &  3.93  &  $1/2^8$  &  9.2721E-09  &  3.92  &  $1/2^8$  &  9.3212E-09  &  3.93\\
					$1/2^9$  &  1.3480E-09  &  3.98  &  $1/2^9$  &  5.9108E-10  &  3.97  &  $1/2^9$  &  5.9259E-10  &  3.98\\
					$1/2^{10}$  &  8.4491E-11  &  4.00  &  $1/2^{10}$  &  3.7117E-11  &  3.99  &  $1/2^{10}$  &  3.7164E-11  &  4.00\\		
					\hline
		\end{tabular}}}}
		\label{Example:3:table}
	\end{table}
	\begin{figure}[htbp]
		\centering
		\begin{subfigure}[b]{0.23\textwidth}
			\includegraphics[width=4cm,height=4cm]{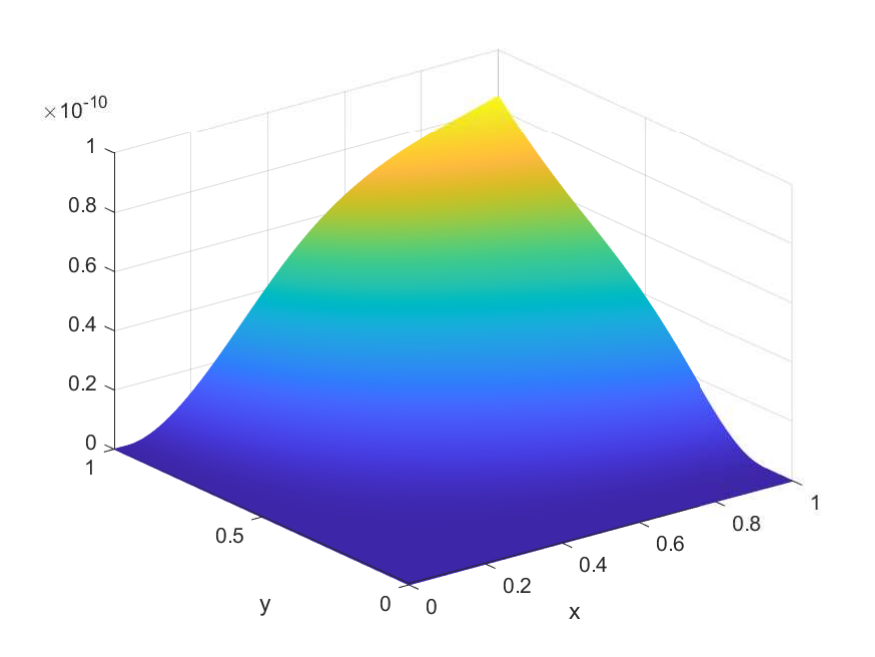}
		\end{subfigure}
		\begin{subfigure}[b]{0.23\textwidth}
			\includegraphics[width=4cm,height=4cm]{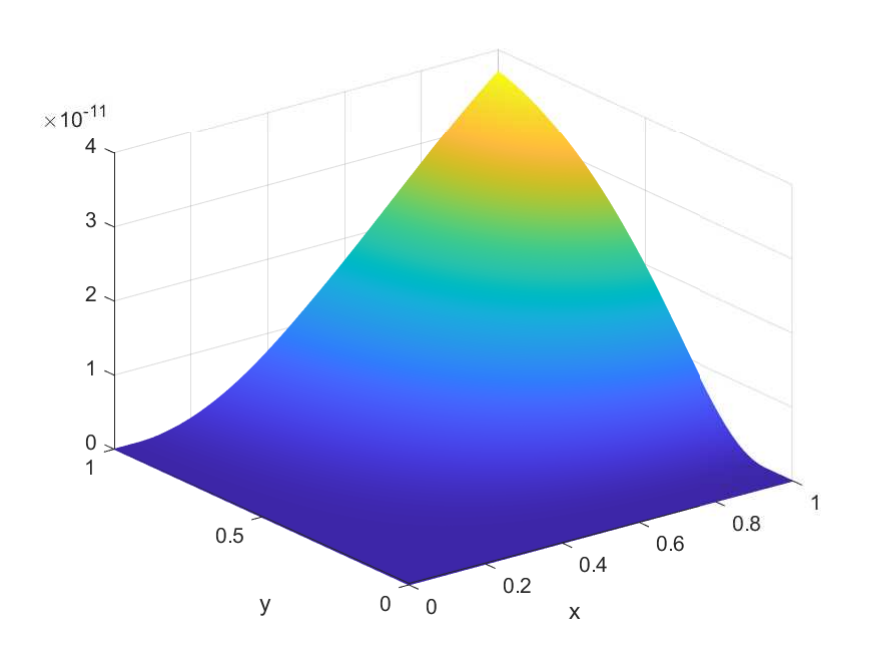}
		\end{subfigure}
		\begin{subfigure}[b]{0.23\textwidth}
			\includegraphics[width=4cm,height=4cm]{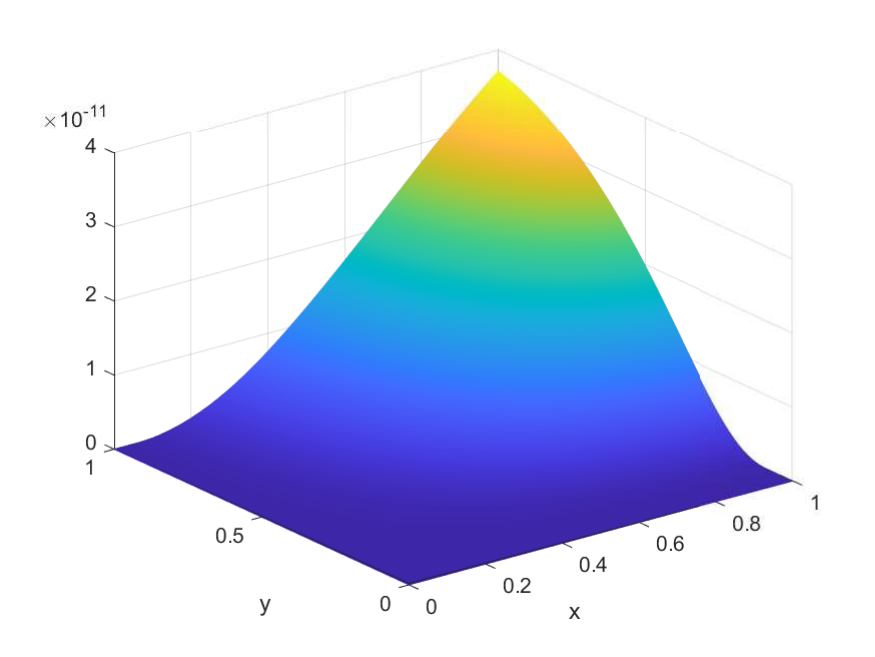}
		\end{subfigure}
		\begin{subfigure}[b]{0.23\textwidth}
			\includegraphics[width=4cm,height=4cm]{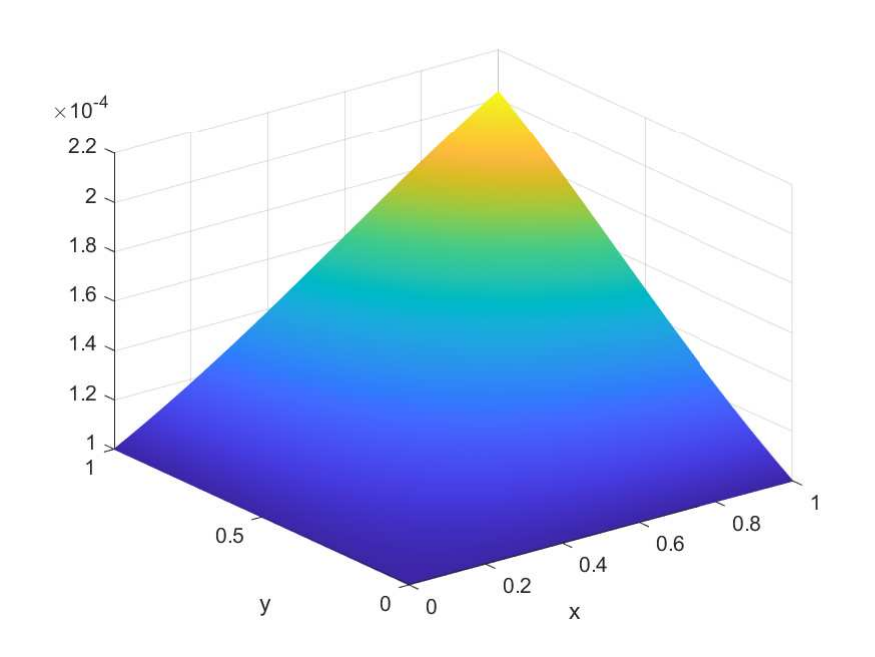}
		\end{subfigure}
		\caption
		{ The errors $|u_h-u|$ of \cref{Example:3} with the numerical solutions $u_h$ computed by the FDM in \cref{thm:FDM:2D} with Settings 1-3 in \cref{Example:3:table}:   $|u_h-u|$ with the Setting 1 (first),  $|u_h-u|$ with the Setting 2 (second),  $|u_h-u|$ with the Setting 3 (third) on the closure of the spatial domain $[0,1]^2$ with $h=1/2^{10}$, the coefficient $\kappa$ (fourth). See details of Settings 1-3 in \cref{Example:3:table}.}
		\label{Example:3:fig}
	\end{figure}	
	
	In the following \cref{Example:4}, we test the 2D nonlinear time-dependent equation \eqref{parabo:nonlinear:2D}. By FDMs in \cref{thm:FDM:2D} and \cref{thm:FDM:parabo:2D} with CN, BDF3, BDF4 methods, we need to  calculate $ \{  \tfrac{\partial^{m+n} \varrho(x,y,t_{n+i}) }{ \partial x^m \partial y^n } : m+n\le 2 \text{ or } 3  \}$ with $\varrho=a_\Qk^{n+i}, b_\Qk^{n+i}, d_\Qk^{n+i}, f_\Qk^{n+i}$ in   \eqref{CN:2D:itera}, \eqref{BDF3:2D:itera}, \eqref{BDF4:2D:itera}, and compute $(u^{n+i}_\Qk)_x$ and $(u^{n+i}_\Qk)_y$ in \eqref{akni} with $i=1/2,3,4$. By \eqref{akni}, \eqref{CN:2D:itera}, \eqref{BDF3:2D:itera}, \eqref{BDF4:2D:itera}, it is obviously that $(u^{n+i}_\Qk)_x$ and $(u^{n+i}_\Qk)_y$  should have a higher order accuracy than $\varrho_x$ and $\varrho_y$  (see \cref{Example:3:table} for an example). To make  codes simple and check the flexibility of our scheme, we use the same approximation in the following \cref{Example:4}.
	\begin{example}\label{Example:4}
		\normalfont
		The functions  of the 2D nonlinear time-dependent equation \eqref{parabo:nonlinear:2D} are given by
		\begin{align*}
			&u=\sin(2x)\sin(2y)\exp(t),\qquad \kappa=2+\sin u,\qquad  \alpha= \cos u, \\
			& \beta= \sin u, \qquad \lambda=\sin(3u), \qquad \Omega=(0,1)^2, \qquad  I=[0,1]\ (\text{temporal domain}),
		\end{align*}
		the source term $\phi$, the initial function $u^0$,	and the Dirichlet boundary function $g$  are obtained by plugging above functions into \eqref{parabo:nonlinear:2D}. We use proposed FDMs in \cref{thm:FDM:2D} and \cref{thm:FDM:parabo:2D} with BDF3 and BDF4 methods to solve this example. Note that we utilize the CN method with FDMs in \cref{thm:FDM:2D,thm:FDM:parabo:2D} using $\tau=h/2$ to calculate $u^1$ and $u^2$ for the BDF3 method and compute $u^1,u^2,u^3$ for the BDF4 method. Furthermore, FD operators
		\eqref{ux:order:4}--\eqref{ux:order:5}, \eqref{uxx:order:43}--\eqref{uxx:order:4}, and \eqref{uxxx:order:2}--\eqref{uxxx:order:3}  are used to calculate  $(u^{n+i}_\Qk)_x$ and $(u^{n+i}_\Qk)_y$ with $i=1/2,3,4$ in \eqref{akni},    the first-order to third-order partial derivatives of $a^{n+i}_\Qk$, $b^{n+i}_\Qk$, $d^{n+i}_\Qk$, and $f^{n+i}_\Qk$ with $i=1/2,3,4$  in \eqref{CN:2D:itera}, \eqref{BDF3:2D:itera} and \eqref{BDF4:2D:itera} used in FDMs of \cref{thm:FDM:2D} and \cref{thm:FDM:parabo:2D}.
		The numerical results are presented in \cref{Example:4:table} and \cref{Example:4:fig}.	
	\end{example}
	\begin{table}[htbp]
		\caption{Performance in \cref{Example:4} of proposed FDMs in \cref{thm:FDM:2D} and \cref{thm:FDM:parabo:2D}. }
		\centering
		{\renewcommand{\arraystretch}{1.0}
		\scalebox{1}{
			\setlength{\tabcolsep}{1.5mm}{
				\begin{tabular}{c|c|c|c|c|c|c|c|c|c|c|c}
					\hline
					\multicolumn{9}{c|}{Use \eqref{ux:order:4}, \eqref{uxx:order:43}, \eqref{uxxx:order:2}} &
					\multicolumn{3}{c}{Use \eqref{ux:order:5}, \eqref{uxx:order:4},  \eqref{uxxx:order:3}} \\
					\hline
					\multicolumn{1}{c|}{} &
					\multicolumn{2}{c|}{BDF3, $\tau=h$} &
					\multicolumn{2}{c|}{BDF4, $\tau=h$} &
					\multicolumn{2}{c|}{ BDF3, $\tau=h$} &
					\multicolumn{2}{c|}{ BDF4, $\tau=h$} &
					\multicolumn{1}{c|}{} &
					\multicolumn{2}{c}{BDF4, $\tau=h$}
					\\
					\hline
					\multicolumn{1}{c|}{} &
					\multicolumn{2}{c|}{\cref{thm:FDM:2D}} &
					\multicolumn{2}{c|}{\cref{thm:FDM:2D}} &
					\multicolumn{2}{c|}{\cref{thm:FDM:parabo:2D} } &
					\multicolumn{2}{c|}{\cref{thm:FDM:parabo:2D} } &
					\multicolumn{1}{c|}{} &
					\multicolumn{2}{c}{\cref{thm:FDM:parabo:2D}} \\
					\cline{1-12}
					\hline
					&   col2    &  &   	  &     &    col6    &  &   col8    & & & & \\
					\hline
					$h$ &   $\|u_{h}-u\|_\infty$    &order &   	   $\|u_{h}-u\|_{\infty}$    &order &    $\|u_{h}-u\|_{\infty}$    &order &   $\|u_{h}-u\|_\infty$    &order & 	$h$ &   \hspace{0.2cm} $\|u_{h}-u\|_\infty$ \hspace{0.2cm}   & order \\
					\hline
					$1/2^2$  &  7.8562E-03  &    &  8.2860E-03  &    &  8.4327E-03  &    &  8.7740E-03  &    &  $1/2^2$  &    &  \\
					$1/2^3$  &  2.5242E-04  &  4.96  &  2.4086E-04  &  5.10  &  2.4106E-04  &  5.13  &  2.2836E-04  &  5.26  &  $1/2^3$  &  2.3560E-04  &  \\
					$1/2^4$  &  1.0415E-05  &  4.60  &  1.2707E-05  &  4.24  &  1.2801E-05  &  4.24  &  1.4055E-05  &  4.02  &  $1/2^4$  &  1.3977E-05  &  4.08\\
					$1/2^5$  &  6.7493E-07  &  3.95  &  1.0386E-06  &  3.61  &  5.9210E-07  &  4.43  &  7.4677E-07  &  4.23  &  $1/2^5$  &  6.6868E-07  &  4.39\\
					$1/2^6$  &  5.3259E-08  &  3.66  &  1.0333E-07  &  3.33  &  3.0387E-08  &  4.28  &  4.4347E-08  &  4.07  &  $1/2^6$  &  4.1398E-08  &  4.01\\
					$1/2^7$  &  5.0304E-09  &  3.40  &  1.1351E-08  &  3.19  &  3.4442E-09  &  3.14  &  2.7525E-09  &  4.01  &  $1/2^7$  &  2.5911E-09  &  4.00\\
					$1/2^8$  &  5.4984E-10  &  3.19  &  1.3239E-09  &  3.10  &  5.5458E-10  &  2.63  &  1.7215E-10  &  4.00  &  $1/2^8$  &  1.6222E-10  &  4.00\\
					\hline
		\end{tabular}}}}
		\label{Example:4:table}
	\end{table}
	\begin{figure}[htbp]
		\centering
		\begin{subfigure}[b]{0.23\textwidth}
			\includegraphics[width=4cm,height=4cm]{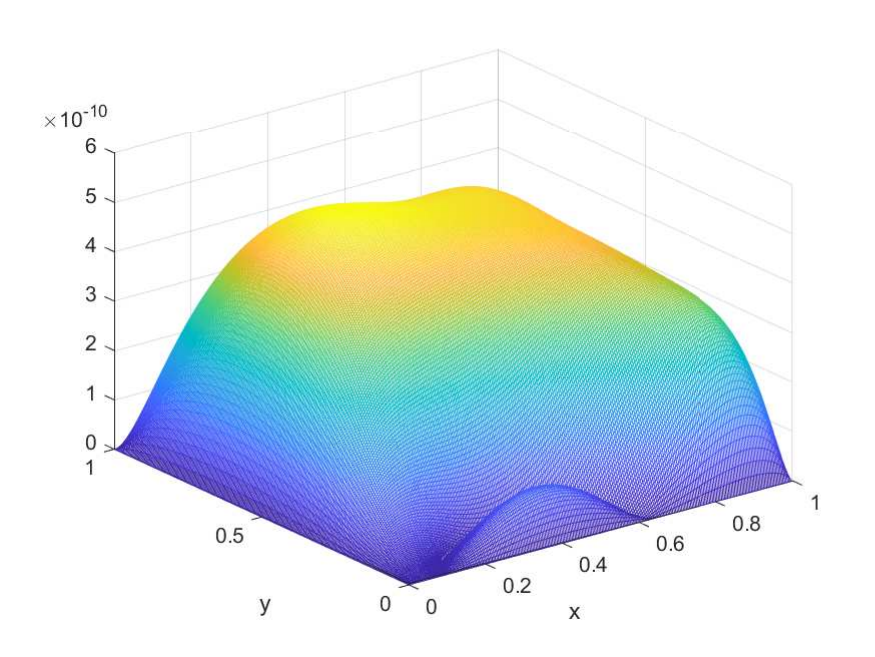}
		\end{subfigure}
		\begin{subfigure}[b]{0.23\textwidth}
			\includegraphics[width=4cm,height=4cm]{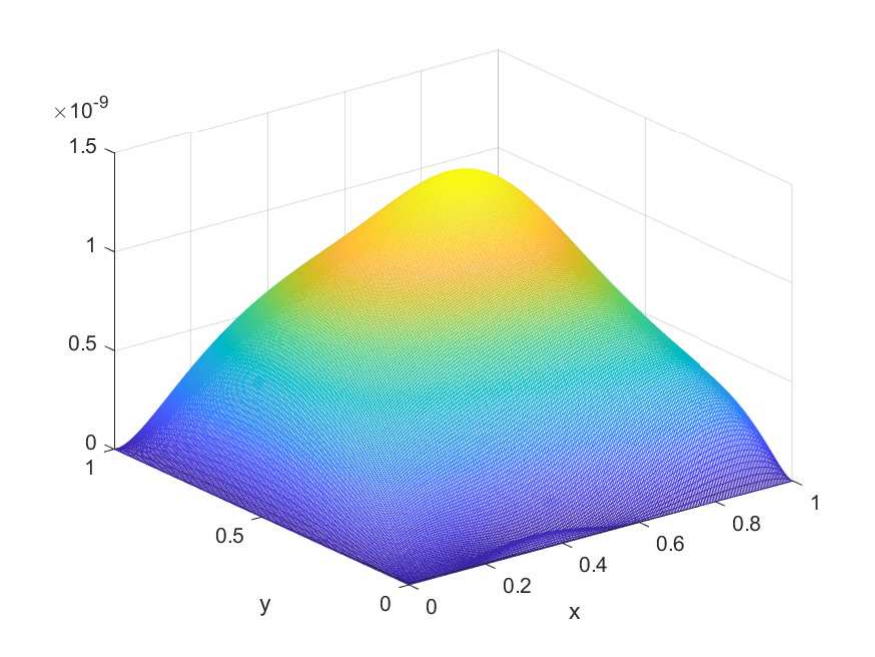}
		\end{subfigure}
		\begin{subfigure}[b]{0.23\textwidth}
			\includegraphics[width=4cm,height=4cm]{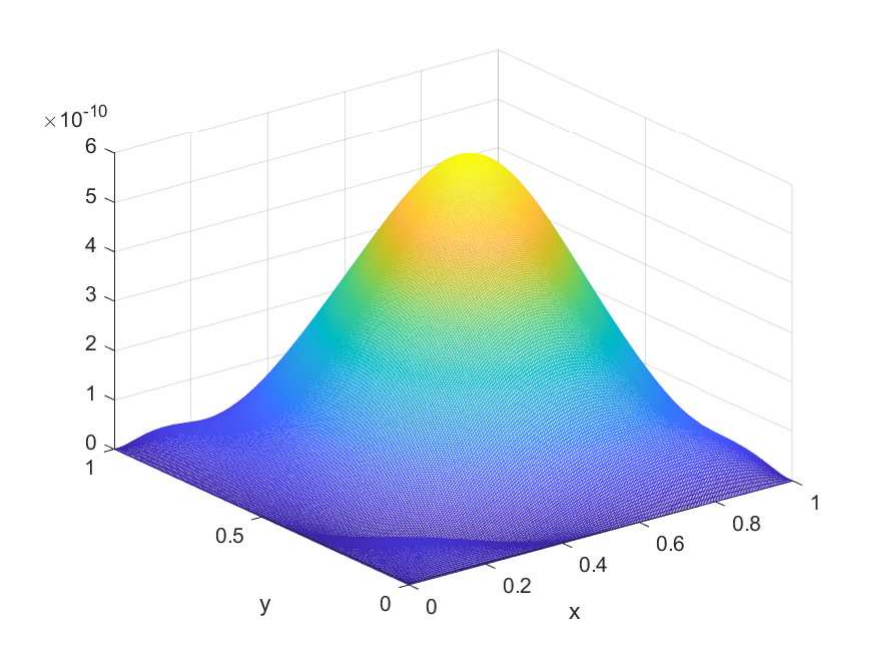}
		\end{subfigure}
		\begin{subfigure}[b]{0.23\textwidth}
			\includegraphics[width=4cm,height=4cm]{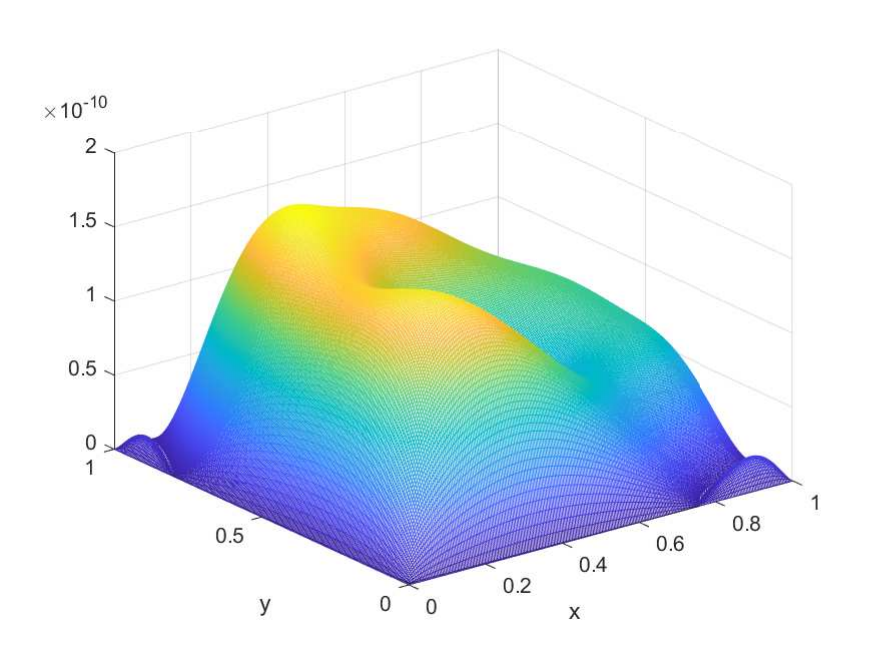}
		\end{subfigure}
		\caption
		{ The errors $|u_h-u|$ at $t=1$ of \cref{Example:4} with numerical solutions $u_h$ computed by FDMs in \cref{thm:FDM:2D} and \cref{thm:FDM:parabo:2D} with BDF3 and BDF4 methods using FD operators \eqref{ux:order:4}, \eqref{uxx:order:43}, and \eqref{uxxx:order:2}:   $|u_h-u|$ with \cref{thm:FDM:2D}, BDF3, and $\tau=h$  (first),   $|u_h-u|$ with \cref{thm:FDM:2D}, BDF4, and $\tau=h$  (second),   $|u_h-u|$ with  \cref{thm:FDM:parabo:2D}, BDF3, and $\tau=h$   (third), $|u_h-u|$ with  \cref{thm:FDM:parabo:2D}, BDF4, and $\tau=h$  (fourth) on the closure of the spatial domain $[0,1]^2$ with $h=1/2^{8}$.}
		\label{Example:4:fig}
	\end{figure}	
	\begin{remark} 
		In \cref{Example:4:table}, the BDF3 method with  \cref{thm:FDM:2D} and \cref{thm:FDM:parabo:2D} achieves the third-order accuracy when $h=1/2^8$, and the BDF4 method with \cref{thm:FDM:parabo:2D} achieves the fourth-order accuracy which satisfy the expectation. From \cref{remark:order:3}, we have that the degree of  $h$	of the leading term of the truncation error of the FDM in \cref{thm:FDM:2D} with the BDF4 method is three if $\tau=rh$ with the positive constant $r$. So the accuracy order of the BDF4 method with  \cref{thm:FDM:2D} is three in \cref{Example:4:table} when $h=1/2^8$. Observing from $\|u_{h}-u\|_\infty$ in col2, col6 and col8, even the BDF3 method with \cref{thm:FDM:2D} only achieves third-order accuracy,  our proposed FDM in \cref{thm:FDM:2D} with the simple stencil can still produce the reliable numerical solution for the  2D nonlinear time-dependent equation \eqref{parabo:nonlinear:2D}.
	\end{remark}	
	
	We test 3D linear and nonlinear, stationary and nonstationary equations  in Examples \ref{Example:5}--\ref{Example:8}. Note that the fourth-order compact 19-point FDM in \cref{thm:FDM:3D} is our main result in 3D.
	
	In the following \cref{Example:5}, we test the 3D linear time-independent equation \eqref{Model：Elliptic:3D}.
	\begin{example}\label{Example:5}
		\normalfont
		The functions  of the 3D linear time-independent equation \eqref{Model：Elliptic:3D} are given by
		\begin{align*}
			&u=\exp(x^2-y^2)\sin(z),\qquad \kappa=\exp(xyz),\qquad \alpha= \sin(x)\sin(y)\sin(z), \\
			& \beta= \cos(x)\cos(y)\cos(z), \qquad \gamma=x^2+2y^2+4z^2,\qquad \lambda=x^3+3y^3+6z^3, \qquad \Omega=(-1,1)^3,
		\end{align*}
		the source term $\phi$	and the Dirichlet boundary function $g$  are obtained by plugging above functions into \eqref{Model：Elliptic:3D}. We use  the proposed FDM in \cref{thm:FDM:3D} to solve this example.	The numerical results are presented in \cref{Example:5:table} and \cref{Example:5:fig}.	
	\end{example}
	\begin{table}[htbp]
		\caption{Performance in \cref{Example:5} of the proposed FDM in \cref{thm:FDM:3D}.}
		\centering
		{\renewcommand{\arraystretch}{1.0}
		\scalebox{1}{
			\setlength{\tabcolsep}{10mm}{
				\begin{tabular}{c|c|c}
					\hline
					$h$ &   $\|u_{h}-u\|_\infty$    &order \\
					\hline
					$2/2$  &  2.6319E-01  &  \\
					$2/2^2$  &  4.8038E-02  &  2.45 \\
					$2/2^3$  &  2.7194E-03  &  4.14 \\
					$2/2^4$  &  1.7121E-04  &  3.99 \\
					$2/2^5$  &  1.2266E-05  &  3.80 \\
					$2/2^6$  &  7.8650E-07  &  3.96 \\
					\hline
		\end{tabular}}}}
		\label{Example:5:table}
	\end{table}
	\begin{figure}[htbp]
		\centering
		\begin{subfigure}[b]{0.31\textwidth}
			\includegraphics[width=6cm,height=5cm]{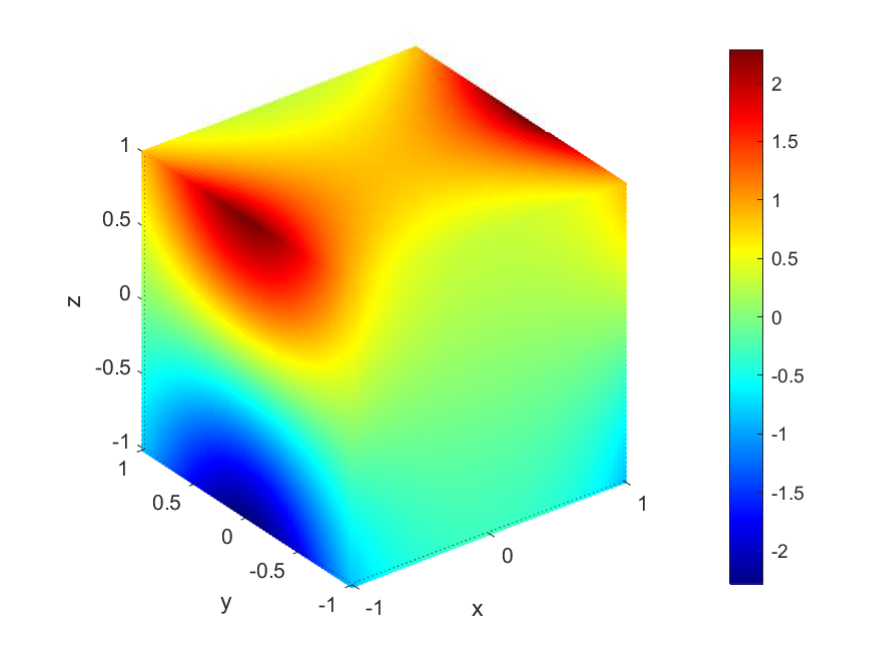}
		\end{subfigure}
		\begin{subfigure}[b]{0.31\textwidth}
			\includegraphics[width=6cm,height=5cm]{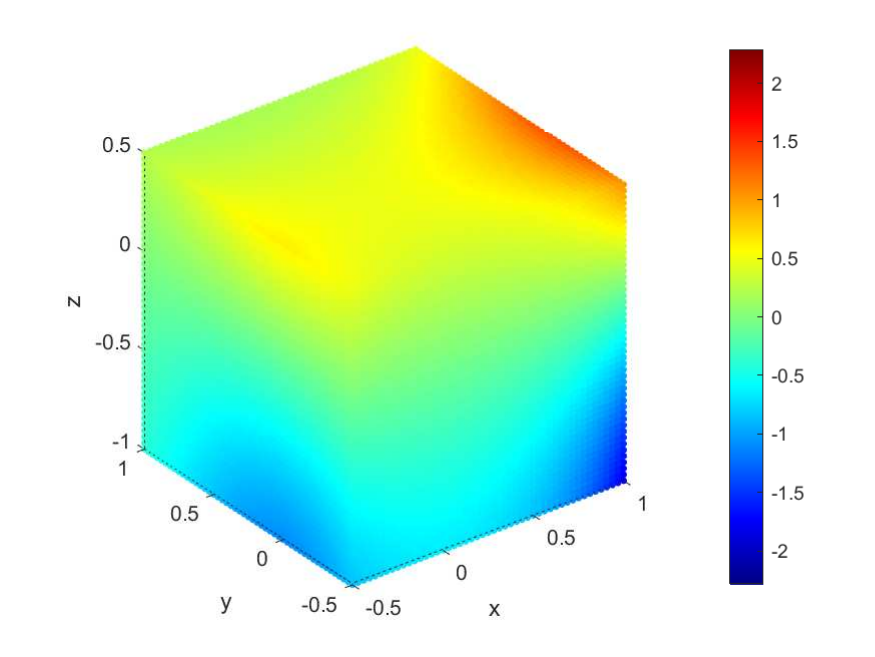}
		\end{subfigure}
		\begin{subfigure}[b]{0.31\textwidth}
			\includegraphics[width=6cm,height=5cm]{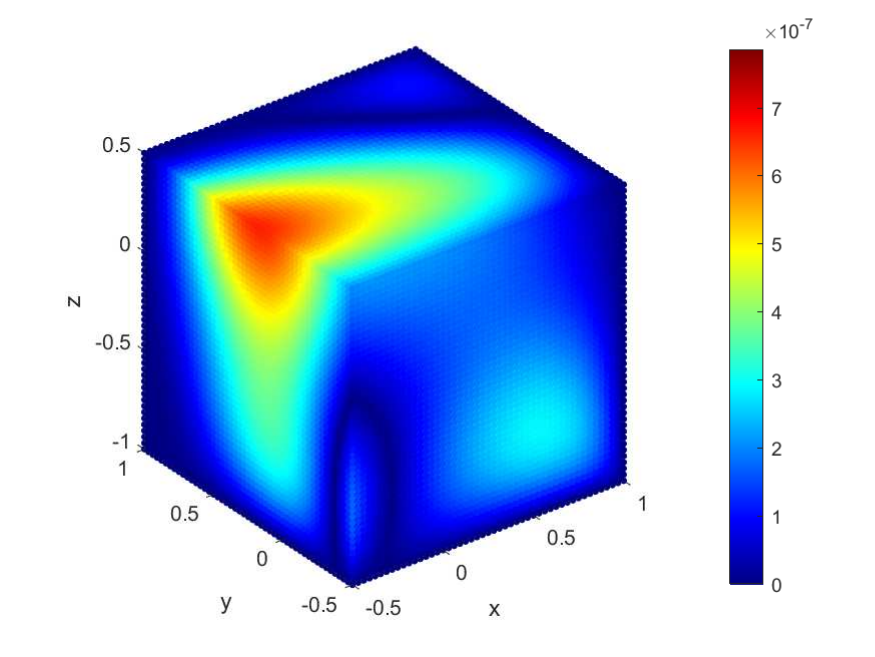}
		\end{subfigure}
		\caption
		{\cref{Example:5}: The exact solution $u$  on $[-1,1]^3$ (first), the exact solution $u$ on the subdomain (spatial domain $\Omega$ is $(-1,1)^3$) $[-1/2,1]\times [-1/2,1] \times [-1,1/2]$ (second), the error $|u_h-u|$ with the numerical solution $u_h$ computed by the FDM in \cref{thm:FDM:3D}  on the subdomain (spatial domain $\Omega$ is $(-1,1)^3$) $[-1/2,1]\times [-1/2,1] \times [-1,1/2]$ (third) with $h=2/2^{6}$.}
		\label{Example:5:fig}
	\end{figure}	
	
	In the following \cref{Example:6}, we test the 3D linear time-dependent equation \eqref{Linear:Parabolic:3D}. According to \cref{remark:examp:2}, we observe that performances of FDMs in  \cref{thm:FDM:2D}  and \cref{thm:FDM:parabo:2D} are nearly identical. To keep numerical results concise, we only use the BDF3 method with \cref{thm:FDM:3D}  to test this example. In  \cref{Example:8}, we compare performances of \cref{thm:FDM:3D} and \cref{thm:FDM:parabo:3D} with the BDF4 method.
	\begin{example}\label{Example:6}
		\normalfont
		The functions  of the 3D linear time-dependent equation \eqref{Linear:Parabolic:3D} are given by
		\begin{align*}
			&u=\cos(2(x+y-z))\exp(t),\qquad \kappa= 2+\sin(2x-y-z+t) ,\qquad \alpha= \exp(x+y+z-t), \\
			& \beta= \exp(x-y+z+2t), \qquad \gamma=\sin(x+y+z-3t),\qquad \lambda=\cos(x+y+z+4t), \\
			& \Omega=(0,1)^3,  \qquad I=[0,1]\ (\text{temporal domain}),
		\end{align*}
		the source term $\phi$, the initial function $u^0$,	and the Dirichlet boundary function $g$  are obtained by plugging above functions into \eqref{Linear:Parabolic:3D}. We use  the proposed FDM in \cref{thm:FDM:3D} with the BDF3 method to solve this example. Note that we use the FDM in \cref{thm:FDM:3D} with the CN method and $\tau=h/2$ to compute $u^1,u^2$ for the BDF3 method. Furthermore,  
		FD operators  \eqref{ux:order:2}--\eqref{ux:order:5},  \eqref{uxx:order:2}--\eqref{uxx:order:4} are used to calculate first-order and second-order partial derivatives of $a^{n+i}$, $b^{n+i}$, $c^{n+i}$, $d^{n+i}$, and $f^{n+i}$ with $i=1/2,3$  in \eqref{CN:eq:3D} for the FDM of \cref{thm:FDM:3D}.	The numerical results are presented in \cref{Example:6:table} and \cref{Example:6:fig}. From results in \cref{Example:6:table}, we observe that higher-order or simpler FD operators of first-order and second-order partial derivatives of $a^{n+i}$, $b^{n+i}$, $c^{n+i}$, $d^{n+i}$, and $f^{n+i}$   that used in \cref{thm:FDM:3D} for \eqref{CN:eq:3D} produce smaller errors and more stable convergence rates when $h=1/2^6$.
	\end{example}
	\begin{table}[htbp]
		\caption{Performance in \cref{Example:6} of the proposed FDM in \cref{thm:FDM:3D} with the BDF3 method.  Note that $\Qw^{(m,n,q)}:=\tfrac{\partial^{m+n+q}\Qw}{\partial x^m \partial y^n  \partial z^q}$ with $m+n+q\le 2$ represent partial derivatives of $\Qw=a^{n+i},b^{n+i},c^{n+i},d^{n+i},f^{n+i}$  with $i=1/2,3$ that used in \eqref{CN:eq:3D} for the FDM of \cref{thm:FDM:3D}. We use FD operators \eqref{ux:order:2}--\eqref{ux:order:5} and  \eqref{uxx:order:2}--\eqref{uxx:order:4} to compute $\Qw^{(m,n,q)}$.}
		\centering
		{\renewcommand{\arraystretch}{1.0}
		\scalebox{1}{
			\setlength{\tabcolsep}{3mm}{
				\begin{tabular}{c|c|c|c|c|c|c|c|c|c}
					\hline
					\multicolumn{2}{c|}{} &
					\multicolumn{2}{c|}{Setting 1} &
					\multicolumn{2}{c|}{Setting 2} &
					\multicolumn{2}{c|}{Setting 3} &
					\multicolumn{2}{c}{Setting 4}  \\
					\hline
					\multicolumn{2}{c|}{} & 
					\multicolumn{2}{c|}{\eqref{ux:order:2} and \eqref{uxx:order:2}} &
					\multicolumn{2}{c|}{\eqref{ux:order:3} and \eqref{uxx:order:2} }& 
					\multicolumn{2}{c|}{\eqref{ux:order:4} and \eqref{uxx:order:43}} &
					\multicolumn{2}{c}{\eqref{ux:order:5} and \eqref{uxx:order:4}} \\
					\hline
					$h$ & $\tau$ &  $\|u_{h}-u\|_\infty$     &order    &   $\|u_{h}-u\|_\infty$   & order  &   $\|u_{h}-u\|_\infty$   & order  &   $\|u_{h}-u\|_\infty$   & order    \\
					\hline
					$1/2^2$  &  $1/2^2$  &  7.5017E-03  &    &  7.7940E-03  &    &  7.8245E-03  &    &    &  \\
					$1/2^3$  &  $1/2^3$  &  4.5961E-04  &  4.03  &  4.9376E-04  &  3.98  &  4.9331E-04  &  3.99  &  4.9302E-04  &  \\
					$1/2^4$  &  $1/2^4$  &  2.4235E-05  &  4.25  &  2.7746E-05  &  4.15  &  2.7480E-05  &  4.17  &  2.7478E-05  &  4.17\\
					$1/2^5$  &  $1/2^5$  &  1.1298E-06  &  4.42  &  1.4141E-06  &  4.29  &  1.3825E-06  &  4.31  &  1.3824E-06  &  4.31\\
					$1/2^6$  &  $1/2^6$  &  1.7630E-07  &  2.68  &  1.6087E-07  &  3.14  &  1.1719E-07  &  3.56  &  1.1723E-07  &  3.56\\
					\hline
		\end{tabular}}}}
		\label{Example:6:table}
	\end{table}
	\begin{figure}[htbp]
		\centering
		\begin{subfigure}[b]{0.23\textwidth}
			\includegraphics[width=4.6cm,height=4.5cm]{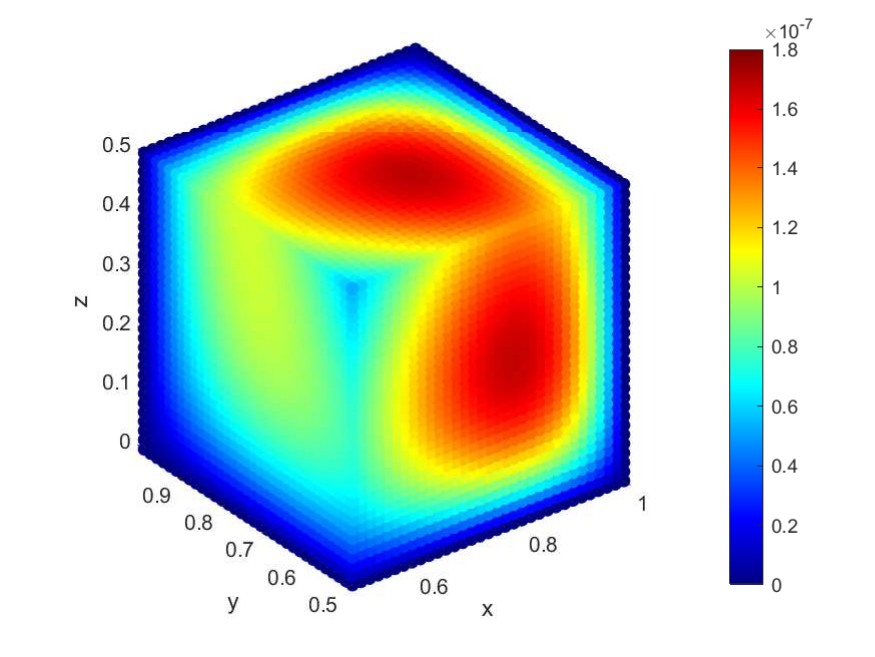}
		\end{subfigure}
		\begin{subfigure}[b]{0.23\textwidth}
			\includegraphics[width=4.6cm,height=4.5cm]{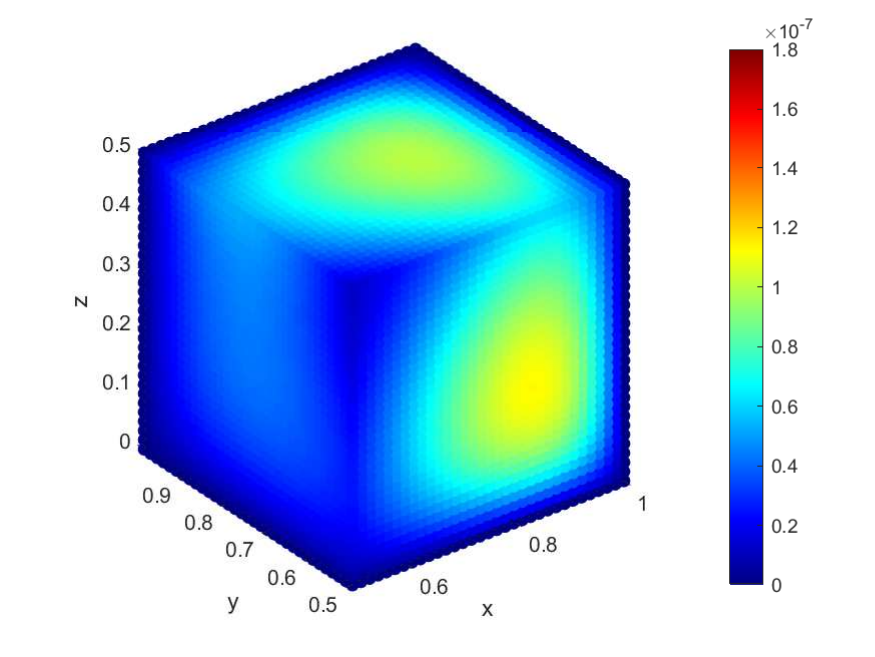}
		\end{subfigure}
		\begin{subfigure}[b]{0.23\textwidth}
			\includegraphics[width=4.6cm,height=4.5cm]{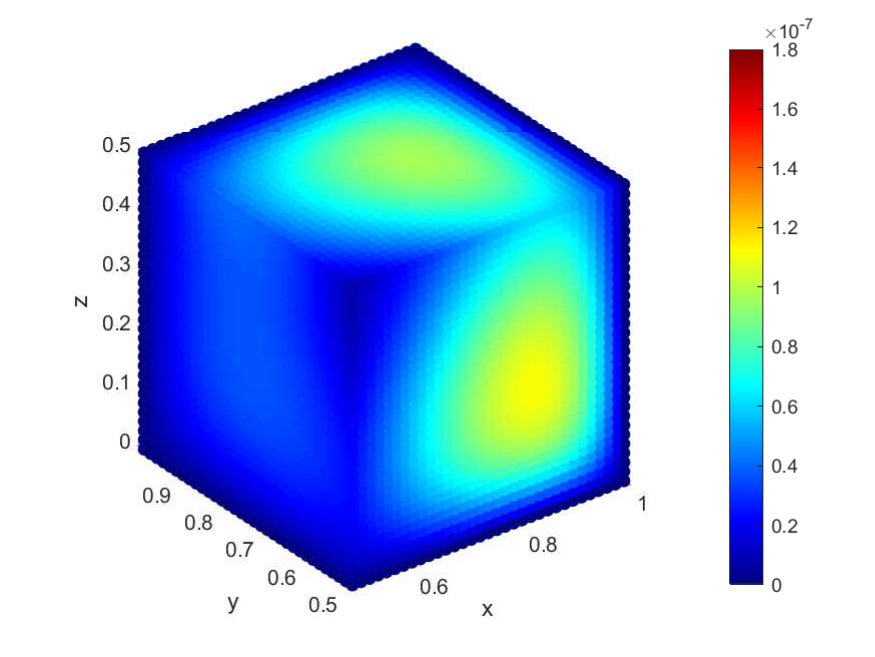}
		\end{subfigure}
		\begin{subfigure}[b]{0.23\textwidth}
			\includegraphics[width=4.6cm,height=4.5cm]{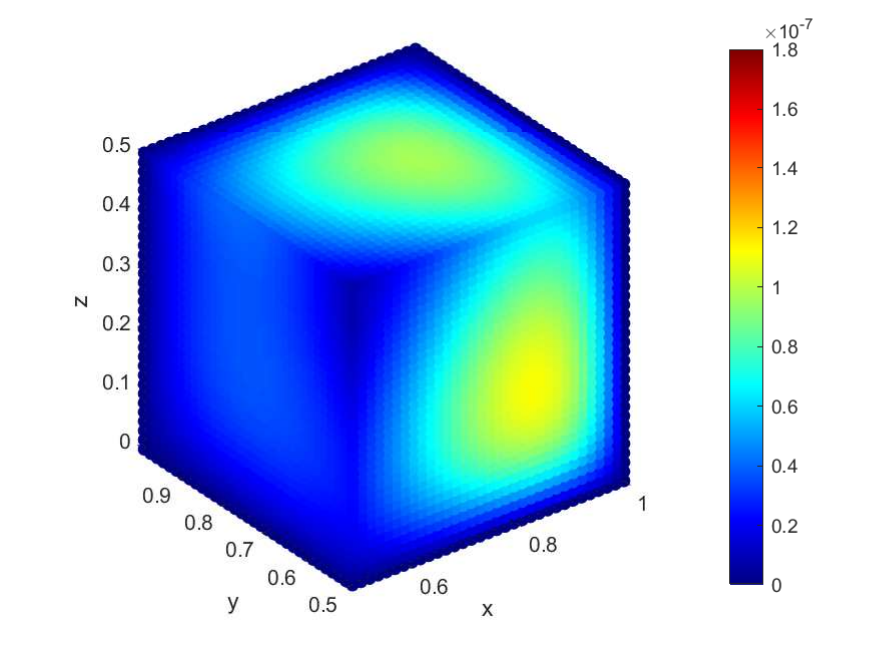}
		\end{subfigure}
		\caption
		{The errors $|u_h-u|$ at $t=1$ of \cref{Example:6} with numerical solutions $u_h$ computed by the FDM in  \cref{thm:FDM:3D} and the BDF3 method:   $|u_h-u|$ with the Setting 1 (first),  $|u_h-u|$ with the Setting 2 (second), $|u_h-u|$ with the Setting 3 (third), and $|u_h-u|$ with the Setting 4 (fourth) on the subdomain (spatial domain $\Omega$ is $(0,1)^3$) $[1/2,1]\times [1/2,1] \times [0,1/2]$ with $\tau=h=1/2^{6}$.  See details of Settings 1-4 in \cref{Example:6:table}.  }
		\label{Example:6:fig}
	\end{figure}	
	
	In the following \cref{Example:7}, we test the 3D nonlinear time-independent equation \eqref{Non:Linear:Elliptic:3D}.
	\begin{example}\label{Example:7}
		\normalfont
		The functions  of the 3D nonlinear time-independent equation \eqref{Non:Linear:Elliptic:3D} are given by
		\begin{align*}
			&u=\cos(x+y-z),\qquad \kappa= \exp(u) ,\qquad \alpha= \sin u, \\
			& \beta= \cos u, \qquad \gamma=u^2,\qquad \lambda=u^3, \qquad \Omega=(-1,1)^3,
		\end{align*}
		the source term $\phi$ and the Dirichlet boundary function $g$  are obtained by plugging above functions into \eqref{Non:Linear:Elliptic:3D}. We use  the proposed FDM in \cref{thm:FDM:3D} to solve this example. Furthermore, FD operators 
		\eqref{ux:order:3}--\eqref{ux:order:5} are used to calculate $(u_\Qk)_x$, $(u_\Qk)_y$, and $(u_\Qk)_z$ for \eqref{ak:bk:dk:fk} and \eqref{cQk:nota}  in \eqref{simplied:eq:non:elli:3D} for the FDM of \cref{thm:FDM:3D}; FD operators  \eqref{ux:order:2}--\eqref{ux:order:4} and \eqref{uxx:order:2}--\eqref{uxx:order:4} are used to calculate first-order and second-order partial derivatives of $a_\Qk$, $b_\Qk$, $c_\Qk$, $d_\Qk$, and $f_\Qk$  in \eqref{simplied:eq:non:elli:3D} for the FDM of \cref{thm:FDM:3D}.	The numerical results are presented in \cref{Example:7:table} and \cref{Example:7:fig}. From results in \cref{Example:7:table}, we also observe that higher-order or simpler FD operators for  $(u_\textup\Qk)_x$,  $(u_\textup\Qk)_y$, $(u_\textup\Qk)_z$, first-order and second-order partial derivatives of $a_\textup\Qk$, $b_\textup\Qk$, $c_\textup\Qk$,  $d_\textup\Qk$, and $f_\textup\Qk$  in \eqref{simplied:eq:non:elli:3D}  produce smaller errors. 
	\end{example}
	\begin{table}[htbp]
		\caption{Performance in \cref{Example:7} of the proposed FDM in \cref{thm:FDM:3D}. Note that $\Qw^{(m,n,q)}:=\tfrac{\partial^{m+n+q}\Qw}{\partial x^m \partial y^n  \partial z^q}$ with $m+n+q\le 2$ represent partial derivatives of $\Qw=a_\Qk,b_\Qk,c_\Qk,d_\Qk,f_\Qk$ that used in  \eqref{simplied:eq:non:elli:3D} for the FDM of \cref{thm:FDM:3D}.}
		\centering
		{\renewcommand{\arraystretch}{1.0}
		\scalebox{1}{
			\setlength{\tabcolsep}{2mm}{
				\begin{tabular}{c|c|c|c|c|c|c|c|c}
					\hline
					\multicolumn{3}{c|}{Setting 1} &
					\multicolumn{3}{c|}{Setting 2} &
					\multicolumn{3}{c}{ Setting 3 } \\
					\hline
					\multicolumn{3}{c|}{\eqref{ux:order:3} for  $(u_\Qk)_x$, $(u_\Qk)_y$, and $(u_\Qk)_z$} &
					\multicolumn{3}{c|}{\eqref{ux:order:4} for  $(u_\Qk)_x$, $(u_\Qk)_y$, and $(u_\Qk)_z$} &
					\multicolumn{3}{c}{ \eqref{ux:order:5} for   $(u_\Qk)_x$, $(u_\Qk)_y$, and $(u_\Qk)_z$} \\
					\hline
					\multicolumn{3}{c|}{\eqref{ux:order:2} and \eqref{uxx:order:2} for $\Qw^{(m,n,q)}$} &
					\multicolumn{3}{c|}{\eqref{ux:order:3} and \eqref{uxx:order:43} for $\Qw^{(m,n,q)}$} &
					\multicolumn{3}{c}{ \eqref{ux:order:4} and \eqref{uxx:order:4} for $\Qw^{(m,n,q)}$ }\\
					\hline
					$h$ &  \hspace{0.7cm} $\|u_{h}-u\|_\infty$ \hspace{0.7cm}    &order   &  $h$ &  \hspace{0.7cm} $\|u_{h}-u\|_\infty$  \hspace{0.7cm}  &order   & $h$ &  \hspace{0.7cm} $\|u_{h}-u\|_\infty$ \hspace{0.7cm}   &order \\
					\hline
					$2/2^2$ &  3.0534E-03 &   &  $2/2^2$ &  4.4444E-03 &   &  $2/2^2$ &   &  \\
					$2/2^3$ &  3.0335E-04 &  3.33 &  $2/2^3$ &  2.5891E-04 &  4.10 &  $2/2^3$ &  2.9907E-04 &  \\
					$2/2^4$ &  2.7869E-05 &  3.44 &  $2/2^4$ &  1.7630E-05 &  3.88 &  $2/2^4$ &  1.9744E-05 &  3.92\\
					$2/2^5$ &  2.8275E-06 &  3.30 &  $2/2^5$ &  1.0873E-06 &  4.02 &  $2/2^5$ &  1.2274E-06 &  4.01\\
					$2/2^6$ &  3.0948E-07 &  3.19 &  $2/2^6$ &  6.8152E-08 &  4.00 &  $2/2^6$ &  7.6902E-08 &  4.00\\			
					\hline
		\end{tabular}}}}
		\label{Example:7:table}
	\end{table}
	\begin{figure}[htbp]
		\centering
		\begin{subfigure}[b]{0.31\textwidth}
			\includegraphics[width=6cm,height=5cm]{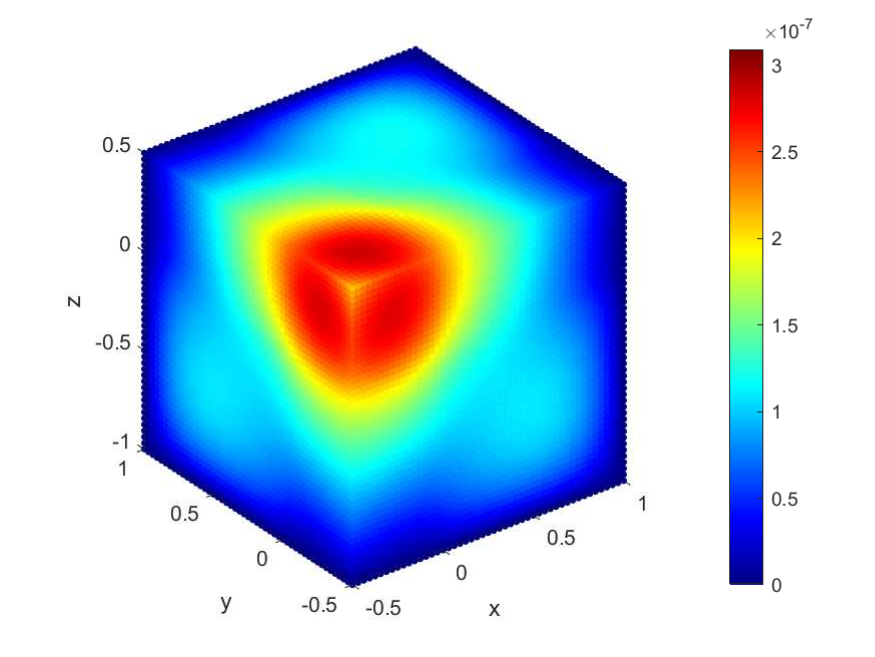}
		\end{subfigure}
		\begin{subfigure}[b]{0.31\textwidth}
			\includegraphics[width=6cm,height=5cm]{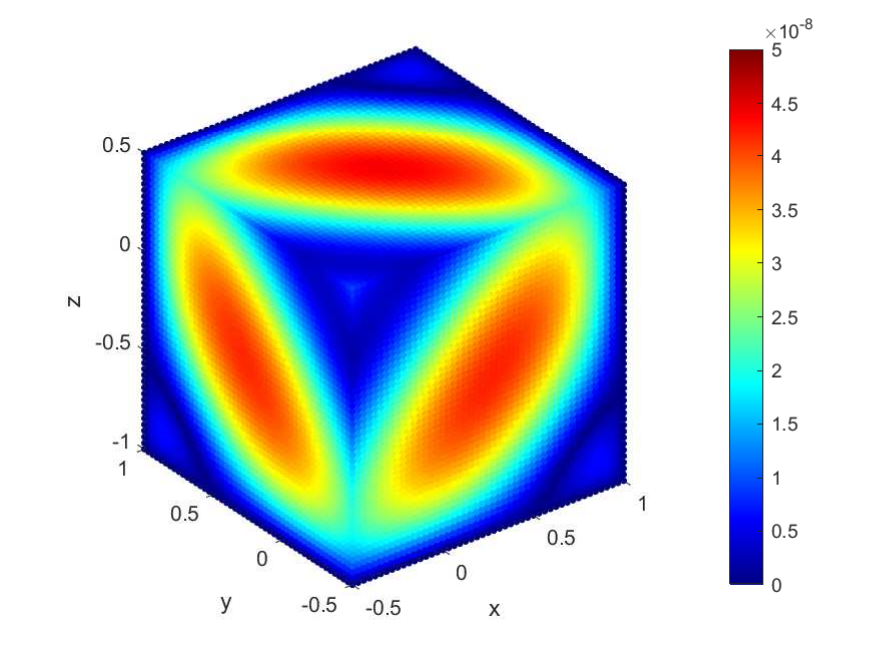}
		\end{subfigure}
		\begin{subfigure}[b]{0.31\textwidth}
			\includegraphics[width=6cm,height=5cm]{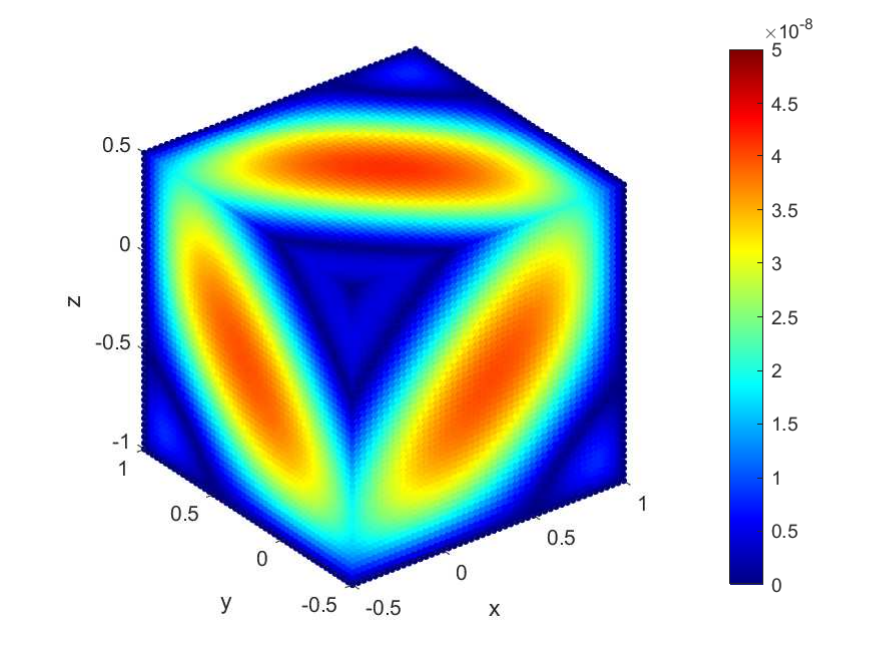}
		\end{subfigure}
		\caption
		{The errors $|u_h-u|$ of \cref{Example:7} with numerical solutions $u_h$ computed by the FDM in \cref{thm:FDM:3D} with Settings 1-3 in \cref{Example:7:table}:   $|u_h-u|$ with the Setting 1 (first),  $|u_h-u|$ with the Setting 2 (second),  $|u_h-u|$ with the Setting 3 (third) on the subdomain (spatial domain $\Omega$ is $(-1,1)^3$) $[-1/2,1]\times [-1/2,1] \times [-1,1/2]$ with $h=2/2^{6}$. See details of Settings 1-3 in \cref{Example:7:table}.  }
		\label{Example:7:fig}
	\end{figure}	
	
	In the following \cref{Example:8}, we test the 3D nonlinear time-dependent equation \eqref{No:Linear:Parabolic:3D}. In \cref{Example:3:table,Example:6:table,Example:7:table}, FD operators \eqref{ux:order:3} and/or \eqref{ux:order:4}, and \eqref{uxx:order:43} yield the smallest errors for FDMs in \cref{thm:FDM:2D} and \cref{thm:FDM:3D}. So we use FD operators \eqref{ux:order:3}--\eqref{ux:order:4} and \eqref{uxx:order:43}  for the FDM in \cref{thm:FDM:3D} in \cref{Example:8}. To demonstrate the accuracy of the FDM in \cref{thm:FDM:3D}, we also test \cref{Example:8} using the FDM in \cref{thm:FDM:parabo:3D} with FD operators \eqref{ux:order:3}--\eqref{ux:order:4}, \eqref{uxx:order:43}, and \eqref{uxxx:order:2}.
	\begin{example}\label{Example:8}
		\normalfont
		The functions  of the 3D nonlinear time-dependent equation \eqref{No:Linear:Parabolic:3D} are given by
		\begin{align*}
			&u=\cos(x+y-z)\exp(-t),\qquad \kappa= \exp(2u) ,\qquad \alpha= \cos u, \\
			& \beta= \sin u, \qquad \gamma=u^2,\qquad \lambda=3+\cos u, \qquad \Omega=(0,1)^3,\qquad I=[0,1]\ (\text{temporal domain}),
		\end{align*}
		the source term $\phi$, the initial function $u^0$, and the Dirichlet boundary function $g$  are obtained by plugging above functions into \eqref{No:Linear:Parabolic:3D}. We use proposed FDMs in \cref{thm:FDM:3D} and \cref{thm:FDM:parabo:3D} with the BDF4 method to solve this example. Note that we utilize FDMs in \cref{thm:FDM:3D} and \cref{thm:FDM:parabo:3D} with the CN method using $\tau=h/2$ to  obtain $u^1,u^2,u^3$ for the BDF4 method. Furthermore, the FD operator
		\eqref{ux:order:4} is used to calculate $(u^{n+i}_\Qk)_x$, $(u^{n+i}_\Qk)_y$, and $(u^{n+i}_\Qk)_z$ with $i=1/2,4$ in \eqref{akni} and \eqref{ckni:3D} in \eqref{simplied:eq:non:para:3D}; FD operators \eqref{ux:order:3}, \eqref{uxx:order:43}, and \eqref{uxxx:order:2} are used to calculate first-order to third-order partial derivatives of $a^{n+i}_\Qk$, $b^{n+i}_\Qk$, $c^{n+i}_\Qk$, $d^{n+i}_\Qk$, and $f^{n+i}_\Qk$ with $i=1/2,4$ in \eqref{simplied:eq:non:para:3D}.	The numerical results are presented in \cref{Example:8:table} and \cref{Example:8:fig}. From results in \cref{Example:8:table}, we observe that although \cref{thm:FDM:3D} and \cref{thm:FDM:parabo:3D} are third-order and fourth-order consistent, respectively, errors from \cref{thm:FDM:3D} are smaller than those from \cref{thm:FDM:parabo:3D} if $h\ge 1/2^5$. The reason may be that the stencil of \cref{thm:FDM:3D} is simple which results in the smaller magnitude in the leading term of the truncation error if $h$ is not small enough, and the complicated expression of \cref{thm:FDM:parabo:3D} has the larger pollution effect for the relatively coarse $h$ and requires the more accurate function evaluation. 
	\end{example}
	\begin{table}[htbp]
		\caption{Performance in \cref{Example:8} of proposed FDMs in  \cref{thm:FDM:3D} and \cref{thm:FDM:parabo:3D} with the BDF4 method.  Note that $\Qw^{(m,n,q)}:=\tfrac{\partial^{m+n+q}\Qw}{\partial x^m \partial y^n  \partial z^q}$ with $m+n+q\le 2 \text{ or } 3$ represent partial derivatives of $\Qw=a^{n+i}_\Qk,b^{n+i}_\Qk,c^{n+i}_\Qk,d^{n+i}_\Qk,f^{n+i}_\Qk$ with $i=1/2,4$ that used in \eqref{simplied:eq:non:para:3D} for \cref{thm:FDM:3D} and \cref{thm:FDM:parabo:3D}.}
		\centering
		{\renewcommand{\arraystretch}{1.0}
		\scalebox{1}{
			\setlength{\tabcolsep}{3mm}{
				\begin{tabular}{c|c|c|c|c|c|c|c}
					\hline
					\multicolumn{4}{c|}{FDM in  \cref{thm:FDM:3D} with BDF4} &
					\multicolumn{4}{c}{FDM in  \cref{thm:FDM:parabo:3D}  with BDF4}  \\
					\hline
					\multicolumn{4}{c|}{\eqref{ux:order:4} for  $(u^{n+i}_\Qk)_x$, $(u^{n+i}_\Qk)_y$, and $(u^{n+i}_\Qk)_z$} &
					\multicolumn{4}{c}{\eqref{ux:order:4}  for  $(u^{n+i}_\Qk)_x$, $(u^{n+i}_\Qk)_y$, and $(u^{n+i}_\Qk)_z$ }\\
					\hline
					\multicolumn{4}{c|}{\eqref{ux:order:3} and \eqref{uxx:order:43} for $\Qw^{(m,n,q)}$} &
					\multicolumn{4}{c}{\eqref{ux:order:3}, \eqref{uxx:order:43}, and \eqref{uxxx:order:2} for $\Qw^{(m,n,q)}$}\\
					\hline
					$h$ & $\tau$ &  \hspace{1cm} $\|u_{h}-u\|_\infty$   \hspace{1cm}  &order   &  $h$  & $\tau$ &  \hspace{1cm} $\|u_{h}-u\|_\infty$  \hspace{1cm} & order    \\
					\hline
					$1/2^2$  &  $1/2^2$  &  1.2255E-04  &    &  $1/2^2$  &  $1/2^2$  &  2.5734E-04  &  \\
					$1/2^3$  &  $1/2^3$  &  5.8982E-06  &  4.38  &  $1/2^3$  &  $1/2^3$  &  1.3109E-05  &  4.30\\
					$1/2^4$  &  $1/2^4$  &  2.5451E-07  &  4.53  &  $1/2^4$  &  $1/2^4$  &  1.0324E-06  &  3.67\\
					$1/2^5$  &  $1/2^5$  &  3.6193E-08  &  2.81  &  $1/2^5$  &  $1/2^5$  &  6.3529E-08  &  4.02\\
					$1/2^6$  &  $1/2^6$  &  4.7213E-09  &  2.94  &  $1/2^6$  &  $1/2^6$  &  4.0274E-09  &  3.98\\		
					\hline
		\end{tabular}}}}
		\label{Example:8:table}
	\end{table}
	\begin{figure}[htbp]
		\centering
		\begin{subfigure}[b]{0.38\textwidth}
			\includegraphics[width=6cm,height=5cm]{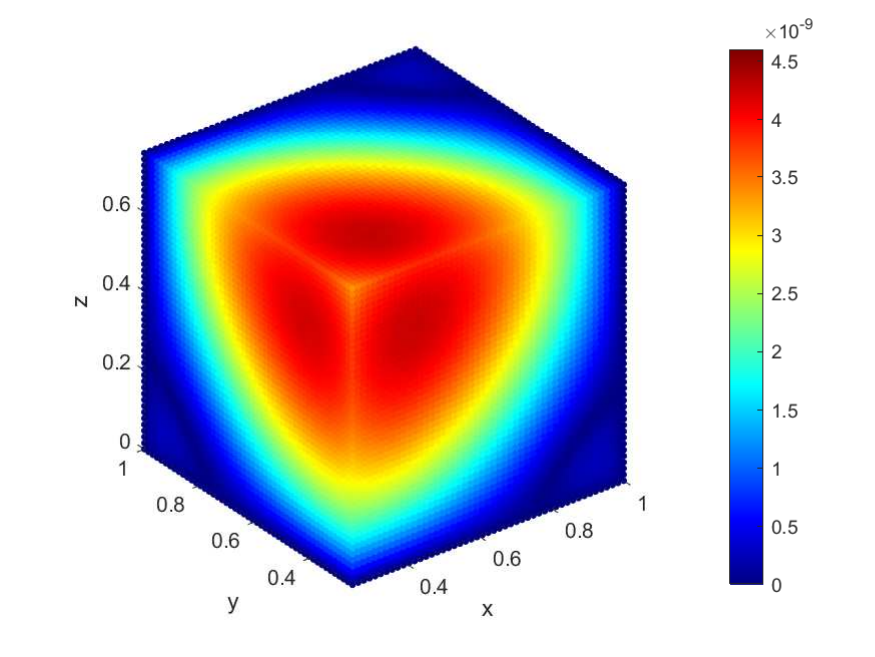}
		\end{subfigure}
		\begin{subfigure}[b]{0.38\textwidth}
			\includegraphics[width=6cm,height=5cm]{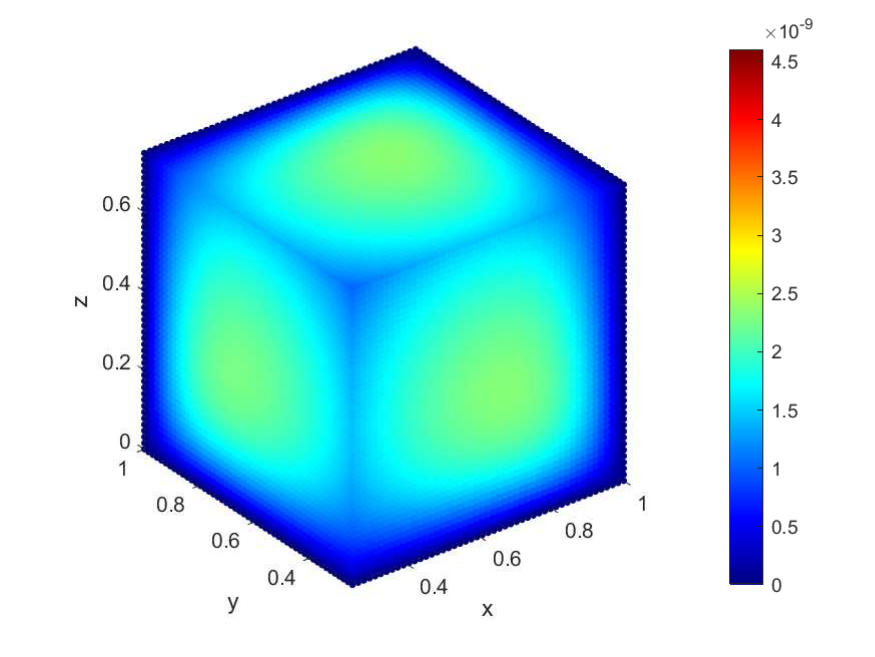}
		\end{subfigure}
		\caption
		{The errors $|u_h-u|$ at $t=1$ of \cref{Example:8} with numerical solutions $u_h$ computed by FDMs in  \cref{thm:FDM:3D} and \cref{thm:FDM:parabo:3D} with the BDF4 method:   $|u_h-u|$ with \cref{thm:FDM:3D} and BDF4 (left),  and $|u_h-u|$ with  \cref{thm:FDM:parabo:3D} and BDF4 (right) on the subdomain (spatial domain $\Omega$ is $(0,1)^3$) $[1/4,1]\times [1/4,1] \times [0,3/4]$ with $\tau=h=1/2^{6}$.  See details of settings in \cref{Example:8:table}.  }
		\label{Example:8:fig}
	\end{figure}	
	\begin{remark} 
		According to results in \cref{Example:1:table}--\cref{Example:8:table}, we can conclude that FDMs in   \cref{thm:FDM:2D} (compact 9-point FDM in 2D) and \cref{thm:FDM:3D} (compact 19-point FDM in 3D) with simple stencils yield convincing numerical solutions for  linear and nonlinear,  time-independent and time-dependent,  convection-diffusion-reaction equations in 2D and 3D.	
		To balance the accuracy and efficiency, 4-point to 5-point FD operators \eqref{ux:order:3}--\eqref{ux:order:4} and \eqref{uxx:order:43} are  optimal FD operators to approximate first-order and second-order partial derivatives used in \cref{thm:FDM:2D,thm:FDM:3D}.
		We plan to use the strategy in \citep[Sections 2.4 and 3.2]{FengTrenchea2026} to minimize truncation errors to reduce pollution effects of  high-order FDMs in future. We plan to combine techniques in \citep{Clain2024,HanSim2025,Mirzadeh2011} with our high-order FDMs to solve the convection-diffusion-reaction equation in an irregular domain.
	\end{remark}

	\section{Contribution}\label{sec:contribu}
	In this paper, we consider the 2D linear time-independent \eqref{Linear:Elliptic:2D} and  time-dependent \eqref{Linear:Parabolic:2D} equations;  2D nonlinear steady \eqref{Non:Linear:Elliptic:2D} and unsteady  \eqref{parabo:nonlinear:2D} equations;  3D linear stationary \eqref{Model：Elliptic:3D} and nonstationary  \eqref{Linear:Parabolic:3D} equations,  3D nonlinear time-independent \eqref{Non:Linear:Elliptic:3D} and time-dependent  \eqref{No:Linear:Parabolic:3D} equations;  where each of $\kappa>0$ (diffusion), $\alpha,\beta,\gamma$ (convection), $\lambda$ (reaction) is the smooth variable function for every linear PDE, and  all $\kappa>0,\alpha,\beta,\gamma,\lambda$ depend on $u$ for all nonlinear equations. Our main contributions in this paper are follows:
	\begin{itemize}
		\item We provide fourth-order compact 9-point (\cref{thm:FDM:2D}, the main result in 2D) and 19-point (\cref{thm:FDM:3D}, the main result in 3D) FDMs with simple stencils to solve linear and nonlinear time-independent equations.
		
		\item We employ CN, BDF3, and BDF4 methods with FDMs in \cref{thm:FDM:2D,thm:FDM:3D} to solve linear and nonlinear unsteady  equations in 2D and 3D. The numerical results present at least third-order convergence rates of errors in the $l_{\infty}$ norm.
		
		\item To demonstrate the accuracy of FDMs in \cref{thm:FDM:2D,thm:FDM:3D} for time-dependent equations, we also propose fourth-order compact 9-point (\cref{thm:FDM:parabo:2D} in 2D) and 19-point (\cref{thm:FDM:parabo:3D} in 3D) FDMs for nonstationary equations. Using these  simple stencils, although \cref{thm:FDM:2D,thm:FDM:3D} can only achieve consistency order three for unsteady equations, one order lower than \cref{thm:FDM:parabo:2D,thm:FDM:parabo:3D}, respectively, errors from  \cref{thm:FDM:2D,thm:FDM:3D} are only slightly larger than those from \cref{thm:FDM:parabo:2D,thm:FDM:parabo:3D} (see col2, col6 and col8 in \cref{Example:4:table} for an example in 2D). For complicated PDEs,  \cref{thm:FDM:2D,thm:FDM:3D} can yield smaller errors than  \cref{thm:FDM:parabo:2D,thm:FDM:parabo:3D}, if $h$ is reasonably fine but not sufficiently small (see \cref{Example:8:table} for an example in 3D).
		
		\item Each FDM in \cref{thm:FDM:2D,thm:FDM:parabo:2D,thm:FDM:3D,thm:FDM:parabo:3D} is the monotone scheme, satisfies discrete maximum principle, and forms an M-matrix for the sufficiently small $h$ if $\kappa>0$ and/or $\lambda\ge 0$ (FDMs in \cref{thm:FDM:parabo:2D,thm:FDM:parabo:3D} only need $\kappa> 0$ without requiring $\lambda\ge 0$).
		
		\item Although derivations of high-order FDMs are still complicated, we provide novel observations of stencils of high-order FDMs to add new restrictions to derive high-order FDMs with simple stencils (\cref{thm:FDM:2D} in 2D and \cref{thm:FDM:3D} in 3D). With the aid of these easy expressions, proofs of truncation errors become straightforward for readers to understand, and the proposed high-order FDMs can be easily implemented by engineers to solve real-word problems.
	\end{itemize}
	
	\section{Declarations}
	\noindent \textbf{Conflict of interest:} The authors declare that they have no conflict of interest.\\
	\noindent \textbf{Data availability:} Data will be made available on reasonable request.
	
	\vspace{0.3cm}
	\noindent\textbf{Acknowledgment}
	Bin Han (bhan@ualberta.ca) and  Jiwoon Sim (jiwoon2@ualberta.ca) (Department of Mathematical and Statistical Sciences, University of Alberta, Edmonton, Alberta, T6G 2N8, Canada) provided initial ideas of properties \eqref{property:1:2D}--\eqref{property:3:2D} for $-\nab\cdot (\kappa \nab u)= \phi$.  Michelle Michelle (mmichell@ualberta.ca, Department of Mathematical and Statistical Sciences, University of Alberta, Edmonton, Alberta, T6G 2N8, Canada) helped to derive \eqref{C:1:Left:3D}--\eqref{F:Right:3D}.

\end{document}